\documentclass[11pt]{amsart}

\usepackage{amsmath,amssymb}
\usepackage{graphicx}
\usepackage{xcolor}

\usepackage{fullpage}
\usepackage{enumerate}
\usepackage{enumitem}
\usepackage{pgf,tikz}
\usetikzlibrary{patterns}
\usetikzlibrary{arrows}
\usetikzlibrary{positioning}
\usetikzlibrary {arrows.meta} 

\def\COMMENT#1{}
\let\COMMENT=\footnote
\def\TASK#1{}

\newdimen\margin   % needed for macros \textdisplay & \ltextdisplay
\def\textno#1&#2\par{%
    \margin=\hsize
    \advance\margin by -4\parindent
           \setbox1=\hbox{\sl#1}%
    \ifdim\wd1 < \margin
       $$\box1\eqno#2$$%
    \else
       \bigbreak
       \hbox to \hsize{\indent$\vcenter{\advance\hsize by -3\parindent
       \sl\noindent#1}\hfil#2$}%
       \bigbreak
    \fi}

\newtheorem{thm}{Theorem}[section]
\newtheorem{examp}[thm]{Example}
\newtheorem{lemma}[thm]{Lemma}

\newtheorem{claim}[thm]{Claim}
\newtheorem{fact}[thm]{Fact}

\newtheorem{col}[thm]{Corollary}

\newtheorem{prop}[thm]{Proposition}
\newtheorem{remark}[thm]{Remark}
\newtheorem*{thm*}{Theorem}
\newtheorem*{define*}{Definition}
\newtheorem*{examp*}{Example}
\newtheorem*{lem*}{Lemma}
\newtheorem*{claim*}{Claim}
\newtheorem*{fact*}{Fact}
\newtheorem*{col*}{Corollary}
\newtheorem*{conj*}{Conjecture}

\def\eps{\varepsilon}

\theoremstyle{definition}
\newtheorem{define}[thm]{Definition}

\DeclareMathOperator{\rank}{rank}
\DeclareMathOperator{\im}{im}

\DeclareMathOperator{\Sol}{Sol}
\DeclareMathOperator{\-Sol}{-Sol}
\DeclareMathOperator{\Var}{Var}
\DeclareMathOperator{\Cov}{Cov}
\newcommand{\roww}[1]{\underline{\mathbf{#1}}}

\newcommand{\qedclaim}{\phantom{.}

\hfill\scalebox{0.6}{$\blacksquare$}

\phantom{.}}

\newenvironment{proofclaim}{\removelastskip\penalty55\medskip\noindent{\it Proof of the claim. }}

\begin{document}

\title{Typical Ramsey properties of the primes, abelian groups and other discrete structures}

\author{Andrea Freschi, Robert Hancock and Andrew Treglown}

\thanks{AF: HUN-REN, Alfr\' ed R\'enyi Institute of Mathematics, Budapest, Hungary, {\tt freschi.andrea@renyi.hu}. \\
\indent RH: Mathematical Institute, University of Oxford, United Kingdom, {\tt robert.hancock@maths.ox.ac.uk}. Research supported by ERC Advanced Grant 883810 and by a Humboldt Research Fellowship at Heidelberg University.\\
\indent AT: School of Mathematics, University of Birmingham, United Kingdom, {\tt a.c.treglown@bham.ac.uk}. Research supported by EPSRC grant EP/V048287/1.\\
}

%\thanks{.}
\date{\today}
\begin{abstract}
Given a matrix $A$ with integer entries, a subset $S$ of an abelian group and $r\in\mathbb N$, we say that $S$ is $(A,r)$-Rado if any $r$-colouring of $S$ yields a monochromatic solution to the system of equations $Ax=0$. 
A classical result of Rado characterises all those matrices $A$ such that $\mathbb N$ is $(A,r)$-Rado for all $r \in \mathbb N$.
R\"odl and Ruci\'nski [Proceedings of the London Mathematical Society, 1997] and Friedgut, R\"odl and Schacht~[Random Structures \& Algorithms, 2010] proved a random version of Rado's theorem where one considers a random subset 
of $[n]:=\{1,\dots,n\}$ instead of $\mathbb N$.

In this paper, we investigate the analogous random Ramsey problem in the more general setting of abelian groups. Given a sequence $(S_n)_{n\in\mathbb N}$ of finite subsets of abelian groups, let $S_{n,p}$ be a random subset of $S_n$ obtained by including each element of $S_n$ independently with probability $p$. We are interested in determining the probability threshold $\hat p:=\hat p(n)$ such that
$$
\lim _{n \rightarrow \infty} \mathbb P [ S_{n,p} \text{ is } (A,r)\text{-Rado}]=
\begin{cases}
0 &\text{ if } p=o(\hat p); \\
1 &\text{ if } p=\omega(\hat p).
\end{cases}
$$  

Our main result, which we coin the \emph{random Rado lemma}, is a general black box to tackle problems of this type. Using this tool in conjunction with a series of supersaturation results, we determine the probability threshold for a number of different cases. A consequence of the Green--Tao theorem [Annals of Mathematics, 2008] is the \emph{van der Waerden theorem for the primes}: every finite colouring of the primes contains arbitrarily long monochromatic arithmetic progressions. Using our machinery, we obtain a random version of this result.
We also prove a novel supersaturation result for  $S_n:=[n]^d$ and use it to prove an integer lattice generalisation of the random version of Rado's theorem. Various threshold results for abelian groups are also given.

Furthermore, we prove a $1$-statement (the $p=\omega(\hat p)$ regime) and a $0$-statement (the $p=o(\hat p)$ regime) for hypergraphs that imply several of the previously known $1$- and $0$-statements in various settings, as well as our random Rado lemma.  
\end{abstract}
\maketitle
%\msc{05C35, 05C70}
%\tableofcontents

\section{Introduction}\label{sec:introduction}

Ramsey theory concerns the study of partitions of mathematical objects, and in particular, what structures one can guarantee in such partitions. For example, Ramsey's original theorem asserts that, given any $r,t \in \mathbb N$, if $n\in\mathbb N$ is sufficiently large then however one partitions the edge set of the complete graph $K_n$ into $r$ colour classes, at least one of these colour classes must contain a copy of $K_t$. 
Analogous Ramsey-type behaviour is also exhibited in arithmetic settings and Ramsey-type results for the integers have been studied since the end of the 19th Century, with early progress due to 
Hilbert~\cite{hilbert} (1892), Schur~\cite{schur} (1916) and van der Waerden~\cite{vdw} (1927). 
In 1933, Rado~\cite{rado} characterised  all those homogeneous systems of linear equations $\mathcal L$ for which every finite colouring of $\mathbb N$ yields a monochromatic solution to $\mathcal L$.
In 1975, Deuber~\cite{deuber} resolved the analogous problem where now one works in the setting of abelian groups rather than $\mathbb N$. See, e.g., \cite{berg, cam,  wolf, serra, pen} for further Ramsey-type results for groups and other arithmetic structures.

Similarly, it is natural to investigate {\it how many} copies of a given structure one can guarantee within a partition of a mathematical object. This is the  topic of a branch of Ramsey theory known as \emph{Ramsey multiplicity}, in which researchers initially focused on graphs.
For example, a seminal result of Goodman~\cite{goodman} from 1959 determines the minimum number of monochromatic triangles in a $2$-edge-colouring of the complete graph $K_n$; see, e.g., the survey~\cite{burr} for further results on the topic. More recently, there has been interest in studying Ramsey multiplicity in arithmetic structures, such as the integers (see, e.g.,~\cite{cost, dav, grr, rob, sch}) and abelian groups. In particular, after earlier work of  Cameron, Cilleruelo and Serra~\cite{cam}, a 2017 paper of Saad and Wolf~\cite{wolf} initiated the systematic study of Ramsey multiplicity for systems of linear equations in abelian groups. A number of subsequent papers have made significant progress on the topic (see, e.g.,~\cite{al, dlz, fox, kam,  ruespiegel, vest}).

In this paper, we consider yet another Ramsey-type question, this time concerning the {\it typical Ramsey properties} of mathematical objects. Broadly speaking, once it is established that any finite partition of a set $S$ yields a copy of a given structure, we are interested in whether we typically expect to witness the same Ramsey-type behaviour in subsets of $S$ of a given size. As we will see in the following sections, this problem is formalised by considering a random subset within $S$, and so we shall refer to it as the {\it random Ramsey problem}. This direction of research is part of the wider study of transferring (combinatorial) theorems into the random setting; see, e.g., the ICM survey of Conlon~\cite{conlonsurvey}. In particular,
the random Ramsey problem for graphs  and the integers 
has been well-studied,  through  the \emph{random Ramsey theorem}~\cite{random1, random2, random3} (see Theorem~\ref{randomramsey}) and the \emph{random Rado theorem~\cite{frs, random4}} (see Theorems~\ref{radores0} and~\ref{r3}), respectively. The statements of both theorems, and of random Ramsey-type results in general, break down into two parts: the so-called $1$-statement, asserting that random subsets of size above a certain threshold exhibit a Ramsey-type behaviour with high probability, and the $0$-statement, asserting that random subsets of size below the threshold do not exhibit a Ramsey-type behaviour with high probability.

The main goal of this paper is to construct a general framework to study the random Ramsey problem for abelian groups.
Our main result, which we call \emph{the random Rado lemma}
(Lemma~\ref{lem:randomRado}), provides a machine that, on input of a homogeneous system of linear equations $\mathcal L$, a sequence of `well-behaved' subsets $S_n$ of abelian groups and a Ramsey-type supersaturation result\footnote{A  Ramsey-type supersaturation result states that in any $r$-colouring of a given mathematical object, a \emph{linear proportion} of a certain class of substructures is monochromatic.} 
for the number of monochromatic solutions of $\mathcal L$ in $S_n$, outputs an appropriate $1$-statement and $0$-statement. In order to rigorously state the random Rado lemma, we need to introduce various terminology as well as recall some fundamental algebraic concepts. As such, we defer the statement of Lemma~\ref{lem:randomRado} to Section~\ref{sec:randomRado}. We give several highlight applications of the random Rado lemma. For example, we prove a random version of the \emph{van der Waerden theorem for the primes} (Theorem~\ref{rvdwp}) as well as a \emph{random Rado theorem for integer lattices} (Theorem~\ref{thm:ramseynd}).\footnote{In fact, we deduce Theorem~\ref{thm:ramseynd} from a simplified version of the random Rado lemma (Lemma~\ref{lem:easyRado}).}

Building towards the aforementioned theory for (subsets of) abelian groups, we also survey the random Ramsey problem for graphs and the integers. As explained in more detail in Sections~\ref{sec:graphsintegers}--\ref{sec:hypergraphs}, there is a well-known heuristic for what the  threshold in the $1$- and $0$-statements of a random Ramsey-type result should be. Indeed, this intuition turns out to be correct in the setting of graphs and the integers as well as for the various new results we prove. Therefore, it would be highly desirable to obtain a unifying theory which implies all previously known random Ramsey-type results. In this direction, we  prove a general $1$-statement and a general $0$-statement for hypergraphs (Theorem~\ref{thm:mainramsey}) which imply several of the previously known $1$- and $0$-statements for discrete structures. Furthermore, the random Rado lemma arises as a direct corollary of this hypergraph result. The proof of Theorem~\ref{thm:mainramsey} makes use of the 
 \emph{hypergraph container method} as well as 
 a recent argument of the 
 second and third authors~\cite{ht}.

 \begin{figure}[!h]
 \centering
 \begin{tikzpicture}[
 squarednodeblue/.style={rectangle,draw=blue!60,fill=white,very thick,minimum size=5mm,text centered,text width=3.5cm,node distance=2.5cm},
 squarednodered/.style={rectangle,draw=red!60,fill=white,very thick,minimum size=5mm,text centered,text width=3.5cm,node distance=2.5cm}]
 
\node[squarednodeblue]  (hypthm)  at (1.7,14)  {{\color{blue}The general $1$- and $0$-statement for hypergraphs (Theorem~\ref{thm:mainramsey})}};
\node[squarednodeblue]  (radolem) at  (1.7,11)      {{\color{blue}The random Rado lemma (Lemma~\ref{lem:randomRado})}};
\node[squarednodered]  (randomram) at (11.7,14)    {{\color{red}The $1$-statement of the random Ramsey theorem (Theorem~\ref{randomramsey})}};
\node[squarednodeblue]  (randomprime) at   (11.7,11)     {{\color{blue}Random van der Waerden theorem for the primes (Theorem~\ref{rvdwp})}};
\node[squarednodeblue]  (simplerado) at  (1.7,8)      {{\color{blue}Simplified random Rado lemma (Lemma~\ref{lem:easyRado})}};
\node[squarednodeblue]  (radond) at  (1.7,3)      {{\color{blue}Random Rado theorem for integer lattices (Theorem~\ref{thm:ramseynd})}};
\node[squarednodered]  (radoorigin) at  (1.7,0)      {{\color{red}Random Rado theorem (Theorems~\ref{radores0} and~\ref{r3})}};
\node[squarednodeblue]  (ramsgroup) at  (8.7,3)      {{\color{blue}A random Ramsey theorem for certain abelian groups (Theorem~\ref{thm:exponent})}};
\node[squarednodeblue]  (ramsfield) at  (11.7,8)    {{\color{blue}Random Ramsey theorems for vector spaces (Theorems~\ref{thm:fields} and~\ref{thm:fields2})}};

\draw[line width=1pt, double distance=3pt,arrows = {-Latex[length=0pt 3 0]}]   (hypthm.south) -- (radolem.north);
\draw[line width=1pt, double distance=3pt,arrows = {-Latex[length=0pt 3 0]}]  (hypthm.east) -- (randomram.west);
\draw[line width=1pt, double distance=3pt,arrows = {-Latex[length=0pt 3 0]}]  (radolem.east) -- (randomprime.west);
\draw[line width=1pt, double distance=3pt,arrows = {-Latex[length=0pt 3 0]}]   (radolem.south) -- (simplerado.north);
\draw[line width=1pt, double distance=3pt,arrows = {-Latex[length=0pt 3 0]}]   (simplerado.south) -- (radond.north);
\draw[line width=1pt, double distance=3pt,arrows = {-Latex[length=0pt 3 0]}]   (radond.south) -- (radoorigin.north);
\draw[line width=1pt, double distance=3pt,arrows = {-Latex[length=0pt 3 0]}]  (simplerado.east) -- (ramsfield.west);
\draw[line width=1pt, double distance=3pt,arrows = {-Latex[length=0pt 3 0]}]  (simplerado.south east) -- (ramsgroup.north west);

\node at (6.7,14.4) {{\color{red}$+$ supersaturation for graphs}}; 
%NOT Theorem~\ref{thm:fgrsupersat}
\node at (6.7,11.4) {{\color{red}$+$ supersaturation Theorem~\ref{thm:supersatprimes}}}; 
\node at (6.7,8.8) {{\color{red}$+$ supersaturation}};
\node at (6.7,8.4) {{\color{red}Theorem~\ref{thm:fieldsupersat}}};
\node at (0,5.9) {{\color{blue}$+$ supersaturation}};
\node at (0,5.5) {{\color{blue}Theorem~\ref{thm:supersat-d}}};
\node at (7.7,5.9) {{\color{red}$+$ supersaturation}};
\node at (7.7,5.5) {{\color{red}Theorem~\ref{thm:exponentsupersat}}};

\node at (8.7,0.4) {Key: \, {\color{red} red $=$ known results}};
\node at (9.0,0.0) {{\color{blue} blue $=$ new results}};

\end{tikzpicture}
\caption{A visualisation of the main new (blue) and previously known (red) results covered in this paper.}
\label{fig:visualisation}
\end{figure}
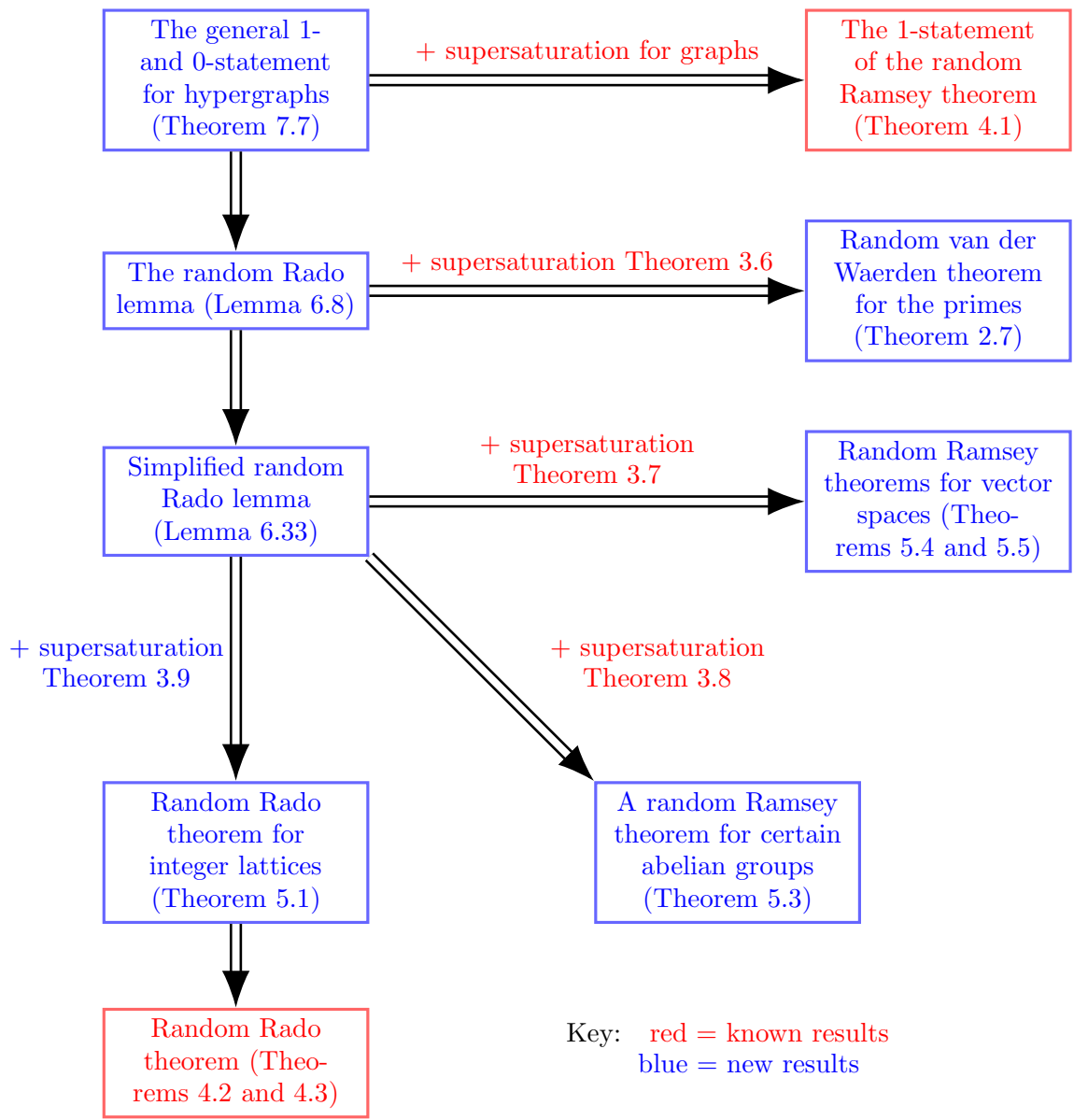

\subsection{Organisation of the paper}
One of the most significant applications of our theory is to arithmetic progressions in the primes; we present our main contribution to this topic  (Theorem~\ref{rvdwp}) in Section~\ref{sec:primes}.  
Sections~\ref{sec:classical} and~\ref{sec:graphsintegers} 
serve as a gentle introduction to some of the key results in the area. 
Namely, in Section~\ref{sec:classical} we briefly cover some of the classical themes in Ramsey theory, specifically partition regularity and supersaturation; in Section~\ref{sec:graphsintegers} we discuss the random Ramsey problem for graphs and the integers.
In Section~\ref{sec:easyapplications}, we highlight the power of the random Rado lemma by presenting various novel results which we obtain from this black box. 
In Sections~\ref{sec:randomRado} and~\ref{sec:hypergraphs}, we formally state the random Rado lemma and our general $1$- and $0$-statements for hypergraphs, respectively. 
Note that
Figure~\ref{fig:visualisation} illustrates how these black box results are connected, and some of their implications.

The rest of the paper is devoted to the technical proofs of our results. In Section~\ref{sec:proofmainramsey}, we prove the $1$- and $0$-statements for hypergraphs and in Section~\ref{sec:proofblackbox} we explain how to deduce the random Rado lemma from them. In Section~\ref{sec:proofapplications}, we give the proofs of all the applications of the random Rado lemma that appear in Sections~\ref{sec:primes}~and ~\ref{sec:easyapplications}. 
For these proofs, we apply a couple of new  results that may be of independent interest: (i) a 
Ramsey-type supersaturation result for $[n]^d$ (Theorem~\ref{thm:supersat-d})
and (ii) an upper bound on the number of $k$-term arithmetic progressions in the primes that contain a given prime number (Lemma~\ref{lem:qAPbound}). 
The former generalises a well-known supersaturation result of Frankl, Graham and R\"odl~\cite{fgr}.
We prove these two results in Section~\ref{sec:aux}.

In Section~\ref{sec:conclusion}, we explain how, with the random Rado lemma at hand, one can attack further related problems. Indeed, we state further consequences of the method as well as 
 discuss potential future research directions.

\section{Arithmetic progressions in the primes}\label{sec:primes}
In this section, we highlight an application of the random Rado lemma to arithmetic progressions in the primes. We defer stating our other applications of the random Rado lemma until
Section~\ref{sec:easyapplications} as they require the introduction of a little more notation and a few definitions.

\subsection{The Green--Tao theorem}
One of the most celebrated results of the 21st Century is the following theorem of Green and Tao~\cite{greentao}. Here we write $\mathbf P$ for the set of prime numbers and, given $n \in \mathbb N$, we write $\mathbf P_n$ for the set of all prime numbers less than or equal to $n$. Recall the prime number theorem asserts that $|\mathbf P_{ n}| \sim n/\log  n$, where throughout this paper we write $\log$ for the natural logarithm. We abbreviate $k$-term arithmetic progressions by $k$-APs.
\begin{thm}[Green and Tao~\cite{greentao}]\label{gt}
Let $k \geq 2$.
The set of primes $\mathbf P$ contains arithmetic progressions of arbitrary length. Moreover, given any subset $S \subseteq \mathbf P_n$ of size $|S|=\Omega (|\mathbf P_n|) $, the number of $k$-APs in $S$ is 
$\Theta (n^2/\log ^k n)$.\footnote{Note that the moreover part of Theorem~\ref{gt} is implicit in~\cite{greentao}. An explicit proof of (a strengthening of) this part of Theorem~\ref{gt} can be found in, e.g.,~\cite[Theorem~1.2]{cmt}.}
\end{thm}
The Green--Tao theorem therefore tells us that any linear size subset of $\mathbf P_n$ is rich in arithmetic progressions, for any large $n \in \mathbb N$.
It is natural to ask about sparser subsets of $\mathbf P_n$: given some $p=p(n)$, does a typical subset $S$ of $\mathbf P_n$ of size approximately $p|\mathbf P_n|$ contain $k$-APs?
Let $\mathbf P_{n,p}$ denote the subset of $\mathbf P_{n}$ obtained by retaining each element of $\mathbf P_{n}$ with probability $p$, independently of all other elements. 
\begin{thm}[Random Green--Tao theorem]\label{gt1}
Let $k \geq 3$. 
$$\lim _{n \rightarrow \infty} \mathbb P [\, \mathbf P_{n,p} \text{ contains a } k\text{-AP}]=\begin{cases}
0 &\text{ if } p=o(n^{-2/k}\log n); \\
 1 &\text{ if } p=\omega ( n^{-2/k}\log n). \end{cases} $$
\end{thm} 
Theorem~\ref{gt1} tells us that typical subsets of $\mathbf P_{n}$ of size significantly above $|\mathbf P_{n}| n^{-2/k}\log n \sim n^{(k-2)/k} $ 
contain  $k$-APs; conversely, typical subsets of $\mathbf P_{n}$ of size significantly smaller than $n^{(k-2)/k} $  do not. Note that, since Theorem~\ref{gt1} is not a Ramsey-type result, we do not use the random Rado lemma to prove it. Instead, we make use of the following result.

\begin{lemma}\label{lem:qAPbound}
Let $k\in\mathbb N$ with $k\ge2$ and let $\ell\in[k]$. For every prime $q\in\mathbf P_n$, the number of $k$-APs in $\mathbf P_n$ such that the $\ell$th term is equal to $q$ is $O(n/\log^{k-1}n)$.
\end{lemma}
We prove Lemma~\ref{lem:qAPbound} in Section~\ref{aux1} via a sieve theory argument.
With Lemma~\ref{lem:qAPbound} and Theorem~\ref{gt} at hand, 
the proof of Theorem~\ref{gt1} is just a simple application of the second moment method. For completeness, a full proof is given in Appendix~\ref{appendix:2ndmoment}.

Given $k \geq 3$ and $\delta >0$, we say a finite set $X  \subseteq \mathbb N$ is \emph{$(\delta,k)$-Szemer\'edi} if every subset of $X$ of size at least $\delta |X|$ contains a $k$-AP. The following result tells us that not too sparse subsets of  $\mathbf P_{n}$ are  typically 
$(\delta,k)$-Szemer\'edi.

\begin{thm}[Balogh, Liu and Sharifzadeh~\cite{bls}]\label{thm:bls}
For any $\delta ,\gamma >0$ and $k \geq 3$, if $p\geq n^{-\frac{1}{k-1}+\gamma}$ then
$$
\lim _{n \rightarrow \infty} \mathbb P [\, \mathbf P_{n,p} \text{ is $(\delta,k)$-Szemer\'edi} ] =1.
$$
\end{thm}
Note that the bound on $p$ in Theorem~\ref{thm:bls} is significantly higher than the corresponding bound in the $1$-statement of Theorem~\ref{gt1}. Further,
Theorem~\ref{thm:bls} is tight up to the $\gamma$ term. 
Indeed, it is not hard to see that if $p =o(n^{-\frac{1}{k-1}}\log n)$ then
$$
\lim _{n \rightarrow \infty} \mathbb P [\, \mathbf P_{n,p} \text{ is $(\delta,k)$-Szemer\'edi} ] =0.
$$
This follows as for such $p$ (suitably bounded away from $0$), the expected number of $k$-APs is much less than the expected number of elements in $\mathbf P_{n,p}$. Hence, with high probability, greedily deleting a number from each $k$-AP produces a linear size subset of $\mathbf P_{n,p}$ containing no $k$-AP. In Section~\ref{sec:res} we give a significant strengthening of Theorem~\ref{thm:bls}; see Theorem~\ref{bls-strengthening}.

\subsection{Monochromatic arithmetic progressions in the primes}
The following seminal result is one of the first proven in Ramsey theory. 

\begin{thm}[van der Waerden~\cite{vdw}]\label{thm:vdw}
Any finite colouring of $\mathbb N$ yields a monochromatic arithmetic progression of arbitrary length.
\end{thm}
One beautiful consequence of the Green--Tao theorem is the following generalisation of van der Waerden's theorem.\footnote{In fact, the Green--Tao theorem implies a supersaturation version of Theorem~\ref{vdwp}, see Theorem~\ref{thm:supersatprimes}.}
\begin{thm}[van der Waerden theorem for the primes~\cite{greentao}]\label{vdwp}
Any finite colouring of $\mathbf P$ yields a monochromatic arithmetic progression of arbitrary length.
\end{thm}
Given $k\geq 3$ and $ r\ge2$, we say a subset $X \subseteq \mathbb N$ is \emph{$(r,k)$-van der Waerden} if whenever $X$ is $r$-coloured there is a monochromatic $k$-AP.
One of the main applications of our machinery is the following random analogue
of Theorem~\ref{vdwp}.
\begin{thm}[Random van der Waerden theorem for the primes]\label{rvdwp}
Let $k \geq 3$ and $r \geq 2$.
There are constants $c,C>0$ such that
$$\lim _{n \rightarrow \infty} \mathbb P [ \, \mathbf P_{n,p} \text{ is } (r,k)\text{-van der Waerden}]=\begin{cases}
0 &\text{ if } p \leq cn^{-\frac{1}{k-1}}\log n; \\
 1 &\text{ if } p \geq Cn^{-\frac{1}{k-1}}\log n.\end{cases}$$
\end{thm}

Theorem~\ref{rvdwp} tells us that typical subsets of $\mathbf P_{n}$ of size significantly above $|\mathbf P_{n}| n^{-\frac{1}{k-1}} \log n \sim n^{\frac{k-2}{k-1}} $ are $(r,k)$-van der Waerden whilst 
 typical subsets of $\mathbf P_{n}$ of size significantly smaller than $n^{\frac{k-2}{k-1}} $  are not. 
 
 Theorem~\ref{rvdwp} follows from the random Rado lemma combined with Theorem~\ref{gt} and Lemma~\ref{lem:qAPbound}. The proof is postponed to Section~\ref{sec:proofapplications}.

\section{Partition regularity and supersaturation}\label{sec:classical}

In this section, we give a brief survey of Ramsey-type results for homogeneous systems of linear equations. 
Partition regularity is covered in the first subsection while supersaturation, including a new result for integer lattices, is covered in the second.

\subsection{Partition regularity}
An integer matrix $A$ is called {\it partition regular} if every finite colouring of $\mathbb N$ yields a monochromatic solution to $Ax=0$. A classical result of Rado~\cite{rado} characterises  all those matrices that are partition regular. In order to state Rado's theorem, and other results in this section, we require the following definition.

\begin{define}[The columns condition over $R$]
Let $R$ be a commutative ring and $A$ be a matrix with entries in $R$. We say $A$ satisfies the {\it columns condition over $R$} if there exists a partition $C_0, C_1, ..., C_m$ of the columns of $A$ such that the following holds. Let $c_j$ denote the $j$th column of $A$ and set 
$$s_i:=\sum_{c_j\in C_i} c_j$$
for every $0\le i\le m$. Then 
\begin{itemize}
\item $s_0=0$ and 
\item $s_i$ can be written as a linear combination of the elements $\{c_j:c_j\in C_k, \, k<i\}$ over $R$, for every $1\le i\le m$.
\end{itemize}
\end{define}

\begin{thm}[Rado~\cite{rado}]
An integer matrix $A$ is partition regular if and only if $A$ satisfies the columns condition over $\mathbb{Q}$.
\end{thm}

There have been several works investigating the analogous problem for other algebraic structures. 
In general, given a ring $R$ acting on an abelian group $G$, a subset $X\subseteq G$ and $r\in\mathbb N$, we say an $\ell\times k$ matrix $A$ with entries in $R$ is {\it $r$-partition regular in $X$} if every $r$-colouring of $X$ yields a monochromatic solution to $Ax=0$, where $0$ denotes the identity of $G^\ell$. Furthermore, $A$ is {\it partition regular in $X$} if it is $r$-partition regular in $X$ for every $r\in\mathbb N$.    
Bergelson, Deuber, Hindman and Lefmann~\cite{berg} gave some sufficient conditions for partition regularity in the case where $G=R$ and $R$ is a commutative ring that acts on itself via ring multiplication. 
In a recent paper, Byszewski and Krawczyk~\cite{bysz} proved several interesting extensions of Rado's theorem for integral domains, noetherian rings and modules.

In this paper, we are interested in monochromatic solutions within abelian groups. 
We mainly focus on  matrices with integer entries since there is an obvious action from $\mathbb Z$ to abelian groups: given an arbitrary abelian group $(G,+)$, for $g\in G$ and $n\in\mathbb N$ we write $ng$ to denote $g+\dots+g$. If $n$ is a negative integer, we write $ng$ to denote the inverse of $|n|g$ in $G$. For the rest of the paper, $\mathbb Z$ will act on abelian groups as we just described. 

Given any abelian group $G$, Deuber~\cite{deuber} provided a necessary and sufficient condition for an integer matrix $A$ to be partition regular in  $G\setminus\{0\}$. 
For \emph{finite} $G$ though it is not interesting as it simply states that 
 $A$ is partition regular in $G\setminus \{0\}$ if and only if there exists $x \in G \setminus \{0\}$ such that $A (x,\dots,x)^T=0$, i.e., there is a \emph{trivial} solution. Indeed, if the number of colours $r$ is larger than $|G\setminus\{0\}|$ then colouring every element of $G\setminus\{0\}$ differently prevents any monochromatic solution other than trivial ones. 

Therefore, in order to ensure a non-trivial monochromatic solution, one must insist that the size of the finite arithmetic structure is large compared to the number of colours. In this direction, Bergelson, Deuber and Hindman~\cite{berg2} proved the following analogue of Rado's theorem for finite vector spaces.
\begin{thm}[Bergelson, Deuber and Hindman \cite{berg2}]\label{thm:bdh} 
Let $\mathbb F$ be a finite field and $A$ be an $\ell \times k$ matrix with entries from $\mathbb F$. The following statements are equivalent:
\begin{itemize}
\item for all $r \in \mathbb{N}$, there exists $n_0=n_0(r, |\mathbb F|, k)\in\mathbb N$ such that for all $n \geq n_0$ every $r$-colouring of $\mathbb F^n\setminus\{0\}$ yields a monochromatic solution to $Ax = 0$ in $\mathbb F^n\setminus\{0\}$;
\item the matrix $A$ is partition regular in $\mathbb{F}^{\mathbb{N}}\setminus \{0\}$;
\item the matrix $A$ satisfies the columns condition over $\mathbb F$.
\end{itemize}
\end{thm} 

Note that Rado's theorem itself may be restated in a style similar to Theorem~\ref{thm:bdh} by using  the following well-known fact.
\begin{fact}\label{fact}
For an integer matrix $A$, the following statements are equivalent:
\begin{enumerate}[label=(\Alph*)]
    \item $A$ is partition regular;\label{item:case1}
    \item for every $r\in\mathbb N$, there exists $n_0\in\mathbb N$ such that for all $n\ge n_0$ every $r$-colouring of $[n]:=\{1,2,\dots,n\}$ yields a monochromatic solution to $Ax=0$ in $[n]$. \qed \label{item:case2}
\end{enumerate}
\end{fact}
Statement~\ref{item:case2} trivially implies statement~\ref{item:case1} while the converse can be proved using a standard compactness argument.

\subsection{Supersaturation}\label{subsec:supersat}
As mentioned in the introduction, the topic of  Ramsey multiplicity concerns how many  copies of a given structure one can guarantee within a partition of a mathematical object. We are interested in a particular class of results in this area which we refer to as {\it supersaturation} results. In the context of this paper, a supersaturation result essentially states that within any finite colouring of an arithmetic structure, a {\it linear proportion} of the total number of solutions of a system of linear equations are monochromatic. As a clarifying example, consider the following  theorem of 
Frankl, Graham and R\"odl~\cite{fgr}.
\begin{thm}[Frankl, Graham and R\"odl~\cite{fgr}]\label{thm:fgrsupersat}
Let $r \in \mathbb{N}$ and let $A$ be a partition regular $\ell \times k$ integer matrix of rank $\ell$.
There exist $\delta,n_0>0$ such that the following holds: for all $n \ge n_0$, 
every $r$-colouring of  $[n]$ yields at least $\delta n^{k-\ell}$ monochromatic solutions to $Ax=0$ in $[n]$.
\end{thm}
It is easy to see that there are at most $n^{k-\ell}$ solutions to $Ax=0$ in $[n]$ (see, e.g., case (ii) of Lemma~\ref{lem:keycounting}). Thus, Theorem~\ref{thm:fgrsupersat} states that in any $r$-colouring of $[n]$, a linear proportion of the total number of solutions are monochromatic. 

Our motivation for considering supersaturation results is that they are a prerequisite for understanding the typical Ramsey properties of subsets of a given size within an arithmetic structure. Specifically, supersaturation results can be used in combination with the hypergraph container method to give an upper bound on the probability that a Ramsey property is not satisfied, thereby yielding a proof for a $1$-statement. 
See Section~\ref{subsec:1-statement-unordered} for a more detailed sketch of how this calculation works.
Note that supersaturation results were employed to prove $1$-statements of Ramsey-type results even before the introduction of the hypergraph container method. In particular, in~\cite{frs}, Theorem~\ref{thm:fgrsupersat} was used in the proof of the $1$-statement of the random Rado theorem (Theorem~\ref{r3} in the next section).

As we will see in Section~\ref{sec:randomRado}, supersaturation is one of the key conditions to check in the random Rado lemma (Lemma~\ref{lem:randomRado}).
In some cases, supersaturation can be  easily obtained from existing results. For example, 
Theorem~\ref{gt} immediately implies the following supersaturation result for the primes.

\begin{thm}\label{thm:supersatprimes}
Given any $r, k\in \mathbb N$ with $k \geq 3$,
any $r$-colouring of $\mathbf P_n$ yields $\Theta(n^2/\log^k n)$ monochromatic $k$-APs in $\mathbf P_n$. \qed
\end{thm}

One can similarly obtain a supersaturation result for finite abelian groups and translation-invariant matrices. We say an $\ell \times k$ integer matrix  $A$ is \emph{translation-invariant} with respect to a group $G$ if $A(g,\dots,g)^T=0$ for every $g\in G$.\footnote{For example, any integer matrix $A$ whose columns sum to $0$ is translation-invariant with respect to any abelian group. This includes integer matrices for which the solutions to $Ax=0$ in $\mathbb Z$ correspond to $k$-APs, for fixed $k\ge3$.}
Consider such a translation-invariant matrix $A$ with respect to an abelian group $G$.
As observed by Saad and Wolf~\cite[page 2]{wolf},
a generalised version of Szemer\'edi's theorem for finite abelian groups~\cite{green} yields that  for any linear sized subset $S$ of $G$, provided $|G|$ is sufficiently large, a constant proportion of all solutions to $Ax=0$ in $G$ are fully contained in $S$. This immediately implies that, given any $r \in \mathbb N$, any $r$-colouring of  $G$ is such that 
a linear proportion of all solutions to $Ax=0$ in $G$ are monochromatic, provided $|G|$ is sufficiently large.

We finish this  section by highlighting three further supersaturation results that we require for our applications of the random Rado lemma appearing in Section~\ref{sec:easyapplications}.
First, Serra and Vena~\cite[Theorem~2.1]{serra} proved an analogue of Theorem~\ref{thm:fgrsupersat} for finite vector spaces.
This can be seen as a quantitative analogue of Theorem~\ref{thm:bdh}.

\begin{thm}[Serra and Vena~\cite{serra}]\label{thm:fieldsupersat}
Let $r \in \mathbb{N}$, let $\mathbb F$ be a finite field and let $A$ be an $\ell \times k$ matrix with entries in $\mathbb F$ and rank $\ell$ which satisfies the columns condition over $\mathbb F$.
There exist $\delta,n_0>0$ such that, for all $n \geq n_0$, every $r$-colouring of  $\mathbb F^n\setminus\{0\}$ yields at least $\delta|\mathbb F|^{n(k-\ell)}$ monochromatic solutions to $Ax=0$ in $\mathbb F^n\setminus\{0\}$.
\end{thm}

One can show that there are at most $|\mathbb F|^{n(k-\ell)}$ solutions to $Ax=0$ in $\mathbb F^n$ and so the lower bound in Theorem~\ref{thm:fieldsupersat} is linear in the total number of solutions.

The works of Serra and Vena~\cite[Theorem 1.3]{serra} and Vena~\cite{vena} extend Theorem~\ref{thm:fieldsupersat} to finite abelian groups.
To state (a special case of) their result, we need a couple of definitions.
We say an abelian group $(G,+)$ has {\it exponent} $s$ if $s$ is the smallest positive integer such that $sg=g+\dots +g=0$ for every $g\in G$.
We say an integer matrix $A$ satisfies the {\it $s$-columns condition} if $A$ satisfies the columns condition over the ring $\mathbb Z_s$ (where the entries of $A$ are treated as elements of $\mathbb Z_s$, i.e., residues modulo $s$). Finally, $S(A,G)$ denotes the set of solutions to $Ax=0$ in $G$.

\begin{thm}[Serra and Vena~\cite{serra}, Vena~\cite{vena}]\label{thm:exponentsupersat}
Let $r \in \mathbb{N}$, let $G$ be a finite abelian group with exponent $s$ and 
let $A$ be an $\ell \times k$ integer matrix which satisfies the $s$-columns condition.
There exist $\delta,n_0>0$ such that, for all $n \geq n_0$, every $r$-colouring of $G^n\setminus\{0\}$ yields at least $\delta|S(A,G^n)|$ monochromatic solutions to $Ax=0$ in $G^n\setminus\{0\}$.
\end{thm}

Note that in~\cite{serra}, a slightly weaker quantitative bound was obtained for Theorem~\ref{thm:exponentsupersat}. 
However, the result used a previous weaker removal lemma of Kr\'al', Serra and Vena~\cite{ksv}. 
Vena~\cite{vena} observed that by using his own removal lemma, one can obtain the stronger bound as stated in Theorem~\ref{thm:exponentsupersat}.

We conclude this section with a new supersaturation result for integer lattices.
\begin{thm}\label{thm:supersat-d}
Let $r,d\in\mathbb N$ and let $A$ be a partition regular $\ell \times k$ integer matrix of rank $\ell$.
There exist $\delta,n_0>0$ such that the following holds: for all $n \ge n_0$, every $r$-colouring of $[n]^d$ yields at least $\delta n^{d(k-\ell)}$ monochromatic solutions to $Ax=0$ in $[n]^d$.
\end{thm}
Again, the lower bound in Theorem~\ref{thm:supersat-d} is a linear proportion of the total number of solutions.
Observe that the case $d=1$ of Theorem~\ref{thm:supersat-d} is exactly Theorem~\ref{thm:fgrsupersat}.  
The proof of Theorem~\ref{thm:supersat-d} is given in Section~\ref{sec:aux}.

{\bf Remark:} Since this paper first appeared on arXiv, we learnt that 
 Theorem~\ref{thm:supersat-d} has been proven independently by Robertson~\cite[Theorem~13]{rob2}.

\section{The random Ramsey problem for graphs and the integers}\label{sec:graphsintegers}
In this section we cover two settings where the random Ramsey problem has been extensively studied: edge-coloured graphs and the integers. Further background on these two problems can be found, e.g., in~\cite{conlonsurvey} and \cite[Section 1]{hst}. We also discuss a few known random Ramsey-type results for $\mathbb Z_n$.

\subsection{Ramsey properties of random graphs}
The \emph{random graph} model provides a framework for studying the typical Ramsey properties of graphs of a given  density. 
The \emph{random graph $G_{n,p}$}
has vertex set $[n]$ where each possible edge is present with probability $p$, independently of all other edges.
An event occurs in $G_{n,p}$ \emph{with high probability} (w.h.p.) if its probability tends to $1$ as $n \rightarrow \infty$. For many properties $\mathcal{P}$ of $G_{n,p}$ (including Ramsey properties), the probability that $G_{n,p}$ has the property exhibits a phase transition,
that is, there is a \emph{threshold} for $\mathcal{P}$: a function $\hat p :=\hat p(n)$ such that $G_{n,p} \text{ has }\mathcal{P}$ w.h.p.~when $p =\omega (\hat p)$ (the \emph{$1$-statement}), while $G_{n,p} \text{ does not have }\mathcal{P}$ w.h.p.~when $p =o(\hat p)$ (the \emph{$0$-statement}).
Bollob\'as and Thomason~\cite{bolthom} proved that every \emph{monotone} property $\mathcal{P}$ has a threshold.

R\"odl and Ruci\'nski~\cite{random1, random2, random3} proved a general Ramsey-type result for the random graph. To state their result we need a few definitions.
Recall that we write $V(H)$ and $E(H)$ for the vertex and edge sets of a graph $H$, respectively, and define $v(H):=|V(H)|$ and $e(H):=|E(H)|$.
 Given an $r\in\mathbb N$ and a graph $G$, an \emph{$r$-edge-colouring} of  $G$ is a function $\sigma:E(G) \rightarrow [r]$. Given a graph $H$,
we say that $G$ is \emph{$(H,r)$-Ramsey} if every $r$-edge-colouring of  $G$ yields a monochromatic copy of $H$ in $G$. 
Given a graph $H$, set $d_2(H):=0$ if $e(H)=0$; $d_2(H):=1/2$ when $H$ is precisely an edge and define $d_2(H) := (e(H)-1)/(v(H)-2)$ otherwise.
Then define $m_2(H) := \max_{H' \subseteq H}d_2(H')$ to be the \emph{$2$-density} of $H$.
For example, if we consider the triangle $K_3$, then $m_2(K_3)=(e(K_3)-1)/(v(K_3)-2)=2$. On the other hand, if we consider the disjoint union of $K_3$ and an edge, i.e., $H:=K_3 \cup K_2$, then we also have that $m_2(H)=2$, even though $d_2(H)=1$, as the `densest' part of $H$ is the copy of $K_3$.

\begin{thm}[The random Ramsey theorem~\cite{random1, random2, random3}]\label{randomramsey} Let $r \geq 2$ 
 be a positive integer and let $H$ be a graph that is not a forest consisting of stars and paths of length $3$.
There are positive constants $c,C$ such that
$$\lim _{n \rightarrow \infty} \mathbb P [G_{n,p} \text{ is } (H,r)\text{-Ramsey}]=\begin{cases}
0 &\text{ if } p \leq cn^{-1/m_2(H)}; \\
 1 &\text{ if } p \geq Cn^{-1/m_2(H)}.\end{cases}$$
\end{thm} 
Thus $n^{-1/m_2(H)}$ is the threshold for the $(H,r)$-Ramsey property. Roughly speaking, the random Ramsey theorem tells us that a typical $n$-vertex graph of edge density significantly greater than $n^{-1/m_2(H)}$ is $(H,r)$-Ramsey, whilst a typical $n$-vertex graph of edge density significantly less than $n^{-1/m_2(H)}$ is not $(H,r)$-Ramsey. 

An intuition for the threshold in Theorem~\ref{randomramsey} is as follows: for small $c>0$, if $p< cn^{-1/m_2(H)}$  then there is a subgraph $H'$ of $H$ for which the expected number of copies of $H'$ in $G_{n,p}$ is much smaller than the expected number of edges in $G_{n,p}$. So intuitively this suggests the copies of $H'$ (and therefore $H$) are `spread out' in the random graph and so one may be able to colour the edges to avoid a monochromatic copy of $H$. On the other hand, for large $C>0$, if $p> C n^{-1/m_2(H)}$ then the expected number of copies of any subgraph $H''$ of $H$ in $G_{n,p}$ is much larger than the expected number of edges; so in this sense, the copies of $H$ in $G_{n,p}$ are more `densely distributed' making it harder to avoid a monochromatic copy of $H$.

In  more recent work, Nenadov and Steger~\cite{ns}  gave a short proof of Theorem~\ref{randomramsey} using the hypergraph container method.
There has also been significant effort to generalise Theorem~\ref{randomramsey} to the setting of \emph{random hypergraphs} (see, e.g.,~\cite{conlongowers, frs, asymmramsey}) and also to asymmetric Ramsey properties; that is, now the monochromatic graph one seeks can be different for each colour class  (see, e.g.,~\cite{bhh, cri, hst, hyde, kreu, ksw, lmms, msss, mns}).

\subsection{Ramsey properties of random sets of integers}\label{subsec:graphsIntegers}
Suppose that $A$ is an $\ell \times k$ integer matrix, and let $S$ be a set of elements of an abelian group (e.g., a set of integers). 
If a vector $x=(x_1,\dots,x_k) \in S^k$ satisfies $Ax=0$  and the $x_i$s are pairwise distinct, we call $x$ a \emph{$k$-distinct solution} to $Ax=0$ in $S$. 
We say that $S$ is \emph{$(A,r)$-Rado} if given any $r$-colouring of $S$, 
there is a monochromatic $k$-distinct solution $x=(x_1,\dots,x_k)$ to $Ax=0$ in $S$. 
Note that in the study of random versions of Rado's theorem authors have (implicitly) considered 
the $(A,r)$-Rado property, rather than seeking a monochromatic solution that is not necessarily
$k$-distinct (as in the original theorem of Rado); see~\cite[Section 4]{ht} for a discussion on why this has been the case.

An $\ell \times k$ integer matrix $A$ is \emph{irredundant} if there exists a $k$-distinct solution to $Ax=0$ in $\mathbb{N}$. Otherwise, $A$ is \emph{redundant}. 
The study of random versions of Rado's theorem has focused on irredundant partition regular matrices. This is natural since
for every redundant $\ell \times k$ matrix $A$ for which $Ax=0$ has solutions in $\mathbb{N}$,
there exists an irredundant $\ell '\times k'$ matrix $A'$
for some $\ell '\le\ell$ and $k'\le k$ with the
same family of solutions (viewed as sets).  See~\cite[Section 1]{random4} for a full explanation.

Index the columns of $A$ by $[k]$. For a partition $W \dot{\cup} \overline{W} = [k]$ of the columns of $A$, we denote by $A_{\overline{W}}$ the matrix obtained from $A$ by restricting to the columns indexed by $\overline{W}$. Let $\rank(A_{\overline{W}})$ be the rank of $A_{\overline{W}}$, where $\rank(A_{\overline{W}}):=0$ for $\overline{W}=\emptyset$. We set 
\begin{align}\label{m(A)def}
m(A):=\max_{\substack{W \dot{\cup} \overline{W} = [k] \\ |W|\geq 2}} \frac{|W|-1}{|W|-1+\rank(A_{\overline{W}})-\rank(A)}.
\end{align}
 The definition of $m(A)$ was introduced in~\cite{random4}. As noted there, the denominator of $m(A)$ is strictly positive provided that $A$ is irredundant and partition regular, and so $m(A)$ is well-defined in this case. We say $A$ is  \emph{strictly balanced} if the expression in (\ref{m(A)def}) is maximised precisely when $W=[k]$. If $A$ has precisely one row consisting of non-zero entries (i.e., it corresponds to a single linear equation) then $A$ is strictly balanced and so 
$$m(A)=\frac{k-1}{k-2}.$$
For example, if $Ax=0$ corresponds to the equation $x+y=z$, then $m(A)=2$.

We write $[n]_p$ for the random subset of $[n]=\{1,\dots,n\}$  where each element in $[n]$ is included with probability $p$ 
independently of all other elements. R\"odl and Ruci\'nski~\cite{random4} showed that $m(A)$ is an important parameter for determining 
whether $[n]_p$ is $(A,r)$-Rado or not. Indeed, one can view $m(A)$ as the arithmetic analogue of the $2$-density $m_2(H)$ which we considered for the graph case in the previous subsection.

\begin{thm}[R\"odl and Ruci\'nski~\cite{random4}]\label{radores0} 
For all irredundant partition regular full rank matrices $A$ and all positive integers $r\ge2$, 
there exists a constant $c>0$ such that 
$$
\lim_{n \rightarrow \infty} \mathbb{P}\left[ [n]_p \text{ is } (A,r)\text{-Rado} \right]=0 \quad\text{ if } p \leq  cn^{-1/m(A)}.
$$
\end{thm}

Roughly speaking, Theorem~\ref{radores0} implies that a typical  subset of $[n]$ with 
significantly fewer than $n^{1-1/m(A)}$ elements is not $(A,r)$-Rado
for any irredundant partition regular matrix $A$.
The following theorem  of Friedgut, R\"odl and Schacht~\cite{frs} complements this result, implying that 
a typical subset of $[n]$ with significantly more than $n^{1-1/m(A)}$ elements is $(A,r)$-Rado.

\begin{thm}[Friedgut, R\"odl and Schacht~\cite{frs}]\label{r3} 
For all irredundant partition regular full rank matrices $A$ and all positive integers $r$, there exists a constant $C>0$ such that $$
\lim_{n \rightarrow \infty} \mathbb{P}\left[ [n]_p \text{ is } (A,r)\text{-Rado}\right]=1 \quad\text{ if } p \geq  Cn^{-1/m(A)}.
$$
\end{thm}
So together, Theorems~\ref{radores0} and~\ref{r3} form the \emph{random Rado theorem} and show that the threshold 
for the property of being $(A,r)$-Rado is $p=n^{-1/m(A)}$.
The intuition for this threshold is similar to that for the threshold in Theorem~\ref{randomramsey}: in the case when $A$ is strictly balanced, when $p=n^{-1/m(A)}$, in expectation there are roughly the same number of elements in $[n]_p$ as there are solutions to $Ax=0$ in $[n]_p$. 
So  if $p\leq cn^{-1/m(A)}$ for some small $c>0$, then one
would expect the solutions to $Ax=0$ to be `spread out' (i.e., one would expect most elements of $[n]_p$  lie in very few solutions to $Ax=0$). Meanwhile, if $p\geq  Cn^{-1/m(A)}$ for some large $C>0$, then one would expect some clustering of solutions to $Ax=0$, making it harder to avoid monochromatic solutions.

Note that  Theorem~\ref{r3} was confirmed earlier by Graham, R\"odl and Ruci\'nski~\cite{grr} in the case where $r=2$ and $Ax=0$ corresponds to $x+y=z$, and then by R\"odl and Ruci\'nski~\cite{random4} in the case when $A$ is so-called density regular.
We remark that both Theorem~\ref{radores0} and Theorem~\ref{r3} are straightforward corollaries of the random Rado lemma (Lemma~\ref{lem:randomRado}), which in turns follows from our general $1$- and $0$-statement for hypergraphs (Theorem~\ref{thm:mainramsey}). Furthermore, the $1$-statement of the random Ramsey theorem (Theorem~\ref{randomramsey}) also follows directly from the general $1$-statement for hypergraphs 
(Theorem~\ref{thm:mainramsey}).

% In~\cite{hst, sp}, Theorem~\ref{r3} was generalised to allow one to consider matrices $A$ that are not partition regular, but for which $\mathbb N$ is 
% $(A,r)$-Rado for some choices of $r$. 
% In particular, the condition that $A$ is partition regular can be replaced with one of the following two equivalent definitions, 
%  subject to a suitable Ramsey-type supersaturation result holding.
% We say that an integer matrix $A$ \emph{satisfies property $(*)$} if 
% under Gaussian elimination $A$ does not have any row which consists of precisely two non-zero rational entries.
% The definition of \emph{abundant}, introduced in~\cite{st-online} and used in~\cite{sp} is that, for an $\ell \times k$ matrix $A$, every $\ell \times (k-2)$ submatrix of $A$ has the same rank as $A$.
% Clearly irredundant full rank matrices which satisfy $(*)$ are equivalent to irredundant full rank abundant matrices.

It is natural to consider whether partition regularity in Theorems~\ref{radores0} and~\ref{r3} can be replaced by a weaker assumption. 
Actually, both Theorems~\ref{radores0} and ~\ref{r3} hold if the condition that~$A$ is partition regular is replaced with~$[n]$ being~$(A,r)$-Rado for all~$n\in \mathbb N$ sufficiently large, or equivalently, that~$\mathbb N$ is~$(A,r)$-Rado, for some \emph{fixed} $r \geq 2$.
This is best possible in the sense that if~$\mathbb N$ is not~$(A,r)$-Rado then $\mathbb P([n]_p \text{ is } (A,r)\text{-Rado})=0$ for any~$n\in \mathbb N$. 
We state this generalisation of Theorems~\ref{radores0} and ~\ref{r3} formally.

\begin{thm}\label{thm:generalisedRR}
Let $r\ge2$ be a positive integer.
For all irredundant full rank matrices $A$ such that~$\mathbb N$ is $(A,r)$-Rado, there exist constants $c,C> 0$ such that
$$\lim _{n \rightarrow \infty} \mathbb P [[n]_p \text{ is } (A,r)\text{-Rado}]=\begin{cases}
0 &\text{ if } p\leq cn^{-1/m(A)}; \\
 1 &\text{ if } p\geq Cn^{-1/m(A)}.\end{cases}$$
\end{thm}

Since this is the first time Theorem~\ref{thm:generalisedRR} has been stated in print,
we now explain how to deduce Theorem~\ref{thm:generalisedRR} from results in the literature.
In~\cite[Theorem 4.1]{hst} it was proved that Theorem~\ref{r3} still holds if the condition that~$A$ is partition regular is replaced with one of the following two equivalent definitions, subject to a suitable Ramsey-type supersaturation result holding.
We say that an integer matrix $A$ \emph{satisfies property $(*)$} if 
under Gaussian elimination $A$ does not have any row which consists of precisely two non-zero rational entries.
The definition of \emph{abundant}, introduced in~\cite{st-online} and used in~\cite{sp} is that, for an $\ell \times k$ matrix $A$, every $\ell \times (k-2)$ submatrix of $A$ has the same rank as $A$.
Clearly irredundant full rank matrices which satisfy $(*)$ are equivalent to irredundant full rank abundant matrices.
Furthermore, if $A$ is an irredundant  matrix and $\mathbb N$ is $(A,r)$-Rado for some $r \geq 2$, then $A$ satisfies property $(*)$; see~Section~4.1 in~\cite{hst} for a justification of this.
In particular, the aforementioned result~\cite[Theorem 4.1]{hst} implies the 1-statement of
Theorem~\ref{thm:generalisedRR} since the necessary
Ramsey-type supersaturation result has now been established~\cite[Theorem~1.3]{dmpt}.
Finally, the $0$-statement follows immediately from a result of the second and third authors~\cite[Theorem~3.1]{ht}.

We remark that
Theorem~\ref{thm:generalisedRR} also follows by combining the aforementioned supersaturation result~\cite[Theorem~1.3]{dmpt} with the simplified version of our random Rado lemma (Lemma~\ref{lem:easyRado}).

\subsection{Ramsey properties of random sets of $\mathbb Z_n$}

Recently there have been a few random Ramsey-type results proven for $\mathbb Z_n$. 
The first explicit result in this area is the following
 \emph{sharp threshold} version of van der Waerden's theorem for random subsets
of $\mathbb Z_n$~\cite{sharp}.
Recalling the definition of $(r,k)$-van der Waerden from earlier, note that, for example, $(r,3)$-van der Waerden and $(A,r)$-Rado are equivalent notions if $A=(1,1,-2)$.
We write $\mathbb Z_{n,p}$ to denote the subset of $\mathbb Z_n$ obtained by including each element of $\mathbb Z_n$ independently and with probability $p$.

\begin{thm}[Friedgut,  H\`an,  Person and  Schacht~\cite{sharp}]\label{vdw}
For all $k \geq$ 3, there exist constants $c_1 , c_0 > 0$ and a function $c(n)$ satisfying
$c_0  \leq c(n) \leq c_1$ such that for every $\eps>0$ we have
$$\lim _{n \rightarrow \infty} \mathbb P [\mathbb Z_{n,p} \text{ is } (2,k)\text{-van der Waerden}]=\begin{cases}
0 &\text{ if } p\leq (1-\eps) c(n) n^{-1/(k-1)}; \\
 1 &\text{ if } p\geq (1+\eps) c(n) n^{-1/(k-1)}.\end{cases}$$
\end{thm}
Note that the threshold in Theorem~\ref{vdw} is sharp compared to, e.g., the threshold given by Theorems~\ref{radores0} and~\ref{r3}. See, e.g., \cite[Section 5.1]{conlonsurvey} for further discussion on sharp thresholds.

Recently, Theorem~\ref{vdw} has been generalised to the $(r,k)$-van der Waerden property (i.e., $r$ does not have to be $2$ now) and an analogous sharp threshold version of the random Schur theorem for $\mathbb Z_n$ (where $n$ is prime) has been obtained; see~\cite{fkssnew}.

\section{Further  applications of  the random Rado lemma}\label{sec:easyapplications}

In addition to the random van der Waerden theorem for the primes (Theorem~\ref{rvdwp}), we prove several other novel random Ramsey-type results via the random Rado lemma (Lemma~\ref{lem:randomRado}).
The proofs of all the results presented in this section are postponed to Section~\ref{sec:proofapplications}.  Given a subset $X$ of an abelian group, we write $X_p$ to
 denote the subset of $X$ obtained by including each element of $X$ independently and with probability $p$.

\smallskip

Recall that Theorems~\ref{radores0} and~\ref{r3}  determine the  threshold for the property that $[n]_p$ is $(A,r)$-Rado. The following result is a generalisation of this statement to $[n]^d_p$ for any fixed $d\in\mathbb N$.   

\begin{thm}[Random Rado theorem for integer lattices]\label{thm:ramseynd}
For all irredundant partition regular full rank matrices $A$ and all positive integers $r\ge2$ and $d\ge1$, there exist constants $C,c>0$ such that the following holds. 
$$\lim _{n \rightarrow \infty} \mathbb P [[n]_p^d \text{ is } (A,r)\text{-Rado}]=\begin{cases}
0 &\text{ if } p \leq cn^{-d/m(A)}; \\
 1 &\text{ if } p \geq Cn^{-d/m(A)}.\end{cases}$$ 
\end{thm}

Observe that the  $d=1$ case corresponds precisely to Theorems~\ref{radores0} and~\ref{r3}. To deduce Theorem~\ref{thm:ramseynd} from the random Rado lemma, we require a corresponding supersaturation result for $[n]^d$, namely Theorem~\ref{thm:supersat-d}.

Note that the probability threshold in Theorem~\ref{thm:ramseynd} is independent of $d$
in the following sense.
As we are interested in how `sparse' a subset can be while typically retaining its Ramsey properties, the probability threshold should always be compared to the size of the initial set we are considering. In this case, the size of $[n]^d$ is $n^d$ and so we may rewrite the probability threshold $n^{-d/m(A)}$ as $|[n]^d|^{-1/m(A)}$. Indeed, this shows that the probability threshold depends exclusively on the matrix $A$. However, this phenomenon does not occur in general: for a fixed matrix $A$, the probability threshold may change depending on the underlying arithmetic structure. For instance, consider  the equation $2x+2y=2z$ and the product group $(\mathbb Z_4)^n$. 

\begin{thm}\label{thm:sec4ex2}
Let $A=(2\;\; 2\; -2)$ and $r\ge2$. There exist constants $c,C>0$ such that 
$$\lim _{n \rightarrow \infty} \mathbb P [(\mathbb Z_4)^n_p \text{ is } (A,r)\text{-Rado}]=\begin{cases}
0 &\text{ if } p \leq c\,4^{-3n/4}; \\
 1 &\text{ if } p \geq C\,4^{-3n/4}.\end{cases}$$  
\end{thm}

As the size of $(\mathbb Z_4)^n$ is $4^n$, we may rewrite $4^{-3n/4}$ as $|(\mathbb Z_4)^n|^{-3/4}$. On the other hand,  $m(A)=2$ and so the corresponding threshold when considering solutions to $Ax=0$ in $[n]_p$ is $n^{-1/2}$. Thus, we can have different thresholds depending on the underlying abelian group.

Note that the probability threshold $p=4^{-3n/4}$ is consistent with the heuristic that the expected size of $(\mathbb Z_4)^n_p$ should `match' the expected number of solutions to $Ax=0$ in $(\mathbb Z_4)^n_p$. Indeed, for $p=4^{-3n/4}$, the expected size of $(\mathbb Z_4)^n_p$ is $4^np=4^{n/4}$. The total number of solutions of $2x+2y=2z$ in $(\mathbb Z_4)^n$ is $4^n\cdot 4^n\cdot 2^n=4^{5n/2}$ and so the expected number of solutions in $(\mathbb Z_4)^n_p$ is $4^{5n/2}p^3=4^{n/4}$.

Observe that the total number of solutions to $2x+2y=2z$ in $[n]$ is at most $n\cdot n\cdot 1=n^2$ since for every choice of $x$ and $y$ there is at most one choice of $z$. The change of threshold in Theorem~\ref{thm:sec4ex2} (and more generally in Theorem~\ref{thm:exponent} below) depends on the fact that we have many more choices for $z$ in $(\mathbb Z_4)^n$. 

\smallskip

Theorem~\ref{thm:sec4ex2} follows easily from the following application of the random Rado lemma to abelian groups. 
The definitions of $\rank_G(A)$ and $m_G(A)$ appearing in the statement require further motivation so
we defer them to the next section; see Definitions~\ref{def:rankfinite} and~\ref{def:mGA}.
However, one may think of $m_G(A)$ as the natural generalisation of $m(A)$ which captures the intuition described for Theorem~\ref{thm:sec4ex2}.
Indeed, $m_{\mathbb{Z}_4}(A)=4/3$ for $A=(2\;\; 2\; -2)$.
For an $\ell \times k$ matrix $A$, we say that $(A,G)$ is \emph{abundant} if every $\ell \times (k-2)$ submatrix $A'$ of $A$ satisfies $\rank_G(A')=\rank_G(A)$. This naturally generalises the previous definition of abundant;  see also Definition~\ref{def:abundant}.

\begin{thm}\label{thm:exponent}
Let $G$ be a finite abelian group with exponent $s\in\mathbb N$. Let $A$ be an integer matrix which satisfies the $s$-columns condition, with $\rank_G(A)>0$ and such that $(A,G)$ is abundant. For
all $r\ge2$, there exist constants $C,c>0$ such that the following holds. 
$$\lim _{n \rightarrow \infty} \mathbb P [G^n_p \text{ is } (A,r)\text{-Rado}]=\begin{cases}
0 &\text{ if } p \leq c|G|^{-n/m_G(A)}; \\
 1 &\text{ if } p \geq C|G|^{-n/m_G(A)}.\end{cases}$$  
\end{thm}
Note that the condition  $\rank_G(A)>0$ in Theorem~\ref{thm:exponent} is very mild: if $A$ has dimension $\ell \times k$, $\rank_G(A)>0$ is equivalent to stating that not every ordered $k$-tuple over $G$ is a solution to $Ax=0$.

Given a finite subset $S$ of an abelian group, we say that an $\ell \times k$ matrix $A$ is \emph{irredundant with respect to $S$} if there exists a $k$-distinct solution to $Ax=0$ in $S$. So a matrix $A$ is irredundant precisely if it is irredundant with respect to $\mathbb N$.
Another application of the random Rado lemma is the following random version of the Ramsey-type result of Bergelson, Deuber and Hindman for vector spaces $\mathbb F^n$ (Theorem~\ref{thm:bdh}).
For an $\ell \times k$ matrix $A$ with entries from a field $\mathbb{F}$, we define $m_{\mathbb F}(A)$ analogously to (\ref{m(A)def}), that is
\begin{align}\label{m(A)deffield}
m_{\mathbb{F}}(A):=\max_{\substack{W \dot{\cup} \overline{W} = [k] \\ |W|\geq 2}} \frac{|W|-1}{|W|-1+\rank_{\mathbb{F}}(A_{\overline{W}})-\rank_{\mathbb{F}}(A)},
\end{align}
where we write $\rank_{\mathbb{F}}$ to denote that we calculate rank with respect to the field $\mathbb{F}$.

\begin{thm}\label{thm:fields}
Let $\mathbb F$ be a non-trivial finite field. Consider any  full rank  matrix $A$ with entries from $\mathbb{F}$ so that $A$ satisfies the columns condition over $\mathbb F$ and $A$ is also 
irredundant with respect to $\mathbb F^n$ for all $n \in \mathbb N$ sufficiently large.
Given any positive integer $r\ge2$, there exist constants $C,c>0$ such that the following holds.
$$\lim _{n \rightarrow \infty} \mathbb P [\mathbb F_p^n \text{ is } (A,r)\text{-Rado}]=\begin{cases}
0 &\text{ if } p \leq c|\mathbb F|^{-n/m_{\mathbb F}(A)}; \\
 1 &\text{ if } p \geq C|\mathbb F|^{-n/m_{\mathbb F}(A)}.\end{cases}$$ 
\end{thm}

We also obtain  the analogue of Theorem~\ref{thm:fields} for matrices $A$ which have integer entries. 
Here we must be slightly careful: as we are interested in solutions to $Ax=0$ in $\mathbb{F}^n$, one must compute the rank of the  integer matrix $A$ with respect to the underlying field $\mathbb{F}$ (see Definition~\ref{def:rankPIDinteger})
and hence the threshold depends on $m_\mathbb{F}(A)$ rather than $m(A)$.\footnote{Note that the definition of $m_\mathbb{F}(A)$ for an integer matrix $A$ is also consistent with putting $G=\mathbb{F}$ into the definition of $m_G(A)$ for groups (see Definition~\ref{def:mGA} and Fact~\ref{fact:equiv}).} 
For example, the matrix $A=\begin{pmatrix} 3 & 3 & -3 \end{pmatrix}$ usually has $\rank(A)=1$, however, in $\mathbb{Z}_3$, we have $3\equiv0 \ (\text{mod} \ 3)$ and so we have $\rank_{\mathbb{Z}_3}(A)=0$.

\begin{thm}\label{thm:fields2}
Let $\mathbb F$ be a finite field of order $q^k$, for some prime $q$ and $k \in \mathbb{N}$. Consider any full rank integer matrix $A$  so that $A$ satisfies the $q$-columns condition and $A$ is also irredundant with respect to $\mathbb F^n$ for all $n \in \mathbb N$ sufficiently large.
Given any positive integer $r\ge2$, there exist constants $C,c>0$ such that the following holds.
$$\lim _{n \rightarrow \infty} \mathbb P [\mathbb F_p^n \text{ is } (A,r)\text{-Rado}]=\begin{cases}

0 &\text{ if } p \leq c|\mathbb F|^{-n/m_{\mathbb F}(A)}; \\

1 &\text{ if } p \geq C|\mathbb F|^{-n/m_{\mathbb F}(A)}.\end{cases}$$ 
\end{thm}
Theorem~\ref{thm:fields2} is just a simple application of Theorem~\ref{thm:fields}; its proof is given in Section~\ref{sec:proofapplications}.

\section{The random Rado lemma}\label{sec:randomRado}

In this section, we rigorously state the random Rado lemma (Lemma~\ref{lem:randomRado}). As previously mentioned, the random Rado lemma serves as a general black box that outputs a random Ramsey-type result for a given sequence of finite subsets of abelian groups provided that certain conditions, most notably supersaturation, are satisfied. 

The content of this section is divided into four subsections. The first subsection contains the statement of Lemma~\ref{lem:randomRado} preceded by a list of all the required definitions. In the second subsection, we prove a general bound for the number of solutions to $Ax=0$ in $S$ when $S$ is either a finite abelian group or a (power of a) finite subset of a field. This bound is used in the third subsection to prove that several technical requirements of the random Rado lemma trivially hold, or can be relaxed, in these specific settings. In light of these results, in the fourth subsection we formulate a simplified version of the random Rado lemma (Lemma~\ref{lem:easyRado}) for finite abelian groups and (powers of) finite subsets of fields. 

We stress once more that the random Rado lemma can be straightforwardly deduced from our general $1$-statement and $0$-statement for hypergraphs (Theorem~\ref{thm:mainramsey}). The proof of Lemma~\ref{lem:randomRado} is deferred to Section~\ref{sec:proofblackbox}.

\subsection{Definitions and main statement}
Let $(S_n)_{n\in\mathbb N}$ be a sequence of finite subsets of abelian groups. We write $S_{n,p}$ to denote the subset of $S_n$ obtained by including each element of $S_n$ independently and with probability $p$.
We say an event occurs in  $S_{n,p}$
\emph{with high probability} (w.h.p.) if its probability tends to $1$ as $n \rightarrow \infty$. Let $A$ be an $\ell\times k$ integer matrix. We are interested in the probability threshold $\hat p:=\hat p(n)$ such that $S_{n,p}$ is $(A,r)$-Rado w.h.p.\ if $p=\omega(\hat p)$ and $S_{n,p}$ is not $(A,r)$-Rado w.h.p.\ if $p=o(\hat p)$. 

Of course, as each subset $S_n$ can lie in a different abelian group, it is easy to come up with a sequence $(S_n)_{n\in\mathbb N}$ for which no such probability threshold $\hat p$ exists. On the other hand, in Definition~\ref{def:threshold} we provide a candidate for 
$\hat p$ for `well-behaved' sequences  $(S_n)_{n\in\mathbb N}$.
 Given a $k$-tuple $(x_i)_{i\in[k]}$ and $W\subseteq [k]$, we write $x_W$ for the $|W|$-tuple $(x_i)_{i \in W}$. 

\begin{define}[Projected solutions]\label{def:projectedsol}
Let $S$ be a finite subset of an abelian group and $A$ be an $\ell\times k$ integer matrix. Let $W\subseteq Y\subseteq[k]$ and $w_0\in S^{|W|}$. The set of {\it projected solutions} $\Sol_S^A(w_0,W,Y)$ is the set of vectors $y\in S^{|Y|}$ such that there exists a solution  $x\in S^k$ to $Ax=0$ with $x_Y=y$ and $x_W=w_0$. For brevity, we write $\Sol_S^A(Y):=\Sol_S^A(\emptyset,\emptyset,Y)$, i.e., $\Sol_S^A(Y)$ is the set of vectors $y\in S^{|Y|}$ such that there exists a solution $x\in S^k$ to $Ax=0$ with $x_Y=y$. In particular, $\Sol_S^A([k])$ is the set of all solutions to $Ax=0$ in $S$.
\end{define}

\begin{define}[Probability threshold for abelian groups]\label{def:threshold}
Let $S$ be a finite subset of an abelian group and $A$ be an $\ell\times k$ integer matrix. For every $W\subseteq[k]$, we let
\begin{align*}
p_W(A,S) :=\left(\frac{|\Sol_S^A (W)|}{|S|}\right)^{-\frac{1}{|W|-1}}.
\end{align*}
Additionally, we let
\begin{align*}
\hat p(A,S) :=\max_{\stackrel{W \subseteq [k]}{|W| \geq 2}}p_W(A,S).
\end{align*}
\end{define}

The intuition for the choice of $\hat p(A,S)$ in Definition~\ref{def:threshold} aligns with the heuristic we discussed in Section~\ref{sec:graphsintegers}, where we considered the random Ramsey problem for the integers. Indeed, it is easy to check that for $p:=p_W(A,S)$ the expected number of projected solutions $|\Sol_S^A(W)|p^{|W|}$  equals the expected size $|S|p$ of the ground set. The probability threshold $\hat p(A,S)$ is obtained by taking the largest of these $p_W(A,S)$.

\begin{remark}
As the definition of $(A,r)$-Rado involves $k$-distinct solutions, the reader may wonder why, in Definition~\ref{def:projectedsol}, we do not insist that the solutions to $Ax=0$ must be $k$-distinct. This choice is designed to simplify later calculations. We will be able to restrict our attention to $k$-distinct solutions by imposing additional assumptions over the sequence $(S_n)_{n\in\mathbb N}$ (namely, conditions (A3) and (A6) of Lemma~\ref{lem:randomRado}).
\end{remark}

Next, we introduce four properties that the sequence $(S_n)_{n\in\mathbb N}$ must satisfy in the hypothesis of the random Rado lemma. 
The first one is supersaturation, which we introduced and gave motivation for in Section~\ref{subsec:supersat}. We will use the following formal definition.

\begin{define}[Supersaturation]\label{def:supersat}
Let $r \in \mathbb{N}$, $(S_n)_{n \in \mathbb N}$ be a sequence of finite subsets of abelian groups and  $A$ be an $\ell\times k$ integer matrix. We say that $(S_n)_{n \in \mathbb N}$ is \emph{$(A,r)$-supersaturated} if there exist an $\eps>0$ and $n_0 \in \mathbb N$ such that the following holds: 
if $n \geq n_0$ then whenever $S_n$ is $r$-coloured 
there are at least $\eps |\Sol_{S_n}^A([k])|$ monochromatic  solutions to $Ax=0$ in $S_n$. 
\end{define}

The second property involves an upper bound on the number of projected solutions. Intuitively, this forces the solutions to be uniformly distributed over the ground set. This condition is rather natural since the heuristic argument for the probability threshold relies on the assumption that the solutions to $Ax=0$ are `spread out'. 

\begin{define}[Extendability]\label{def:extendability}
Let $S$ be a finite subset of an abelian group, let $A$ be an $\ell\times k$ integer matrix and let $B \ge 1$ be a real number. We say the pair $(A,S)$ is \emph{$B$-extendable} if
for all non-empty $W \subseteq Y \subseteq [k]$ and any $w_0 \in S^{|W|}$ we have
\begin{align*}
|\Sol_S^A(w_0,W,Y)|\le B\cdot\frac{|\Sol_S^A(Y)|}{|\Sol_S^A(W)|}.   
\end{align*}
\end{define}

The third property is more involved and perhaps the most artificial out of all requirements of the random Rado lemma. This property is tied to the calculations appearing in the proof of the general $0$-statement for hypergraphs (Theorem~\ref{thm:mainramsey}) in Section~\ref{sec:proofmainramsey}.

\begin{define}[Compatibility]\label{def:compatibility}
Let $(S_n)_{n\in\mathbb N}$ be a sequence of finite subsets of abelian groups and $A$ be an $\ell\times k$ integer matrix with $k\ge3$. We say $(S_n)_{n\in\mathbb N}$ is \emph{compatible} with respect to $A$ if there exists a sequence $(X_n)_{n\in\mathbb N}$ of subsets $X_n\subseteq[k]$ with $|X_n|\ge3$ for all $n\in\mathbb N$ such that the following holds. We have $p_{X_n}(A,S_n)= \Omega(\hat p(A,S_n))$ and, for all sequences $(W_n)_{n\in\mathbb N}$ and $(W'_n)_{n\in\mathbb N}$ with $W_n,W'_n\subset  X_n$, $|W_n|=2$ and $|W'_n|\ge2$,
\begin{align}\label{eq:compatible}
\frac{|S_n|^2}{|\Sol_{S_n}^A(W_n)|}\left(\frac{p_{W'_n}(A,S_n)}{p_{X_n}(A,S_n)}\right)^{|W'_n|-1}\to0
\end{align}
as $n\to\infty$.
\end{define}

The fourth property is a mild technical requirement that will be needed in the proof of Lemma~\ref{lem:randomRado} to allow us to restrict our attention to $k$-distinct solutions (see Lemma~\ref{lem:kdistinct}).

\begin{define}[Weak compatibility]\label{def:weakcompatibility}
Let $(S_n)_{n\in\mathbb N}$ be a sequence of finite subsets of abelian groups and $A$ be an $\ell\times k$ integer matrix with $k\ge2$. We say $(S_n)_{n\in\mathbb N}$ is \emph{weakly compatible} with respect to $A$ if for every $W\subseteq[k]$ with $|W|=2$ we have 
$$\frac{|S_n|}{|\Sol_{S_n}^A(W)|}\to0$$
as $n\to\infty$.
\end{define}

We are now ready to formally state the random Rado lemma.

\begin{lemma}[The random Rado lemma]\label{lem:randomRado}
Let $r\geq 2$, $k  \geq 3$ and $\ell\geq 1$ be integers. Let $(S_n)_{n \in \mathbb N}$ be a sequence of finite subsets of abelian groups, $A$ be an $\ell \times k$ integer matrix and set $\hat p_n:=\hat p(A,S_n)$.

Suppose that 
\begin{itemize}
\item[(A1)] $\hat p_n\to 0$ and $|S_n|\hat p_n\to\infty$ as $n\to\infty$;
\item[(A2)] $(S_n)_{n \in \mathbb N}$ is $(A,r)$-supersaturated; 
\item[(A3)] there exists $B\ge1$ such that the pair $(A,S_n)$ is $B$-extendable for all sufficiently large $n\in\mathbb N$;
\item[(A4)] there exists $D>0$ such that for every $W\subset Y\subseteq[k]$ with $|W|=1$ and $w_0\in S_n$, we have $|\Sol_{S_n}^A(w_0,W,Y)|\le D\frac{|\Sol_{S_n}^A(Y)|}{|S_n|}$ for all sufficiently large $n\in\mathbb N$; 
\item[(A5)] $(S_n)_{n \in \mathbb N}$ is compatible with respect to $A$;
\item[(A6)] $(S_n)_{n \in \mathbb N}$ is weakly compatible with respect to $A$.
\end{itemize}
Then there exist constants $c,C>0$ such that
$$\lim _{n \rightarrow \infty} \mathbb P [ S_{n,p} \text{ is } (A,r)\text{-Rado}]=\begin{cases}
0 &\text{ if } p\leq {c}\hat p_n; \\
 1 &\text{ if } p\geq {C}\hat p_n.\end{cases}$$
\end{lemma}
We find the generality of Lemma~\ref{lem:randomRado} very interesting - it confirms that for `well-behaved' sequences $(S_n)_{n \in \mathbb N}$, the probability threshold essentially depends  on the number of (projected) solutions to $Ax=0$ in $S_n$.

Note that (A4) is a similar condition to the $|W|=1$ case of the definition of $B$-extendable; we have written (A3) and (A4) as separate conditions as they correspond to conditions (P3) and (P4) in our general $1$-statement and $0$-statement for hypergraphs (Theorem~\ref{thm:mainramsey}). In fact,
the random Rado lemma is a simple corollary of Theorem~\ref{thm:mainramsey} and each of conditions (A1)--(A5)  are  adaptations of conditions (P1)--(P5) from Theorem~\ref{thm:mainramsey}. Condition (A6) is a mild requirement that we impose in order to restrict our attention to $k$-distinct solutions (see Lemma~\ref{lem:kdistinct} below).

For most natural applications, conditions (A3)--(A6) can be further simplified. Thus, in practice, condition (A2) is the crucial assumption that one needs to verify. The goal of the next three subsections is to state a simplified version of the random Rado lemma (Lemma~\ref{lem:easyRado}) which holds for sequences of finite abelian groups and sequences of (powers of) finite subsets of fields. Nonetheless, there are interesting examples where the generality of the random Rado lemma as stated in Lemma~\ref{lem:randomRado} is required. The most notable such example is the random van der Waerden theorem for the primes (Theorem~\ref{rvdwp}), proved in Section~\ref{sec:proofapplications}. 

It would be interesting to see if one can relax some of the assumptions in
 the random Rado lemma; if so, condition (A5) is the most likely candidate for where the hypothesis can be weakened. See Section~\ref{sec:ring} for a brief discussion related to this.

\smallskip

The next result states that conditions (A3) and (A6) imply a linear proportion of projected solutions are induced by $k$-distinct solutions. The proof is postponed to Section~\ref{sec:proofblackbox} but we state the result here in order to give the reader a better intuition  why we did not impose in Definition~\ref{def:projectedsol} that the solutions are $k$-distinct.

\begin{lemma}\label{lem:kdistinct}
Let $(S_n)_{n\in\mathbb N}$ be a sequence of finite subsets of abelian groups and let $A$ be an $\ell\times k$ integer matrix with $k\ge2$. Suppose that conditions (A3) and (A6) from Lemma~\ref{lem:randomRado} hold; so there is a $B\ge1$ such that $(A,S_n)$ is $B$-extendable for every $n\in\mathbb N$ sufficiently large. 

For each non-empty $Y \subseteq [k]$, let $k\-Sol_{S_n}^A(Y)$ be the set of $y\in S_n^{|Y|}$ such that there exists a $k$-distinct solution $x\in S^k$ to $Ax=0$ with $x_Y=y$. For $n$ sufficiently large, we have 
$$\frac{1}{2B}\le\frac{|k\-Sol_{S_n}^A(Y)|}{|\Sol_{S_n}^A(Y)|}\le1.$$
Furthermore, for $Y=[k]$ we have
$$\lim_{n\to\infty}\frac{|k\-Sol_{S_n}^A([k])|}{|\Sol_{S_n}^A([k])|}=1.$$
\end{lemma}

\begin{remark}\label{remark:irredundant}
Recall that given a finite subset $S$ of an abelian group, we say that an $\ell \times k$ matrix $A$ is {irredundant with respect to $S$} if there exists a $k$-distinct solution to $Ax=0$ in $S$. 
In many of the results mentioned previously in this paper (e.g.,  Theorem~\ref{thm:ramseynd}) we require that the matrix   considered is irredundant. However,
we do not explicitly state that $A$ is irredundant with respect to $S_n$ in Lemma~\ref{lem:randomRado} as conditions (A3) and (A6) combined with Lemma~\ref{lem:kdistinct} already imply that many solutions are $k$-distinct. 
\end{remark}

\subsection{Rank and counting solutions}

The notion of extendability and (weak) compatibility introduced in the previous subsection are tied to the quantity $|\Sol_S^A(W)|$. In many natural cases we are able to approximate, or at least bound, this number using standard methods from group theory and linear algebra. For example, it is well known that the size of $\Sol_{[n]}^A([k])$, i.e., the number of solutions to $Ax=0$ in $[n]$, is at most $n^{k-\rank(A)}$ (see, e.g., case (ii) of Lemma~\ref{lem:keycounting}).

The notion of rank is useful here. For instance, the parameter $m(A)$ which governs the probability threshold for the random Rado theorem for the integers is expressed in terms of $\rank(A_W)$ for $W\subseteq[k]$. We now provide an appropriate generalisation of the notion of rank for the cases where $S$ is either a finite abelian group or (a power of) a finite subset of a field.

\begin{define}[Rank for finite abelian groups]\label{def:rankfinite}
Let $S$ be a finite abelian group. Let $A$ be an $\ell\times k$ integer matrix and let $f_A:S^k\to S^\ell$ be the function defined by $f_A:x\mapsto Ax$. We define
\begin{align*}
\rank_S(A):=
\begin{cases}
\log_{|S|} |\im(f_A)| & \text{if } |S|>1;\\
0 & \text{if } |S|=1.
\end{cases}
\end{align*}
\end{define}

By the first isomorphism theorem for groups, we have $|S|^k=|\im(f_A)|\cdot|\ker(f_A)|$. In particular, the number of solutions to $Ax=b$ is equal to 
\begin{align}\label{eq:rank1}
\begin{cases}
|\ker(f_A)|=\frac{|S|^k}{|\im(f_A)|}=|S|^{k-\rank_S(A)} & \text{if } b \in \im(f_A);\\
0 & \text{otherwise.}
\end{cases}
\end{align}

\smallskip

A motivation for the choice of $\rank_S(A)$ in Definition~\ref{def:rankfinite} is that it will allow us to rewrite $p_W(A,S)$ from Definition~\ref{def:threshold} as a power of $|S|$. Indeed, in all applications of the random Rado lemma discussed in Sections~\ref{sec:primes} and~\ref{sec:easyapplications}, the probability threshold is expressed in such form.  For example, it follows immediately from equation~\eqref{eq:rank1} that if $A$ is an $\ell \times k$ integer matrix and $S$ is a finite abelian group then $|\Sol_S^A([k])|=|S|^{k-\rank_S(A)}$ and so $p_{[k]}(A,S)=|S|^{-(k-1-\rank_S(A))/(k-1)}$.

\smallskip

Next, we consider the case where $S=L^d$ for some finite subset $L$ of a field $\mathbb F$ and $d\in\mathbb N$. We are mostly interested in integer matrices, but it will be convenient to first consider matrices with entries in $\mathbb F$. We write $0_{\mathbb F}$ and $1_{\mathbb F}$ for the additive and multiplicative identities of $\mathbb F$, respectively.

\begin{define}\label{def:rankPID}
Let $S=L^d$ where $L$ is a finite subset of a field $\mathbb F$ and $d\in\mathbb N$. Let $A$ be a matrix with entries from $\mathbb F$. If $L=\{0_{\mathbb F}\}$, let $\rank_S(A):=0$. Otherwise, let $\rank_S(A)$ be the rank of $A$ with respect to $\mathbb F$.
\end{define}

The following simple fact shows that $\rank_S(A)$ in Definition~\ref{def:rankPID} is well-defined: the value of $\rank_S(A)$ is independent of the choice of $\mathbb F$.

\begin{fact}\label{fact:FsubsetF'}
Let $ \mathbb F \subseteq \mathbb F'$ be two fields. Let $A$ be a matrix whose entries lie in $\mathbb F$. Then, the rank of $A$ with respect to $\mathbb F$ is the same as the rank of $A$ with respect to $\mathbb F'$.
\end{fact}
\begin{proof}
The rank of $A$ is the size of the largest square submatrix of $A$ with non-zero determinant. The determinant of a submatrix is computed by performing multiplications and additions using the entries of the submatrix. Let $\mathbb F_0$ be the field generated by the entries of $A$. Since the entries of $A$ lie in $\mathbb F_0\subseteq \mathbb F\subseteq \mathbb F'$ and  $\mathbb F_0$ is closed under addition and multiplication, the determinant of any submatrix of $A$ is the same whether we are working in $\mathbb F$ or $\mathbb F'$. 
\end{proof}

Definition~\ref{def:rankPID} easily extends to matrices with integer entries by treating the entries of such matrices as elements of the field $\mathbb F$.

\begin{define}[Rank for powers of finite subsets of fields]\label{def:rankPIDinteger}
Let $S=L^d$ where $L$ is a finite subset of a field $\mathbb F$ and $d\in\mathbb N$. Let $A=\{a_{ij}\}$ be an $\ell\times k$ matrix with integer entries. If $L=\{0_{\mathbb F}\}$, let $\rank_S(A):=0$. Otherwise, let $\rank_S(A)$ be the rank of $B$ with respect to $\mathbb F$, where $B:=\{b_{ij}\}$ is the $\ell\times k$ matrix with $b_{ij}:=a_{ij}\cdot 1_{\mathbb F}\in \mathbb F$.
\end{define}

Note that the entries of the matrix $B$ in Definition~\ref{def:rankPIDinteger} lie in the subfield $\langle 1_{\mathbb F}\rangle\subseteq \mathbb F$ generated by $1_{\mathbb F}$. Thus, by Fact~\ref{fact:FsubsetF'}, the rank of $B$ is the same with respect to $\mathbb F$ and $\langle 1_{\mathbb F}\rangle$. If $\mathbb F$ and $\mathbb F'$ are non-trivial fields with $L \subseteq\mathbb F$ and $L \subseteq \mathbb F'$, then $1_{\mathbb F}=1_{\mathbb F'}$ and in particular $\langle 1_{\mathbb F}\rangle=\langle 1_{\mathbb F'}\rangle$. Thus $\rank_S(A)$ from Definition~\ref{def:rankPIDinteger} is independent from different choices of $\mathbb F$ and so it is well-defined.

The following fact states that Definition~\ref{def:rankPIDinteger} is consistent with Definition~\ref{def:rankfinite}.

\begin{fact}\label{fact:equiv}
Let $A=\{a_{ij}\}$ be an $\ell\times k$ integer matrix and $S=L^d$ where $L$ is a finite subset of a field $\mathbb F$ and $d\in\mathbb N$. If $S$ is a finite abelian group, then $\rank_S(A)$ from Definition~\ref{def:rankPIDinteger} is the same as $\rank_S(A)$ from Definition~\ref{def:rankfinite}.
\end{fact}

\begin{proof}
As $S$ is a finite abelian group, $L$ is a finite abelian group too. If $L=\{0_{\mathbb F}\}$ then $\rank_S(A)=0$ in both Definitions~\ref{def:rankfinite} and~\ref{def:rankPIDinteger}; so suppose $L\not=\{0_{\mathbb F}\}$. 

It is well-known that every element in $(\mathbb F,+)$ other than $0_{\mathbb F}$ has (additive) order exactly $p$ where either $p=\infty$ or $p$ is a prime number. Since $L$ is a finite abelian group, we must have $L\cong\mathbb ((\mathbb Z_p)^t,+)$ for some prime $p$ and $t\in\mathbb N$. Furthermore, we have $\langle 1_{\mathbb F}\rangle\cong(\mathbb Z_p,+,\times)$. 

Let $B:=\{b_{ij}\}$ where $b_{ij}:=a_{ij}\cdot 1_{\mathbb F}\in\langle 1_{\mathbb F}\rangle\subseteq \mathbb F$ and let $r$ be the rank of $B$ with respect to $\mathbb F$. By Fact~\ref{fact:FsubsetF'}, $r$ equals the rank of $B$ with respect to $\langle 1_{\mathbb F} \rangle$. This in turn is equal to the rank of $A$ with respect to $\mathbb Z_p$, where we treat the entries of $A$ modulo $p$, by the definition of $B$ and since $\langle 1_{\mathbb F}\rangle\cong(\mathbb Z_p,+,\times)$. 

Let $f_A:S^k\to S^\ell$, $g_A:L^k\to L^\ell$, $h_A:(\mathbb Z_p)^k\to (\mathbb Z_p)^\ell$ be defined by $f_A,g_A,h_A:x\mapsto Ax$. Note that the rank of $h_A$ is precisely $r$ by the previous observations, and so $|\im(h_A)|=|\mathbb Z_p|^r=p^r$. It is easy to see that
$$|\im(f_A)|=|\im(g_A)|^d=|\im(h_A)|^{dt}.$$
In particular, $|\im(f_A)|=p^{rdt}=|S|^r$ and so $r=\log_{|S|}|\im(f_A)|$, as required.
\end{proof}

\smallskip

Next, we prove a general upper bound for the number of projected solutions for finite abelian groups and powers of finite subsets of fields. We will use this tool in Section~\ref{sec:6.3} to simplify conditions (A3)--(A6) for these settings.

\begin{lemma}[Key bounding result]\label{lem:keycounting}
Let $A$ be an $\ell\times k$ integer matrix. Suppose that either
\begin{itemize}
\item[(i)] $S$ is a finite abelian group or 
\item[(ii)] $S=L^d$ for some finite subset $L$ of a field and $d\in\mathbb N$. 
\end{itemize}
Then for all  $W \subseteq Y \subseteq [k]$ and $w_0 \in S^{|W|}$ we have
\begin{align}\label{eq:key1}
|\Sol_S^A(w_0,W,Y)|\le|S|^{|Y|-|W|-\rank_{S}(A_{\overline{W}})+\rank_{S}(A_{\overline{Y}})}.    
\end{align}
\end{lemma}
Note that here we define $\rank_S(A_{\overline{W}}):=0$ if 
$\overline{W}=\emptyset$.
\begin{proof}
Fix $W\subseteq Y\subseteq[k]$ and $w_0\in S^{|W|}$.
If $W=Y$ then $|\Sol_S^A(w_0,W,Y)|=1$ and (\ref{eq:key1}) holds. So we may assume that $W \subset Y$.

\smallskip

{\bf Case (i):  $S$ is a finite abelian group.} We double count the quantity $|\Sol_S^A(w_0,W,[k])|$. This is equal to the number of solutions $x$ to $A_{\overline W}x=-A_Ww_0$ which in turn is equal to either $0$ or $|S|^{k-|W|-\rank_S({A_{\overline W}})}$ by~\eqref{eq:rank1}. In the former case, $|\Sol_S^A(w_0,W,[k])|=0$ implies $|\Sol_S^A(w_0,W,Y)|=0$ and so~\eqref{eq:key1} holds. Hence, suppose that
\begin{equation}\label{eq:count1}
|\Sol_S^A(w_0,W,[k])|=|S|^{k-|W|-\rank_S({A_{\overline W}})}.
\end{equation}
In particular, we have $|\Sol_S^A(w_0,W,Y)|>0$. Let $y_0\in\Sol_S^A(w_0,W,Y)$. By the same argument as above, we have $|\Sol_S^A(y_0,Y,[k])|=|S|^{k-|Y|-\rank_S(A_{\overline Y})}$. Thus,
$$|\Sol_S^A(w_0,W,[k])|=\sum_{y_0\in\Sol_S^A(w_0,W,Y)}|\Sol_S^A(y_0,Y,[k])|=|\Sol_S^A(w_0,W,Y)|\cdot|S|^{k-|Y|-\rank_S(A_{\overline Y})}.$$ 
Combining the above with~\eqref{eq:count1} yields
$$|\Sol_S^A(w_0,W,Y)|=|S|^{|Y|-|W|-\rank_S(A_{\overline W})+\rank_S(A_{\overline Y})},$$ 
as required.

\smallskip

{\bf Case (ii): $S=L^d$ for some finite subset $L$ of a field $\mathbb F$ and $d\in\mathbb N$.} If $L=\{0_{\mathbb F}\}$ then $|\Sol_S^A(w_0,W,Y)|\le1$ and so~\eqref{eq:key1} holds. Hence, assume $L\not=\{0_{\mathbb F}\}$. For the rest of this case, we treat the entries of $A$ as elements of $\mathbb F$ (namely, we view $n\in\mathbb N$ as $n\cdot 1_{\mathbb F}$). Let $I_0\subseteq[k]$ be the indices of a maximal set of linearly independent columns in $A_{\overline Y}$ and $I_1$ be the indices of a maximal set of columns in $A_{Y\setminus W}$ such that the columns with indices $I_0\cup I_1$ are linearly independent. 
Then recalling $W \subseteq Y$, we have $\rank_S(A_{\overline{W}})=|I_0 \cup I_1|$.

\begin{claim}\label{claim:uvsols}
Let $u,v$ be two solutions $Au=Av=0$ in $\mathbb F$. If $u_{Y\setminus I_1}=v_{Y\setminus I_1}$ then $u_{I_1}=v_{I_1}$.
\end{claim}
\begin{proofclaim}
Since $A(u-v)=0$ and $u_{Y\setminus I_1}=v_{Y\setminus I_1}$
$$A_{I_1}(u-v)_{I_1}+A_{\overline Y}(u-v)_{\overline Y}=0,$$
and thus $A_{I_1}(u-v)_{I_1}\in\langle A_{\overline Y}\rangle$, where $\langle A_{\overline Y}\rangle$ denotes the vector space spanned by the columns of $A_{\overline Y}$. By the maximality of $I_0$ we have $\langle A_{\overline Y}\rangle=\langle A_{I_0}\rangle$ and so $A_{I_1}(u-v)_{I_1}\in\langle A_{I_0}\rangle$. Since the columns with indices $I_0\cup I_1$ are linearly independent, this implies that $(u-v)_{I_1}=0$.\qedclaim
\end{proofclaim}

Now we bound $|\Sol_S^A(w_0,W,Y)|$ from above. Recall $w_0\in S^{|W|}$ and $S=L^d$. We can therefore write $w_0$ in the form $w_0=(t_1,\dots,t_d)$ where $t_i\in L^{|W|}$. Note that $$|\Sol_S^A(w_0,W,Y)|=\prod_{i=1}^d|\Sol_L^A(t_i,W,Y)|,$$
so it suffices to prove that $|\Sol_L^A(t_i,W,Y)|\le|L|^{|Y|-|W|-\rank_S(A_{\overline W})+\rank_S(A_{\overline Y})}$ for every $i$. By Claim~\ref{claim:uvsols}, there is at most one projected solution for every choice of the entries with indices $Y\setminus I_1$. Note that the entries with indices in $W$ are already fixed and $W$ and $I_1$ are disjoint, thus   
$$|\Sol_L^A(t_i,W,Y)|\leq|L|^{|Y|-|W|-|I_1|}=|L|^{|Y|-|W|-|I_0\cup I_1|+|I_0|}=|L|^{|Y|-|W|-\rank_S(A_{\overline W})+\rank_S(A_{\overline Y})}.$$
\end{proof}

Applying inequality~\eqref{eq:key1} from Lemma~\ref{lem:keycounting} with $W=\emptyset$ and $Y=[k]$ immediately implies that there are at most $|S|^{k-\rank_S(A)}$ solutions to $Ax=0$ in $S$. Informally, we say $(A,S)$ is {\it rich} if this bound is close to tight.

\begin{define}[Richness]
Let $A$ be an $\ell\times k$ integer matrix, let $\eps>0$ and let $S$ be either a finite abelian group or a power of a finite subset of a field.  We say $(A,S)$ is {\it $\eps$-rich} if there are at least $\eps|S|^{k-\rank_S(A)}$ solutions to $Ax=0$ in $S$.
\end{define}

Intuitively, richness is a measure of how strongly the elements of $S$ are related within the arithmetic structure of $S$. For example, if $S$ is a finite abelian group then $(A,S)$ is $1$-rich by~\eqref{eq:rank1} for any integer matrix $A$. If $A$ is partition regular then
Theorem~\ref{thm:fgrsupersat} implies (in a very strong form) that
 $S=[n]$ is $\delta$-rich for some $\delta >0$ and $n \in \mathbb N$ sufficiently large. On the other hand, this is not the case for $S=\mathbf P_n$ (see, e.g., Theorem~\ref{gt}), and indeed, the primes have a very loose arithmetic structure. 

The next corollary states that if $(A,S)$ is rich then the bound given by Lemma~\ref{lem:keycounting} is close to tight for certain projected solutions.

\begin{col}\label{col:keycounting}
Let $A$ be an $\ell\times k$ integer matrix and let $0<\eps \leq 1$. Suppose that either
\begin{itemize}
\item[(i)] $S$ is a finite abelian group or 
\item[(ii)] $S=L^d$ for some finite subset $L$ of a field and $d\in\mathbb N$.  
\end{itemize}
If $(A,S)$ is $\eps$-rich, then for every non-empty $Z\subseteq[k]$ we have
\begin{align}\label{eq:key2}
\eps|S|^{|Z|-\rank_S(A)+\rank_S(A_{\overline Z})}\le|\Sol_S^A(Z)|\le|S|^{|Z|-\rank_S(A)+\rank_S(A_{\overline Z})}.
\end{align}
\end{col}

\begin{proof}
Fix $Z\subseteq[k]$. Applying~\eqref{eq:key1} from Lemma~\ref{lem:keycounting} with $Y:=Z$ and $W:=\emptyset$ yields the upper bound
$$|\Sol_S^A(Z)|\le|S|^{|Z|-\rank_S(A)+\rank_S(A_{\overline Z})}.$$
Fix $z_0\in\Sol_S^A(Z)$. By applying~\eqref{eq:key1} with $Y:=[k]$ and $W:=Z$, we obtain that
\begin{equation}\label{eq:upperbound1}
|\Sol_S^A(z_0,Z,[k])|\le|S|^{k-|Z|-\rank_S(A_{\overline Z})}.
\end{equation}
The total number of solutions to $Ax=0$ in $S$ is 
$$|\Sol_S^A([k])|=\sum_{z_0\in\Sol_S^A(Z)}|\Sol_S^A(z_0,Z,[k])|\stackrel{\eqref{eq:upperbound1}}{\le}|\Sol_S^A(Z)|\cdot|S|^{k-|Z|-\rank_S(A_{\overline Z})}.$$
Combining the above with $|\Sol_S^A([k])|\ge\eps|S|^{k-\rank_{S}(A)}$ yields
\begin{align*}
|\Sol_S^A(Z)|&\ge\eps|S|^{k-\rank_S(A)}\cdot|S|^{-k+|Z|+\rank_S(A_{\overline Z})}=\eps|S|^{|Z|-\rank_S(A)+\rank_S(A_{\overline Z})}.
\end{align*}
\end{proof}

\subsection{Conditions (A3)--(A6) for finite groups and powers of finite subsets of fields}\label{sec:6.3}
In this subsection, we use the key bounding results (Lemma~\ref{lem:keycounting} and Corollary~\ref{col:keycounting}) from the previous subsection to investigate conditions (A3)--(A6) when the sets $S_n$ in the sequence $(S_n)_{n\in\mathbb N}$ are either finite abelian groups or powers of finite subsets of fields. We start by showing that in this case, richness implies condition (A3). 

\begin{lemma}\label{lem:simplifiedA3}
Let $A$ be an $\ell\times k$ integer matrix. Suppose that either
\begin{itemize}
\item[(i)] $S$ is a finite abelian group or 
\item[(ii)] $S=L^d$ for some finite subset $L$ of a field and $d\in\mathbb N$. 
\end{itemize}
If there exists $B>0$ such that $(A,S)$ is $(1/B)$-rich, then $(A,S)$ is $B$-extendable. 
\end{lemma}
\begin{proof}
Let $W\subseteq Y\subseteq[k]$ and $w_0\in S^{|W|}$. It follows from Lemma~\ref{lem:keycounting} and Corollary~\ref{col:keycounting} that
\begin{align*}
\frac{B|\Sol_S^A(Y)|}{|\Sol_S^A(W)|}\stackrel{\eqref{eq:key2}}{\ge}\frac{|S|^{|Y|-\rank_S(A)+\rank_S(A_{\overline Y})}}{|S|^{|W|-\rank_S(A)+\rank_S(A_{\overline W})}}=|S|^{|Y|-|W|-\rank_S(A_{\overline W})+\rank_S(A_{\overline Y})}\stackrel{\eqref{eq:key1}}{\ge}|\Sol_S^A(w_0,W,Y)|.
\end{align*}
Hence $(A,S)$ is $B$-extendable.
\end{proof}

For finite abelian groups, we deduce the following corollary.

\begin{col}\label{col:simplifiedA3}
Let $A$ be an $\ell\times k$ integer matrix and $S$ be a finite abelian group. Then $(A,S)$ is $1$-extendable.
\end{col}
\begin{proof}
By~\eqref{eq:rank1}, $(A,S)$ is $1$-rich. Lemma~\ref{lem:simplifiedA3} implies that $(A,S)$ is $1$-extendable.
\end{proof}

Next, we consider conditions (A4), (A5) and (A6). First, we give the definition of {\it abundance}. This notion appeared in Section~\ref{sec:graphsintegers} when we considered the random Ramsey problem for the integers.

\begin{define}[Abundancy]\label{def:abundant}
Let $A$ be an $\ell\times k$ integer matrix and $S$ be either a finite abelian group or a power of a finite subset of a field. We say $(A,S)$ is {\it abundant} if $\rank_S(A)=\rank_S(A_{\overline W})$ for every $W\subseteq[k]$ with $|W|=2$.
\end{define}

Before we show how abundancy can help simplify conditions (A4), (A5) and (A6), we show it is a mild requirement in the case of powers of finite subsets of a field $\mathbb F$. For brevity, we will write $0$ and $1$ for the additive and multiplicative identities of $\mathbb F$, respectively. The next lemma establishes that irredundancy and $3$-partition regularity imply abundancy. We remark that the field $\mathbb F$ in the statement of the lemma may be finite or infinite.

\begin{lemma}\label{lem:primpliesabundant}
Let $A$ be an $\ell \times k$ integer matrix and $S=L^d$ for some finite subset $L$ of a field $\mathbb F$ and $d\in\mathbb N$. If $A$ is irredundant with respect to $\mathbb F^d$ and $3$-partition regular in $S\setminus\{0\}^d$, then $(A,S)$ is abundant.
\end{lemma}
\begin{proof}
If $L=\{ 0\}$ then $\rank_S(A')=0$ for any integer matrix $A'$, and we are done. Thus, suppose that $L\ne \{ 0\}$. As before, we can view $A$ as a matrix with entries in $\mathbb F$.  Let $C_1,\dots,C_k$ denote the columns of $A$. To prove that $(A,S)$ is abundant it suffices to
  show that $C_1$ and $C_2$ are both spanned by $C_3,...,C_k$ in $\mathbb F$.  For the rest of the proof, given a solution $x$ to $Ax=0$, we write $x_i$ for the $i$th entry of $x$. 

If $C_1$ cannot be expressed as a linear combination of $C_2,\dots,C_k$, then a solution $x$ to $Ax=0$ in $\mathbb F$ satisfies $x_1=0$. In turn, this implies that a solution $x$ to $Ax=0$ in $\mathbb F^d$ satisfies $x_1=\{0\}^d$. This contradicts the assumption that $A$ is $3$-partition regular in $S\setminus\{0\}^d$.

Hence, $C_1$ is spanned by $C_2,\dots,C_k$. In particular, there exists $\alpha\in \mathbb F$ such that $C_1+\alpha C_2$ is spanned by $C_3,\dots,C_k$. If $\alpha$ is not unique then one can write $C_2$ as a linear combination of $C_3,\dots,C_k$  and thus we are done. 

Suppose therefore that $\alpha$ is unique. Thus, any solution $x$ to $Ax=0$ in $\mathbb F^d$ must satisfy $x_2=\alpha\cdot x_1$. If $\alpha=1$ then $x_1=x_2$ and $A$ is therefore not irredundant with respect to $\mathbb F^d$, a contradiction. If $\alpha=0$ then $x_2=\{0\}^d$, and so $A$ is not $3$-partition regular in $S\setminus\{0\}^d$, a contradiction. Thus, $\alpha\not=0,1$. 

Consider the graph $G$ with vertex set $V(G)=S$ and edge set $E(G)=\{ss': s'=\alpha\cdot s\}$ (here the edges are unordered pairs). Since $\alpha\not=0$, each $s\in V(G)$ is incident to at most two edges, namely $\alpha\cdot s$ and $\alpha^{-1}\cdot s$. Thus $\Delta(G)\le2$. In particular, there exists a proper $3$-vertex-colouring of $G$ which  induces a $3$-colouring of $S$. 

Recall that for any solution $x$ of $Ax=0$ in $\mathbb F^d$ we have $x_2=\alpha\cdot x_1$. Since $\alpha\not=1$, $x_1$ and $x_2$ are distinct unless $x_1=x_2=\{0\}^d$. In the former case, $x_1$ and $x_2$ are adjacent in $G$ and so they are coloured differently. It follows that there is no monochromatic solution to $Ax=0$ in $S\setminus\{0\}^d$. Hence $A$ is not $3$-partition regular in $S\setminus\{0\}^d$, a contradiction.
\end{proof}

Note that the statement of Lemma~\ref{lem:primpliesabundant} does not necessarily hold for finite abelian groups. Indeed, consider the example $A=\begin{pmatrix} 2 & 2 & 1 & 1\end{pmatrix}$ and $S=\mathbb{Z}_6^n$.  The columns of $A$ sum to zero (mod $6$), so $A$ satisfies the $6$-columns condition. 
By Theorem~\ref{thm:exponentsupersat}, for fixed $r$ and $n$ sufficiently large, every $r$-colouring of $\mathbb{Z}_6^n\setminus\{0\}^n$ yields (many) monochromatic solutions to $Ax=0$. 
Additionally $A$ is irredundant with respect to $\mathbb{Z}_6^n$. This follows easily by observing that $(2,4,1,5)$ is a $4$-distinct solution to $Ax=0$ in $\mathbb{Z}_6$. However $(A,\mathbb{Z}_6^n)$ is not abundant, since the rank decreases after deleting the last two columns of $A$.

\smallskip

Next, we show how abundancy can help simplify conditions (A4), (A5) and (A6). First, if $(A,S)$ is abundant and rich then condition (A4) trivially holds. 

\begin{lemma}\label{lem:spreadedges}
Let $A$ be an $\ell\times k$ integer matrix. Suppose that either
\begin{itemize}
\item[(i)] $S$ is a finite abelian group or 
\item[(ii)] $S=L^d$ for some finite subset $L$ of a field and $d\in\mathbb N$.
\end{itemize}
Furthermore, assume that $(A,S)$ is abundant and $\eps$-rich for some $0<\eps\leq 1$. Then, for every $W\subseteq Y\subseteq[k]$ with $|W|=1$ and $w_0 \in S$ we have $|\Sol_S^A(w_0,W,Y)|\le |\Sol_S^A(Y)|/(\eps |S|)$.
\end{lemma}
\begin{proof}
By Lemma~\ref{lem:keycounting} we have 
\begin{align*}
|\Sol_S^A(w_0,W,Y)|\le|S|^{|Y|-1-\rank_{S}(A_{\overline{W}})+\rank_{S}(A_{\overline{Y}})},
\end{align*}
while by Corollary~\ref{col:keycounting} we have
\begin{align*}
|\Sol_S^A(Y)|\ge\eps|S|^{|Y|-\rank_{S}(A)
+\rank_{S}(A_{\overline{Y}})}.
\end{align*}
Combining the two inequalities above yields
\begin{align}\label{newlabel1}
|\Sol_S^A(w_0,W,Y)|\le \frac{1}{\eps} \cdot\frac{|\Sol_S^A(Y)|}{|S|}\cdot|S|^{\rank_{S}(A)
-\rank_{S}(A_{\overline{W}})}.
\end{align}
Let $W\subseteq W'$ with $|W'|=2$ and note that $\rank_S(A_{\overline{W'}})\le\rank_S(A_{\overline W})\le\rank_S(A)$. We have $\rank_S(A_{\overline{W'}})=\rank_S(A)$ since $|W'|=2$ and $(A,S)$ is abundant. Thus,
 $\rank_S(A_{\overline W})=\rank_S(A)$ and so (\ref{newlabel1}) implies that $|\Sol_S^A(w_0,W,Y)|\le|\Sol_S^A(Y)|/(\eps |S|)$.    
\end{proof}

Next, we consider conditions (A5) and (A6). If $(A,S)$ is abundant and rich then the following bound holds.

\begin{lemma}\label{lem:simplifiedA4-5}
Let $A$ be an $\ell\times k$ integer matrix. Suppose that either
\begin{itemize}
\item[(i)] $S$ is a finite abelian group or 
\item[(ii)] $S=L^d$ for some finite subset $L$ of a field and $d\in\mathbb N$.
\end{itemize}
Furthermore, assume that $(A,S)$ is abundant and $\eps$-rich for some $0< \eps\leq 1$. Then for any $W \subseteq [k]$ with $|W|=2$, we have
$$\frac{|S|^2}{|\Sol_S^A(W)|}\le\frac{1}{\eps}.$$
\end{lemma}

\begin{proof}
We have
$$\frac{|S|^2}{|\Sol_S^A(W)|}\stackrel{\eqref{eq:key2}}{\le}\frac{|S|^2}{\eps |S|^{|W|-\rank_S(A)+\rank_S(A_{\overline W})}}=\frac{|S|^2}{\eps|S|^2}=\frac{1}{\eps},$$
where the equality follows from the fact that $|W|=2$ and $\rank_S(A)=\rank_S(A_{\overline W})$ since $(A,S)$ is abundant. 
\end{proof}

Observe that if $A$ and $(S_n)_{n\in\mathbb N}$ satisfy the hypothesis of Lemma~\ref{lem:simplifiedA4-5} with the same $\eps>0$ for all $n \in \mathbb{N}$, and $|S_n|\to\infty$ as $n\to\infty$, then condition (A6) trivially holds. 
Condition (A5) is also close to being satisfied: it would suffice to additionally prove that there exists some sequence $(X_n)_{n\in\mathbb N}$ of subsets $X_n\subseteq[k]$ with $|X_n|\ge3$ for all $n\in\mathbb N$ where $p_{X_n}(A,S_n)=\Omega(\hat p(A,S_n))$ and such that 
$$\frac{p_{W'_n}(A,S_n)}{p_{X_n}(A,S_n)}\to0$$
as $n\to\infty$ for any sequence $(W'_n)_{n\in\mathbb N}$ with $W'_n\subset X_n$ and $|W'_n|\ge2$. The next result states that this is the case for powers of subsets of fields as long as we have abundancy and richness.

\begin{lemma}\label{lem:pWXtozero}
Let $0< \eps \leq 1$ and let $A$ be an $\ell\times k$ integer matrix and $S=L^d$ for some finite subset $L$ of a field and $d\in\mathbb N$. Suppose $\rank_S(A)>0$. If $(A,S)$ is $\eps$-rich  and abundant, then there exists $X\subseteq[k]$ such that $|X|\ge3$, $p_X(A,S)\ge\eps\hat p(A,S)$ and for every $W'\subset X$ with $|W'|\ge2$ we have
$$\frac{p_{W'}(A,S)}{p_X(A,S)}\le(1/\eps)|S|^{-1/k^2}.$$
\end{lemma}

\begin{proof}
For every $Z\subseteq[k]$ with $|Z|\ge2$ set 
$$D(Z):=\frac{|Z|-1-\rank_S(A)+\rank_S(A_{\overline Z})}{|Z|-1}.$$
As $0<\eps \leq 1$, by definition of  $p_Z(A,S)$ and Corollary~\ref{col:keycounting}, we have  that
\begin{equation}\label{eq:pZbounds}
|S|^{-D(Z)}\stackrel{\eqref{eq:key2}}{\le} p_Z(A,S)\stackrel{\eqref{eq:key2}}{\le}(1/\eps)|S|^{-D(Z)}.
\end{equation}

Pick $X\subseteq[k]$ with $|X|\ge2$ so that the function $D(X)$ is minimised;  additionally choose $X$ so that $|X|$ is as small as possible under this assumption. We have
$$p_X(A,S)\stackrel{\eqref{eq:pZbounds}}{\ge}|S|^{-D(X)}\ge|S|^{-D(Z)}\stackrel{\eqref{eq:pZbounds}}{\ge}\eps p_Z(A,S)$$
for every $Z\subseteq[k]$ with $|Z|\ge2$. By maximising over $Z$, it follows that $p_X(A,S)\ge\eps\hat p(A,S)$, as required.
Moreover, 
 $|X|\ge3$. Indeed, for any $Z\subseteq[k]$ with $|Z|=2$, we have $\rank_S(A)=\rank_S(A_{\overline Z})$ since $(A,S)$ is abundant. In particular, we have $D(Z)=1$. On the other hand, $D([k])=(k-1-\rank_S(A))/(k-1)<1$ since $\rank_S(A)>0$. Thus, $|X|\neq 2$. 

Let $W'\subset X$ with $|W'|\ge2$. We have
$$\frac{p_{W'}(A,S)}{p_X(A,S)}\stackrel{\eqref{eq:pZbounds}}{\le}(1/\eps)|S|^{-D(W')+D(X)}.$$

By the minimality of $X$, we have $D(X)<D(W')$. Since $S=L^d$ for some finite subset $L$ of a field, $\rank_S(\cdot)$ is always an integer and so 
$$D(W')-D(X)\ge\frac{1}{(|X|-1)(|W'|-1)}\ge\frac{1}{k^2}.$$
It follows that 
$$\frac{p_{W'}(A,S)}{p_X(A,S)}\le(1/\eps)|S|^{-1/k^2}.$$
\end{proof}

\subsection{Simplified random Rado lemma}
In this subsection we state a simplified version of the random Rado lemma for finite abelian groups and powers of finite subsets of fields. This version of the random Rado lemma will be used to prove most of our applications (see Section~\ref{sec:proofapplications}). 
First, we introduce a generalisation of the parameter $m(A)$ to this setting.

\begin{define}\label{def:mGA}
Let $S$ be either a finite abelian group or a power of a finite subset of a field. Let $A$ be an $\ell\times k$ integer matrix. We set
\begin{align}\label{eq:mGAdef}
m_S(A) := \max_{\stackrel{W \subseteq [k]}{|W| \geq 2}} \frac{|W|-1}{|W|-1+\rank_S(A_{\overline{W}}) - \rank_S(A)}.
\end{align}
Here we recall that $\rank_S(A_{\overline{W}}):=0$ if 
$\overline{W}=\emptyset$.
We call $A$ \emph{strictly balanced with respect to $S$} if the expression in (\ref{eq:mGAdef}) is  maximised precisely when $W=[k]$. 
\end{define}
Note that if $S:=[n]$ then as $S$ is a subset of the field $\mathbb Q$, we have that $m_S(A)=m(A)$. Indeed, in this setting the function
$\rank_S$ is the same as the rank function appearing in the definition of $m(A)$.

\begin{remark}\label{rmk:m(a)abb}
Observe that $m_S(A)$ is not always well-defined since the denominator in~\eqref{eq:mGAdef} might be zero. However, if $(A,S)$ is abundant then for all $|W| \geq 2$ we have 
\begin{align}\label{eq:rankW2}
\rank_S(A_{\overline{W}}) \geq \rank_S(A) - |W| +2,
\end{align} 
since the rank of a matrix can decrease by at most one by deleting a column.\footnote{This fact 
follows immediately from the definition of $\rank_S(A)$ when
$S$ is the power of a subset of a field $F$. When $S$ is a finite abelian group, it can be proved using equation~\eqref{eq:rank1}.} If~(\ref{eq:rankW2}) is satisfied then the denominator in~\eqref{eq:mGAdef} is always strictly positive and so $m_S(A)$ is well-defined in this case.
\end{remark}

Before proceeding further, we provide several examples and facts about the parameter $m_S(A)$. Firstly, we show that $m_{\mathbb{Z}_4}(A)=4/3$ for $A=(2\;\; 2\; -2)$, as stated in Section~\ref{sec:easyapplications}. 

\begin{examp}\label{examp:mZ4}
Let $A=(2\;\; 2\; -2)$. The number of solutions $x=(x_1,x_2,x_3)$ to $Ax=0$ in $\mathbb Z_4$ is precisely $2\cdot4^2=4^{3-1/2}$: we have $4^2$ choices for $x_1,x_2$ and, for each  fixed choice, we have $2$ choices for $x_3$. It follows from (\ref{eq:rank1}) that $\rank_{\mathbb Z_4}(A)=1/2$. Now, let $W\subseteq[3]$. If $|W|=3$ we trivially have $\rank_{\mathbb Z_4}(A_{\overline W})=0$. If $|W|=2$ then $A_{\overline W}x=0$ is the same as $2x=0$, which has $2=4^{1-1/2}$ solutions in $\mathbb Z_4$. So we have $\rank_{\mathbb Z_4}(A_{\overline W})=1/2$ again. Thus,
$$m_{\mathbb{Z}_4}(A)=\max\left\{\frac{3-1}{3-1+0-1/2},\frac{2-1}{2-1+1/2-1/2}\right\}=4/3.$$
\end{examp}

Next, we highlight certain similarities  and differences  between the parameters $m(A)$ and $m_S(A)$. Recall that if $A$ is a $1\times k$ integer matrix with non-zero entries then $A$ is strictly balanced and $m(A)=\frac{k-1}{k-2}$. This is not always the case for $m_S(A)$, as shown in the next example.

\begin{examp}\label{examp:mSnotstrictlyblc}
Let $m\in\mathbb N$ and $A=(1\;\; m\;\; m)$. The number of solutions $x=(x_1,x_2,x_3)$ to $Ax=0$ in $\mathbb Z_{2m}$ is precisely $(2m)^2=(2m)^{3-1}$. It follows that $\rank_{\mathbb Z_{2m}}(A)=1$. 
If $W=[3]$ we trivially have $\rank_{\mathbb Z_4}(A_{\overline W})=0$. If $W=\{2,3\}$ then $A_{\overline W}x=0$ is the same as $x=0$, which has~$1=(2m)^{1-1}$ solution in $\mathbb Z_{2m}$, and so $\rank_{\mathbb Z_{2m}}(A_{\overline W})=1$. If $W=\{1,2\}$ or $W=\{1,3\}$ then $A_{\overline W}x=0$ is the same as $mx=0$, which has $m=(2m)^{1-\log_{2m}2}$ solutions in $\mathbb Z_{2m}$, and so $\rank_{\mathbb Z_{2m}}(A_{\overline W})=\log_{2m}2$. Thus,
$$m_{\mathbb Z_{2m}}(A)=\max\left\{\frac{3-1}{3-1+0-1},\frac{2-1}{2-1+1-1},\frac{2-1}{2-1+\log_{2m}2-1}\right\}=\max\{2,1,(\log_{2m}2)^{-1}\}.$$
In particular, for $m\ge2$ we have $m_{\mathbb Z_{2m}}(A)=(\log_{2m}2)^{-1}$. Furthermore, for $m>2$, the value of $m_{\mathbb Z_{2m}}(A)$ is achieved precisely by $W=\{1,2\}$ and $W=\{1,3\}$, and so $A$ is not strictly balanced with respect to~$\mathbb Z_{2m}$.       
\end{examp}

Despite the example above, by imposing some additional restrictions on the $1\times k$ matrix $A$, one can infer that $A$ is strictly balanced with respect to $S$ and, furthermore, $m_S(A)=m(A)=\frac{k-1}{k-2}$.

\begin{examp}\label{examp:mSstrictlyblc}
Let $A$ be a $1\times k$ integer matrix with non-zero entries. Let $S$ be a finite subset of an abelian group $G$. Suppose that $S\neq\{0_G\}$ and that either
\begin{itemize}
\item $S$ is a finite abelian group or
\item $S$ is a power of a finite subset of a field $\mathbb F$.
\end{itemize}
Furthermore, suppose that 
\begin{itemize}
    \item[$(\alpha)$] for every entry $a$ of $A$ and $s\in S\setminus\{0_G\}$ we have $a\cdot s\neq0$. 
\end{itemize} 
Then $A$ is strictly balanced with respect to $S$ and $m_S(A)=m(A)=\frac{k-1}{k-2}$.

Indeed, let $W\subseteq[k]$ be non-empty. If $S$ is a 
power of a finite subset of a field $\mathbb F$ then 
$(\alpha)$ implies $A_W$ is not the $1\times|W|$ zero 
vector in $\mathbb F$ and thus $A_W$ has rank $1$ 
with respect to $\mathbb F$. In particular, 
$\rank_S(A_W)=1$. Similarly, if $S$ is a finite abelian 
group then the number of solutions to $A_Wx=0$ in $S$ 
is precisely $|S|^{|W|-1}$: we can pick the first 
$|W|-1$ entries of $x$ arbitrarily and for each such 
choice we have $1$ choice for the last entry (this 
follows from $(\alpha)$). Thus, we have $\rank_S(A_W)=1$. 
In both cases, we have
$$m_S(A)=\max_{2\le w\le k-1}\left\{\frac{k-1}{k-1+0-1},\frac{w-1}{w-1+1-1}\right\}=\frac{k-1}{k-2}.$$ 
In particular, the value of $m_S(A)$ is achieved precisely when 
$W=[k]$ and so $A$ is strictly balanced with respect to $S$. 
\end{examp}

Note that if a $1\times k$ integer matrix $A$ is strictly balanced then $m(A)=\frac{k-1}{k-2}$. 
Conversely, if $A$ is strictly balanced with respect to $S$, it may not be the case that $m_S(A)=\frac{k-1}{k-2}$. For instance, in Example~\ref{examp:mZ4}, we saw $A=(2\;\; 2\; -2)$ is strictly balanced with respect to $\mathbb Z_4$ but~$m_{\mathbb{Z}_4}(A)=4/3<2$. 

Another observation is that, for $A$ fixed, $m_S(A)$ may change depending on $S$. For example, consider $A=(1\;\; 3\;\; 3)$. From Example~\ref{examp:mSnotstrictlyblc} we have that $m_{\mathbb Z_6}(A)=(\log_6 2)^{-1}\approx 2.58$, while from Example~\ref{examp:mSstrictlyblc} we have $m_{\mathbb Z_2}(A)=\frac{3-1}{2-1}=2$ (since $A$ and $\mathbb Z_2$ satisfy property $(\alpha)$).

There are  cases where $m_S(A)$ is independent of $S$, as long as $S$ is non-trivial. Indeed, let $A$ be a $1\times k$ matrix whose entries are all from $\{1,-1\}$. Then $(\alpha)$ from Example~\ref{examp:mSstrictlyblc} holds for every $S$ which is either a finite abelian group or a power of a finite subset of a field, provided that $S$ includes a non-identity element. In particular, we have $m_S(A)=\frac{k-1}{k-2}$ for any such $S$.

\smallskip

The next result states that $\hat p(A,S)$ and $|S|^{-1/m_S(A)}$ are within a constant factor of each other, as long as we are in a rich setting.

\begin{lemma}\label{lem:pSA}
Let $A$ be an $\ell\times k$ integer matrix and $0<\eps \leq 1$. Suppose that either
\begin{itemize}
\item[(i)] $S$ is a finite abelian group or 
\item[(ii)] $S=L^d$ for some finite subset $L$ of a field and $d\in\mathbb N$.
\end{itemize} 
If $m_S(A)$ is well-defined and $(A,S)$ is $\eps$-rich  then 
$$|S|^{-\frac{1}{m_S(A)}}\le\hat p(A,S)\le(1/\eps^2)|S|^{-\frac{1}{m_S(A)}}.$$
\end{lemma}
\begin{proof}
For every $Z\subseteq[k]$ with $|Z|\ge2$ let 
$$D(Z):=\frac{|Z|-1-\rank_S(A)+\rank_S(A_{\overline Z})}{|Z|-1}.$$
As $0<\eps \leq 1$, by definition of  $p_Z(A,S)$ and Corollary~\ref{col:keycounting}, we have  that
\begin{equation}\label{eq:pZbounds2}
|S|^{-D(Z)}\stackrel{\eqref{eq:key2}}{\le} p_Z(A,S)\stackrel{\eqref{eq:key2}}{\le}(1/\eps)|S|^{-D(Z)}.
\end{equation}

Pick $X\subseteq[k]$ with $|X|\ge2$ so that the function $D(X)$ is minimised; so $m_S(A)=1/D(X)$. For any $Z\subseteq[k]$ with $|Z|\ge2$, we have 
$$p_X(A,S)\stackrel{\eqref{eq:pZbounds2}}{\ge}|S|^{-D(X)}\ge|S|^{-D(Z)}\stackrel{\eqref{eq:pZbounds2}}{\ge}\eps p_Z(A,S).$$
As $Z$ was arbitrary, it follows that 
$$p_X(A,S)\le\hat p(A,S)\le(1/\eps) p_X(A,S).$$
Combining the above with~\eqref{eq:pZbounds2} yields
$$|S|^{-\frac{1}{m_S(A)}}=|S|^{-D(X)}\stackrel{\eqref{eq:pZbounds2}}{\le}\hat p(A,S)\stackrel{\eqref{eq:pZbounds2}}{\le}(1/\eps^2)|S|^{-D(X)}=(1/\eps^2)|S|^{-\frac{1}{m_S(A)}}.$$
\end{proof}

We can now state the simplified version of the random Rado lemma.

\begin{lemma}[Simplified random Rado lemma]\label{lem:easyRado}
Let $r\ge2$, $k \ge 3$ and $\ell\geq 1$ be integers. Let $(S_n)_{n\in\mathbb N}$ be a sequence where $S_n$ is either a finite abelian group or $S_n=L^d$ for some finite subset $L$ of a field and $d\in\mathbb N$, for every $n\in\mathbb N$. 
Let $A$ be an $\ell \times k$ integer matrix and set $\hat p_n:=\hat p(A,S_n)$. Suppose that 
\begin{itemize}
\item[(C1)] $\hat p_n \to 0$ and $|S_n|\hat p_n \to \infty$ as $n \to \infty$;
\item[(C2)] $(S_n)_{n \in \mathbb N}$ is $(A,r)$-supersaturated;
\item[(C3)] there exists $0<\eps\leq 1$ such that $(A,S_n)$ is $\eps$-rich for every sufficiently large $n\in\mathbb N$; 
\item[(C4)] $(A,S_n)$ is abundant for every sufficiently large $n\in\mathbb N$;
\item[(C5)] given any  $n\in\mathbb N$ there exists $ X_n\subseteq[k]$ with $|X_n|\ge3$  such that $p_{X_n}(A,S_n)=\Omega(\hat p(A,S_n))$ and for every sequence $(W'_n)_{n\in\mathbb N}$ with $W'_n\subset X_n$ and $|W'_n|\ge2$ we have
$$\frac{p_{W'_n}(A,S_n)}{p_{X_n}(A,S_n)}\to0$$
as $n\to\infty$.
\end{itemize}
Then there exist constants $c,C>0$ such that
$$\lim _{n \rightarrow \infty} \mathbb P [ S_{n,p} \text{ is } (A,r)\text{-Rado}]=\begin{cases}
0 &\text{ if } p\leq {c}|S_n|^{-\frac{1}{m_{S_n}(A)}}; \\
 1 &\text{ if } p\geq {C}|S_n|^{-\frac{1}{m_{S_n}(A)}}.\end{cases}$$
\end{lemma}

\begin{remark}\label{rmk:easierRado}
Note that Lemma~\ref{lem:easyRado} can be further simplified if we consider finite abelian groups and powers of finite subsets of fields separately. Namely, if all sets in $(S_n)_{n\in\mathbb N}$ are finite abelian groups then condition (C3) holds since by 
(\ref{eq:rank1}) $(A,S_n)$ is $1$-rich. On the other hand, if all sets in $(S_n)_{n\in\mathbb N}$ are powers of finite subsets of fields and $\rank_{S_n}(A)>0$ for all $n\in\mathbb N$, then condition (C5) follows from conditions (C1), (C3) and (C4) combined with Lemma~\ref{lem:pWXtozero}. Furthermore, in this case,
condition (C4) can be replaced by irredundancy and $3$-partition regularity; see Lemma~\ref{lem:primpliesabundant}.
\end{remark}

If we set $S_n:=[n]$ and let $A$ be an irredundant partition regular matrix, then it is easy to check that (C1) and (C3)--(C5) hold. Theorem~\ref{thm:fgrsupersat} gives (C2), and therefore
Lemma~\ref{lem:easyRado} recovers the random Rado theorem (i.e., Theorems~\ref{radores0} and~\ref{r3}).\footnote{We formally verify this in the more general setting of integer lattices in the proof of Theorem~\ref{thm:ramseynd} in Section~\ref{sec:proofapplications}.}

Recall that Lemma~\ref{lem:randomRado}   confirms the probability threshold essentially depends  on the number of (projected) solutions to $Ax=0$ in $S_n$.
For sequences that satisfy the hypothesis of Lemma~\ref{lem:easyRado}, we see that we can express this phenomenon cleanly in terms of the parameter
$m_{S_n}(A)$. In particular,
if $(S_n)_{n \in \mathbb N}$
is a sequence as in Lemma~\ref{lem:easyRado} then,  for large $n \in \mathbb N$, a typical subset of $S_n$ of size significantly greater than $|S_n|^{1-1/m_{S_n}(A)}$ will be $(A,r)$-Rado, whilst a typical subset of size significantly smaller than $|S_n|^{1-1/m_{S_n}(A)}$ will not be $(A,r)$-Rado.

The proof of Lemma~\ref{lem:easyRado} follows easily from our original random Rado lemma, Lemma~\ref{lem:randomRado}, combined with the results from the previous subsection.

\begin{proof}[Proof of Lemma~\ref{lem:easyRado}]
Let $r\ge2$, $k\ge3$ and $\ell\geq 1$ be integers. Let $(S_n)_{n\in\mathbb N}$ be a sequence where $S_n$ is either a finite abelian group or $S_n=L^d$ for some finite subset $L$ of a field and $d\in\mathbb N$, for every $n\in\mathbb N$. Suppose that conditions (C1)--(C5) hold. First we show that conditions (A1)--(A6) hold too.

Conditions (A1) and (A2) are identical to (C1) and (C2) respectively. Condition (A3) follows from condition (C3) and Lemma~\ref{lem:simplifiedA3}. 

Pick arbitrary $W\subset Y\subseteq[k]$ with $|W|=1$. As (C3) and (C4) hold, we can apply Lemma~\ref{lem:spreadedges} and deduce that $|\Sol_{S_n}^A(w_0,W,Y)|\le |\Sol_{S_n}^A(Y)|/(\eps|S_n|)$ for every sufficiently large $n\in\mathbb N$ and every $w_0\in S_n$. This immediately implies condition (A4).

As (C3) and (C4) hold, we can apply Lemma~\ref{lem:simplifiedA4-5} to $A$ and $S_n$, for any sufficiently large $n\in\mathbb N$, to obtain the following: for any $W \subseteq [k]$ with $|W|=2$, we have
$$\frac{|S_n|^2}{|\Sol_{S_n}^A(W)|}\le\frac{1}{\eps}.$$
Combining the above with conditions (C1) and (C5) yields conditions (A6) and (A5) respectively.

\smallskip

As conditions (A1)--(A6) hold, by Lemma~\ref{lem:randomRado} there exist constants $c_0,C_0>0$ such that
\begin{equation}\label{eq:step1}
\lim _{n \rightarrow \infty} \mathbb P [ S_{n,p} \text{ is } (A,r)\text{-Rado}]=\begin{cases}
0 &\text{ if } p\leq {c_0}\cdot\hat p(A,S_n); \\
 1 &\text{ if } p\geq {C_0}\cdot\hat p(A,S_n).\end{cases}
 \end{equation}

Condition (C4) implies $m_{S_n}(A)$ is well-defined for every sufficiently large $n\in\mathbb N$ (see, e.g., Remark~\ref{rmk:m(a)abb}). Condition (C3) combined with Lemma~\ref{lem:pSA} implies that, for sufficiently large $n \in \mathbb N$,
\begin{equation}\label{eq:step2}
|S_n|^{-\frac{1}{m_{S_n}(A)}}\le\hat p(A,S_n)\le(1/\eps^2)|S_n|^{-\frac{1}{m_{S_n}(A)}}.
\end{equation}
Let $C:=(1/\eps^2)C_0$ and $c:=c_0$. Combining~\eqref{eq:step1} and~\eqref{eq:step2} yields
$$\lim _{n \rightarrow \infty} \mathbb P [ S_{n,p} \text{ is } (A,r)\text{-Rado}]=\begin{cases}
0 &\text{ if } p\leq {c}|S_n|^{-\frac{1}{m_{S_n}(A)}}; \\
 1 &\text{ if } p\geq {C}|S_n|^{-\frac{1}{m_{S_n}(A)}}.
 \end{cases}$$
\end{proof}

\section{The general $1$-statement and $0$-statement for hypergraphs}\label{sec:hypergraphs}

In this section, we formulate our general random Ramsey $1$-statement and $0$-statement for hypergraphs. We combine both statements in a single result (Theorem~\ref{thm:mainramsey}), but first we provide some definitions and motivation.

\subsection{Basic definitions for hypergraphs}

\begin{define}[Hypergraphs]
Let $k\in\mathbb N$. A {\it $k$-uniform hypergraph $\mathcal H$} is a pair $(V(\mathcal H),E(\mathcal H))$ where $V(\mathcal H)$ is a finite set and $E(\mathcal H)$ is a set of \emph{unordered} $k$-tuples \mbox{$\{v_1,\dots,v_k\} \in V(\mathcal H)^k$} with $v_i\ne v_j$ for $i\ne j$. A {\it $k$-uniform ordered hypergraph $\mathcal H$} is a pair $(V(\mathcal H),E(\mathcal H))$ where $V(\mathcal H)$ is a finite set and $E(\mathcal H)$ is a set of \emph{ordered} $k$-tuples $(v_1,\dots,v_k)\in V(\mathcal H)^k$ with $v_i\ne v_j$ for $i\ne j$.    
\end{define}

$V(\mathcal H)$ and $E(\mathcal H)$ are respectively the {\it vertex set} and the {\it edge set} of $\mathcal H$. We write $v(\mathcal H):=|V(\mathcal H)|$ and $e(\mathcal H):=|E(\mathcal H)|$. Next, we define the analogues of the notions of {\it partition regularity} and {\it supersaturation} for hypergraphs and ordered hypergraphs.

\begin{define}[$r$-Ramsey, $(r,\eps)$-supersaturated and $r$-supersaturated]\label{ramseydef}
Let $\mathcal H$ be an (ordered) hypergraph, $r\in\mathbb N$ and $\eps>0$. We say $\mathcal H$ is {\it $r$-Ramsey} if every $r$-colouring of $V(\mathcal H)$ yields a monochromatic edge in $E(\mathcal H)$. Similarly, $\mathcal H$ is {\it $(r,\eps)$-supersaturated} if every $r$-colouring of $V(\mathcal H)$ yields at least $\eps e(\mathcal H)$ monochromatic edges.
We say that a sequence $(\mathcal{H}_n)_{n \in \mathbb N}$ of (ordered) hypergraphs is \emph{$r$-supersaturated} if there exists $\eps>0$ and $n_0 \in \mathbb N$ such that for all $n \geq n_0$, $\mathcal{H}_n$ is $(r,\eps)$-supersaturated.
\end{define}

Before proceeding further, we give an informal explanation of how a Ramsey problem can be rephrased in terms of (ordered) hypergraphs. Say $\mathcal F$ is some discrete object (e.g., a graph, an abelian group etc.) over a ground set of elements. Given a colouring of the elements in $\mathcal F$, we are interested in finding a monochromatic copy of a certain structure $F$ in $\mathcal F$. We consider the (ordered) hypergraph $\mathcal H$ whose vertex set is the set of elements of $\mathcal F$ and the edge set is the set of copies of $F$ in $\mathcal F$. Clearly, the colouring of the elements of $\mathcal F$ induces a colouring of the vertices of $\mathcal H$. In particular, every monochromatic copy of $F$ in $\mathcal F$ corresponds to a monochromatic edge in $\mathcal H$ and vice versa. Consider the following example, where $\mathcal F:=[n]$ and $F$ is a $3$-distinct solution to the equation $x+y=z$. 

\begin{examp}\label{exp:hypergraph}
Any $2$-colouring of $[n]$ yields a monochromatic $3$-distinct solution to the equation $x+y=z$ as long as $n$ is sufficiently large ($A$). In fact, Theorem~\ref{thm:fgrsupersat} implies  that there exists $c>0$ so that if $n$ is sufficiently large, then  any $2$-colouring of $[n]$ yields at least $cn^2$ $3$-distinct
monochromatic solutions\footnote{While the statement of Theorem~\ref{thm:fgrsupersat}  does not insist that the solutions are $3$-distinct, since there are at most $O(n)$ solutions to $x+y=z$ which are not $3$-distinct, it is easy to obtain the result as stated here.}  ($B$).

Consider the $3$-uniform ordered hypergraph $\mathcal H$ with $V(\mathcal H)=[n]$ and $E(\mathcal H)=\{(x,y,z):x+y=z\}$. Statement $(A)$ is equivalent to $\mathcal H$ being $2$-Ramsey as long as $v(\mathcal H)$ is sufficiently large while statement ($B$) is equivalent to $\mathcal H$ being $(2,c')$-supersaturated for some constant $c'>0$, provided  $v(\mathcal H)$ is sufficiently large.
\end{examp}
\smallskip

In practice, we can usually work in the setting of \emph{unordered} hypergraphs.

\begin{define}
Given an ordered hypergraph $\mathcal H$, the {\it unordered restriction of $\mathcal H$} is the hypergraph $\mathcal H'$ obtained from $\mathcal H$ by ignoring the order of the edges in $E(\mathcal H)$. Formally, $V(\mathcal H'):=V(\mathcal H)$ and
$$E(\mathcal H'):=\{\{v_1,\dots,v_k\}:(v_1,\dots,v_k) \in E(\mathcal H)\}.$$
\end{define}
Note that we do not allow multiple edges here; that is, if two ordered edges $e_1, e_2$ become the same unordered edge $e$, then we only keep one copy of $e$.

\begin{remark}\label{remark:restriction}
By definition, a $k$-uniform ordered hypergraph $\mathcal H$ is $r$-Ramsey if and only if the unordered restriction $\mathcal H'$ of $\mathcal H$ is also $r$-Ramsey. Similarly, a sequence of $k$-uniform ordered hypergraphs $(\mathcal H_n)_{n\in\mathbb N}$ is $r$-supersaturated if and only if the sequence $(\mathcal H'_n)_{n\in\mathbb N}$ is $r$-supersaturated, where $\mathcal H'_n$ is the unordered restriction of $\mathcal H_n$.
\end{remark}

Next, we define  notions of maximum degree for (ordered) hypergraphs. Let $\mathcal H$ be a $k$-uniform ordered hypergraph and let $W\subseteq[k]$, say $W=\{i_1,i_2,\dots,i_{|W|}\}$ where $i_1<i_2<\dots <i_{|W|}$. The \emph{restriction} $\mathcal H_W$ of $\mathcal H$ is the $|W|$-uniform ordered hypergraph with $V(\mathcal H_W):=V(\mathcal H)$ and
$$E(\mathcal H_W):=\{(x_{i_1},\dots,x_{i_{|W|}}):\exists\,(y_1,\dots,y_k)\in E(\mathcal H)\text{ with $y_{i_s}=x_{i_s}$ for each $s\in\{1,\dots,|W|\}$}\}.$$
Again, we do not allow multiple edges.
Given sets $W\subseteq Y\subseteq [k]$, 
we define $\Delta_W(\mathcal{H}_Y)$ to be the maximum number of edges in $\mathcal{H}_Y$ that restrict to the same edge in $\mathcal{H}_W$. More precisely, 
given any ordered tuple $f \in E(\mathcal{H}_W)$,
let $E_{\mathcal{H}}(f,Y)$ be the set of ordered tuples
$e \in E(\mathcal{H}_Y)$
so that $f$ equals the ordered  tuple obtained from $e$ by  deleting  the entries of $e$ in the positions corresponding to $Y\setminus W$.
Then we define
$$\Delta_W(\mathcal{H}_Y):= \max_{f \in E(\mathcal{H}_W)} | E_{\mathcal{H}}(f,Y)
|.$$
Let $\mathcal H'$ be a $k$-uniform hypergraph. For $i\in[k]$, we set
$$\Delta_i(\mathcal{H'}):=\max_{f\subseteq V(\mathcal H'), \, |f|=i}|\{e\in E(\mathcal H'):f\subseteq e\}|.$$
Note that if $\mathcal H'$ is the unordered restriction of an ordered hypergraph $\mathcal H$ then
\begin{align}\label{eq:unordered}
\frac{1}{k!} \cdot \max_{\stackrel{W \subseteq [k]}{|W|=i}} \Delta_W(\mathcal{H}) \leq  
\frac{1}{(k-i)!} \cdot \max_{\stackrel{W \subseteq [k]}{|W|=i}} \Delta_W(\mathcal{H})
\leq
\Delta_i(\mathcal{H'}) \leq 
 \binom{k}{i} \cdot i! \cdot \max_{\stackrel{W \subseteq [k]}{|W|=i}} \Delta_W(\mathcal{H})
\leq
k! \cdot \max_{\stackrel{W \subseteq [k]}{|W|=i}} \Delta_W(\mathcal{H}).
\end{align}
The second inequality in (\ref{eq:unordered}) uses that, given any edge $e$ in $\mathcal H'$,
there are $(k-i)!$ ordered $k$-tuples that have $e$ as their  underlying (unordered) set  and have the entries in $i$ of their positions fixed.

For an (ordered) hypergraph $\mathcal{H}$,
we write $\mathcal{H}_p^v$ to denote the hypergraph obtained from $\mathcal{H}$ by including each \emph{vertex} of $\mathcal{H}$ with probability $p$ independently of all other vertices, while including each edge of $\mathcal{H}$ if and only if all of its vertices are kept.
We use the superscript $v$ to remind the reader that it is the vertices that we keep with probability $p$ rather than edges (as in the binomial random hypergraph).

\subsection{Heuristic for the probability threshold}\label{sec:heur}
The random Ramsey-type results we covered in 
Section~\ref{sec:primes} (for the primes), in Section~\ref{sec:graphsintegers} (for graphs and integers) and in Section~\ref{sec:easyapplications} (for  abelian groups) are formulated as follows: given a sequence $(\mathcal F_n)_{n \in \mathbb N}$ of objects that exhibit a Ramsey property in a robust way (i.e., supersaturation), there is probability threshold $\hat p=\hat p(n)$ such that, as $n$ tends to infinity, the random binomial subset $\mathcal F_{n,p}$ retains the Ramsey property w.h.p. if $p=\omega(\hat p)$ and loses it w.h.p. if $p=o(\hat p)$. 
Our general $1$-statement and $0$-statement for hypergraphs are formulated in a similar fashion. 

\begin{define}[Probability threshold for hypergraphs]\label{defpt}
Given a sequence $(\mathcal H_n)_{n\in\mathbb N}$ of $k$-uniform ordered hypergraphs and $W \subseteq [k]$, let
\begin{align}\label{eq:keyprobability}
f_{n,W}:= \left(\frac{e(\mathcal{H}_{n,W})}{v(\mathcal{H}_{n})}\right)^{-\frac{1}{|W|-1}} 
\quad \text{and} \quad
\hat p(\mathcal{H}_n) := \max_{\stackrel{W \subseteq [k]}{|W| \geq 2}} f_{n,W}.
\end{align}
\end{define}
Note that in Definition~\ref{defpt} we slightly abuse notation, writing $\mathcal{H}_{n,W}$ for the restriction $\mathcal H_{n \, W}$ of $\mathcal H_n$.

Our choice of $\hat p(\mathcal{H}_n)$ follows precisely the same heuristic as the probability threshold for the random Ramsey problem for graphs and the integers (see Section~\ref{sec:graphsintegers}). 
Namely, consider the expected number of vertices and edges in the random hypergraph $\mathcal{H}^v_{n,p}$. 
Assuming that the edges in $\mathcal{H}^v_{n,p}$ are distributed uniformly in some way, if $\mathbb E(e(\mathcal{H}^v_{n,p}))$ is much less than $\mathbb E(v(\mathcal{H}^v_{n,p}))$ then intuitively the edges in $\mathcal{H}^v_{n,p}$ should be fairly spread out. Therefore, one should be able to $2$-colour the vertices of $\mathcal{H}^v_{n,p}$ without creating monochromatic edges. 

However, similarly to the random Ramsey problem for graphs and the integers, it may be the case that the above reasoning is applicable to a `lower' level in the following sense. 
For any $W\subseteq [k]$, a $2$-colouring that does not yield monochromatic edges in $\mathcal H^v_{n,W,p}$ also does not yield monochromatic edges in $\mathcal H^v_{n,p}$. By the previous reasoning, if $\mathbb{E}(e(\mathcal H^v_{n,W,p}))$ is much less than $\mathbb E(v(\mathcal H^v_{n,W,p}))$ then one should be able to $2$-colour the vertices of $\mathcal H^v_{n,W,p}$ while avoiding any monochromatic edge in $\mathcal H^v_{n,W,p}$, and therefore in  $\mathcal H^v_{n,p}$. For any $W\subseteq[k]$ with $|W|\geq 2$, we have
$$\mathbb E(v(\mathcal H^v_{n,W,p}))=pv(\mathcal H_{n})\qquad \text{and} \qquad \mathbb E(e(\mathcal H^v_{n,W,p}))=p^{|W|}e(\mathcal H_{n,W}).$$
In particular, $\mathbb E(e(\mathcal H^v_{n,W,p}))\leq \mathbb E(v(\mathcal H^v_{n,W,p}))$ if and only if 
$$p\leq \left(\frac{e(\mathcal H_{n,W})}{v(\mathcal H_{n})}\right)^{-\frac{1}{|W|-1}}=f_{n,W}.$$
It is then natural to let $\hat p(\mathcal{H}_n)$ be equal to the largest of the $f_{n,W}$'s.

\subsection{General $1$- and $0$-statements for hypergraphs}

We are now ready to state our general $1$-statement and $0$-statement for hypergraphs.

\begin{thm}\label{thm:mainramsey}
Let  $r,k \geq 2$ be integers, and let $b>0$.
Let $(\mathcal{H}_n)_{n \in \mathbb{N}}$ be a sequence of $k$-uniform ordered hypergraphs and set $\hat p_n:=\hat p(\mathcal{H}_n)$ as given by~(\ref{eq:keyprobability}). Let (P1)--(P5) be the following properties. 
\begin{enumerate}
\item[(P1)] Probabilities: $\hat p_n \to 0$ and $\hat p_n v(\mathcal{H}_n) \to \infty$ as $n \to \infty$.
\item[(P2)] Supersaturation: $(\mathcal{H}_n)_{n \in \mathbb{N}}$ is $r$-supersaturated.
\item[(P3)] Bounded degree: 
for any $W \subseteq Y \subseteq [k]$, and any sufficiently large $n \in \mathbb N$  we have
$\Delta_W(\mathcal{H}_{n,Y}) \leq b \frac{e(\mathcal{H}_{n,Y})}{e(\mathcal{H}_{n,W})}$. 

\item[(P4)] Bounded $1$-degree: for any $W \subset Y \subseteq [k]$ with $|W|=1$ and any sufficiently large $n \in \mathbb N$, we have
$\Delta_W(\mathcal{H}_{n,Y}) \leq b \frac{e(\mathcal{H}_{n,Y})}{v(\mathcal{H}_{n})}$.
\item[(P5)] Bounded $2$-degree: there exists a sequence $(X_n)_{n\in\mathbb N}$ of subsets $X_n\subseteq[k]$ with $|X_n|\ge3$ for all $n\in\mathbb N$ such that
\begin{equation}\label{eq:P4a}
f_{n,X_n}=\Omega(\hat p_n)
\end{equation}
and 
\begin{equation}\label{eq:P4b}
\Delta_{W_n}(\mathcal{H}_{n,X_n}) \cdot \Delta_{W'_n}(\mathcal{H}_{n,X_n})\cdot e(\mathcal{H}_{n,X_n}) \cdot (\hat p_n)^{3|X_n|-2-|W'_n|} = o(1)
\end{equation} 
for any sequences $(W_n)_{n\in\mathbb N}$ and $(W'_n)_{n\in\mathbb N}$ with $W_n,W'_n\subset X_n$, $|W_n|=2$ and $|W'_n|\ge2$.
\end{enumerate}

\medskip

If (P1)--(P4) hold, then there exists a constant $C>0$ such that if $q_n \geq C\hat p_n$ then
\begin{align*}
\lim_{n \rightarrow \infty}  \mathbb{P}[\mathcal{H}^v_{n,q_n} \text{ is $r$-Ramsey}]
= 1.
\end{align*}

If (P1) and (P3)--(P5) hold, then there exists a constant $c>0$ such that if $q_n \leq c\hat p_n$ then
\begin{align*}
\lim_{n \rightarrow \infty}  \mathbb{P}[\mathcal{H}^v_{n,q_n} \text{ is $r$-Ramsey}]
= 0.
\end{align*}
\end{thm}
All of our new random Ramsey-type results are essentially corollaries of Theorem~\ref{thm:mainramsey}. Indeed, we use Theorem~\ref{thm:mainramsey} to prove the random Rado lemma (Lemma~\ref{lem:randomRado}) which we have in turn used to prove the simplified random Rado lemma (Lemma~\ref{lem:easyRado}). Further, the proofs of our new results for the primes (Theorem~\ref{rvdwp}), integer lattices (Theorem~\ref{thm:ramseynd}), groups (Theorem~\ref{thm:exponent}) and finite vector spaces 
(Theorems~\ref{thm:fields} and~\ref{thm:fields2})
all apply one of these two random Rado lemmas (see Section~\ref{sec:proofapplications}).
This in turn implies that the random Rado theorem
(Theorems~\ref{radores0} and~\ref{r3}) is a consequence of Theorem~\ref{thm:mainramsey}.
As we  describe in Appendix~\ref{seca2}, the $1$-statement of the random Ramsey theorem (Theorem~\ref{randomramsey})
can also be deduced from Theorem~\ref{thm:mainramsey}.

The reader should therefore take away the following philosophy.
Theorem~\ref{thm:mainramsey} is the most general of our black box results, and its applications are not just restricted to systems of linear equations. 
On the other hand, the random Rado lemma is more readily applicable to the setting of Ramsey properties in arithmetic structures. 
The simplified random Rado lemma has the most restricted hypothesis of these three results, but has the advantage that it explicitly produces the corresponding probability threshold.

\smallskip

The proof of Theorem~\ref{thm:mainramsey} is given in Section~\ref{sec:proofmainramsey}.
The conditions required for the $1$-statement to hold, namely (P1)--(P4), are  natural. 
Condition (P1) is needed to ensure the expected value of $v(\mathcal H^v_{n,p})$ tends to infinity. In Section~\ref{subsec:supersat}, we already discussed the notion of supersaturation (condition (P2)) and how it is typically required in random Ramsey-type results.
(P3) and (P4) are recurring assumptions when applying the hypergraph container method, which 
 we indeed use to prove Theorem~\ref{thm:mainramsey}. 
They are also  natural conditions to make to ensure our heuristic described in Section~\ref{sec:heur} is valid. Indeed, (P3) and (P4) ensure that edges are fairly spread out in $\mathcal H_n$ and its associated restrictions 
 $\mathcal H_{n,Y}$.

 On the other hand, condition (P5) is more artificial: it arises naturally in the setting of abelian groups but fails to hold in general. Indeed, while the $0$-statement of the random Rado lemma (Lemma~\ref{lem:randomRado}) follows from Theorem~\ref{thm:mainramsey}, the $0$-statement for random Ramsey theorem (Theorem~\ref{randomramsey}) does not. 
In Section~\ref{sec:proofmainramsey} we explain how the $0$-statement of Theorem~\ref{thm:mainramsey} follows from the combination of a {\it deterministic lemma}, stating that certain substructures must arise in the hypergraph assuming it is $r$-Ramsey, and a {\it probabilistic lemma}, stating that such structures do not appear w.h.p. The proof of the latter consists of bounding the expected number of the forbidden structures so that the expectation tends to $0$ as $n$ goes to infinity. Condition (P5) is precisely designed to achieve these bounds.  

Returning to Theorem~\ref{randomramsey}, 
its $0$-statement is proved by using a much stronger deterministic lemma than what we use here. 
In particular, this statement says that graphs with density at most $m_2(H)$ can be $2$-coloured so to avoid a monochromatic copy of $H$. 
Using this strong result, 
one can prove a suitable corresponding probabilistic lemma without needing (P5) to be satisfied.

\smallskip

Note that the use of ordered hypergraphs in Theorem~\ref{thm:mainramsey} is only really necessary for the $0$-statement.
Indeed, we will deduce the $1$-statement of Theorem~\ref{thm:mainramsey} very quickly from a theorem for unordered hypergraphs, Theorem~\ref{thm:general1statement}. The proof of Theorem~\ref{thm:general1statement} is obtained via the method of containers, 
which was first formalised in the setting of hypergraphs by Balogh, Morris and Samotij~\cite{bms} and independently by Saxton and Thomason~\cite{st}. 
The use of containers is now standard in proofs of $1$-statements for various Ramsey-type problems, see, e.g.,~\cite{AP,bms,bckmn,sharp,fkssnew,asymmramsey,hst,mns,ns,sp}.
We emphasise that the proof of Theorem~\ref{thm:general1statement} is not particularly difficult. 
Indeed, a researcher familiar with the container method would be  able to come up with such a proof,
and in particular, our proof here is an easy adaption of the short proof of the $1$-statement of Theorem~\ref{randomramsey} given by Nenadov and Steger~\cite{ns}.
The proof of the $0$-statement of Theorem~\ref{thm:mainramsey}  adapts  a recent argument of the second and third authors~\cite{ht}.

\smallskip

There are several results in the literature of a similar flavour to Theorem~\ref{thm:mainramsey}. Indeed,
Friedgut, R\"odl and Schacht proved a similar  result to our $1$-statement;  see~\cite[Theorem 2.5]{frs}.
In their hypothesis though, they require 
the sequence $(\mathcal{H}_n)_{n \in \mathbb{N}}$ to be $R$-supersaturated for some $R$ much larger than $r$.
Furthermore, Aigner-Horev and Person~\cite[Theorem 2.8]{AP}  proved an \emph{asymmetric} version of the $1$-statement of Theorem~\ref{thm:mainramsey}, 
where one considers a sequence of $r$-tuples of hypergraphs $(\mathcal{H}_{n,1},\dots,\mathcal{H}_{n,r})_{n \in \mathbb{N}}$ on the same vertex set $V_n$ and ask for the probability threshold such that whenever the vertices are $r$-coloured, there is some $i \in [r]$ such that there is an edge in $ \mathcal{H}^v_{n,i,p}$ which is monochromatic in colour $i$.
Note that their result requires some rather technical conditions on the $\mathcal{H}_{n,i}$ which are not necessary in the symmetric setting.

Finally, in their recent paper on sharp Ramsey thresholds, Friedgut, Kuperwasser, Samotij and Schacht~\cite[Theorem 2.1]{fkssnew} gave general criteria for a sequence of hypergraphs $(\mathcal{H}_n)_{n \in \mathbb{N}}$ to have a sharp threshold for being $r$-Ramsey. 
The threshold function in their theorem is $\Theta(f_{n,[k]})$, which is the value of our $\hat p_n$ in the case where $W=[k]$ is the maximiser of the $f_{n,W}$'s: note that this is indeed the case when considering the Schur property or sets avoiding monochromatic $k$-APs in $\mathbb{Z}_n$; as mentioned at the end of Section~\ref{sec:graphsintegers}, these are two of the specific settings for which sharp thresholds are obtained in~\cite{fkssnew}.

\section{Proof of Theorem~\ref{thm:mainramsey}}\label{sec:proofmainramsey}

\subsection{Proof of the $1$-statement of Theorem~\ref{thm:mainramsey}}
In this subsection we state a general $1$-statement, Theorem~\ref{thm:general1statement} for \emph{hypergraphs}. 
We then show that the $1$-statement for ordered hypergraphs (Theorem~\ref{thm:mainramsey}) follows immediately from it. We prove Theorem~\ref{thm:general1statement} in Section~\ref{genproof}.

\subsubsection{The Ramsey property in  supersaturated hypergraphs}\label{subsec:1-statement-unordered}

\begin{thm}\label{thm:general1statement}
Let $r,k \geq 2$ be integers and let $c>0$. 
Suppose that $(\mathcal{H}_n)_{n \in \mathbb{N}}$ is a sequence of $k$-uniform hypergraphs. Let $(p_n)_{n\in\mathbb N}$ with $p_n\in(0,1)$ be such that $p_n \to 0$ and $p_nv(\mathcal{H}_n) \to \infty$ as $n \to \infty$,
and for each sufficiently large $n \in \mathbb{N}$ and each $\ell \in [k]$, we have
$$\Delta_{\ell}(\mathcal{H}_n) \leq c (p_n)^{\ell-1} \frac{e(\mathcal{H}_n)}{v(\mathcal{H}_n)}.$$
Further, suppose that $(\mathcal{H}_n)_{n \in \mathbb{N}}$ is $r$-supersaturated.
Then there exists $C \geq 1$ such that for $q_n \geq C p_n$,  
$$ \lim_{n \to \infty} \mathbb{P} [\mathcal{H}^v_{n,q_n} \text{ is $r$-Ramsey}]=1.$$
\end{thm}

The proof of Theorem~\ref{thm:general1statement} is in the same spirit as that of the proof of the $1$-statement of Theorem~\ref{randomramsey} given by Nenadov and Steger in~\cite{ns}.
In particular, as mentioned in Section~\ref{sec:hypergraphs}, it utilises the hypergraph container method. 

For a set $\mathcal{A}$, let $\mathcal{P}(\mathcal{A})$ denote the powerset of $\mathcal{A}$. 
For a hypergraph $\mathcal{H}$, let $\mathcal{I}_r(\mathcal{H}) $ be the collection of $r$-tuples  $(I_1,\dots,I_r)$ such that each $I_i\in  \mathcal{P}(V(\mathcal{H}))$ is an independent set in $\mathcal{H}$ and $I_i\cap I_j=\emptyset$ for every $i \ne j\in[r]$.
For a sequence of hypergraphs $(\mathcal{H}_n)_{n \in \mathbb{N}}$,
to show that w.h.p.\ $\mathcal{H}^v_{n,p}$ is $r$-Ramsey, 
it is equivalent to show that with probability tending to $0$ the vertex set of $\mathcal{H}^v_{n,p}$ can be partitioned into $r$ parts such that each part induces an independent set, i.e., it admits an $r$-partition $I \in \mathcal{I}_r(\mathcal{H}_n)$.
Bounding this probability by using a union bound over all $r$-tuples in $\mathcal{I}_r(\mathcal{H}_n)$ does not work because $|\mathcal{I}_r(\mathcal{H}_n)|$ is too large.
The power of a hypergraph container result  is that one can instead use a union bound over all `containers', since informally such a result says that every independent set lies in a container, and there are not too many containers overall.

\subsubsection{Deducing the $1$-statement of Theorem~\ref{thm:mainramsey} from Theorem~\ref{thm:general1statement}}

\begin{proof}[Proof of the $1$-statement of Theorem~\ref{thm:mainramsey}]
Let $(\mathcal{H}_n)_{n \in \mathbb{N}}$ be a sequence of $k$-uniform ordered hypergraphs satisfying  (P1)--(P4).
For all $W \subseteq [k]$ with $|W| \geq 2$ we have 
\begin{align}\label{eq:m1}
\hat p_n= \max_{\stackrel{X \subseteq [n]}{|X| \geq 2}} \left( \frac{e(\mathcal{H}_{n,X})}{v(\mathcal{H}_{n})} \right)^{\frac{-1}{|X|-1}} & \geq  \left( \frac{e(\mathcal{H}_{n,W})}{v(\mathcal{H}_{n})} \right)^{\frac{-1}{|W|-1}}; \nonumber \\
\implies e(\mathcal{H}_{n,W}) & \geq (\hat p_n)^{-(|W|-1)} \cdot v(\mathcal{H}_n)  .
\end{align}
Therefore, for such $W \subseteq [k]$ with $|W| \geq 2$, and each $n \in \mathbb N$ sufficiently large, we have
\begin{align}\label{eq:m3}
\Delta_W(\mathcal{H}_n)   \stackrel{\text{(P3)}}{\leq} b \frac{e(\mathcal{H}_n)}{e(\mathcal{H}_{n,W})}   \stackrel{(\ref{eq:m1})}{\leq} b \cdot (\hat p_n)^{|W|-1} \cdot \frac{e(\mathcal{H}_n)}{v(\mathcal{H}_n)}.
\end{align}
For any $W \subseteq [k]$ with $|W| =1$,
 and each $n \in \mathbb N$ sufficiently large,
we have
\begin{align}\label{eq:m3new}
\Delta_W(\mathcal{H}_n)   \stackrel{\text{(P4)}}{\leq} b \frac{e(\mathcal{H}_n)}{v(\mathcal{H}_{n})}  
= b \cdot (\hat p_n)^{|W|-1} \cdot \frac{e(\mathcal{H}_n)}{v(\mathcal{H}_n)}.
\end{align}

For each $n \in \mathbb{N}$, let $\mathcal{H}'_n$ be 
the unordered restriction of $\mathcal{H}_n$. We now 
check the hypothesis of 
Theorem~\ref{thm:general1statement} with respect to 
$(\mathcal H'_n)_{n\in\mathbb N}$ and $(p_n)_{n\in\mathbb N}$ where $p_n:=\hat p_n$. 
(P1) implies $p_n \to 0$ and $p_nv(\mathcal{H}'_n) \to \infty$ as $n \to \infty$.
Remark~\ref{remark:restriction} and (P2) imply $(\mathcal H'_n)_{n\in\mathbb N}$ is $r$-supersaturated.   Finally, equations~(\ref{eq:m3}), (\ref{eq:m3new}) and~(\ref{eq:unordered}) yield the required upper bound on $\Delta_\ell(\mathcal H'_n)$ where $c:=b \cdot k!$.

Now, by Remark~\ref{remark:restriction}, the conclusion of Theorem~\ref{thm:general1statement} implies the $1$-statement of Theorem~\ref{thm:mainramsey}.
\end{proof}

\subsubsection{Proof of Theorem~\ref{thm:general1statement}}\label{genproof}

First, we introduce some notation. Let $\mathcal{H}$ be a hypergraph and let $\mathcal{F}\subseteq\mathcal P(V(\mathcal H))$. We say $\mathcal F$ is {\it increasing} if $A \in \mathcal F$ and $A\subseteq B\subseteq V(\mathcal H)$ imply $B \in \mathcal F$. We write $\overline{\mathcal F}$ to denote the complement family of $\mathcal F$, i.e., the family of sets $A \in \mathcal{P}(V(\mathcal F))$ which are not in $\mathcal F$. Finally, for  $\eps \in (0,1]$, we say that $\mathcal{H}$ is $(\mathcal{F},\eps)$-dense if $e(\mathcal{H}[A]) \geq \eps e(\mathcal{H})$ for every $A \in \mathcal{F}$.

For this proof, when we consider $r$-tuples of sets, the $r$-tuples are \emph{always} ordered. 
For two $r$-tuples of sets $I=(I_1,\dots,I_r)$ and $J=(J_1,\dots,J_r)$ we write $I \subseteq J$ if $I_i \subseteq J_i$ for all $i \in [r]$;
we define $I \cup J:=(I_1 \cup J_1, \dots, I_r \cup J_r)$.
If $\mathcal{X}$ is a collection of sets, then we write $\mathcal{X}^r$ for the collection of $r$-tuples $(X_1,\dots,X_r)$ so that $X_i \in \mathcal{X}$ for all $i \in [r]$. 

We now state the following container result from~\cite{hst}, from which Theorem~\ref{thm:general1statement} will follow.
Note that Proposition~\ref{prop:rcontainers} itself  follows from the general hypergraph container theorem of Balogh, Morris and Samotij~\cite[Theorem~2.2]{bms}.

\begin{prop}
\cite[Proposition 3.2]{hst}\label{prop:rcontainers}
For every $k,r \in \mathbb N$ and all $c,\eps>0$, 
there exists $D>0$ such that the following holds. 
Let $\mathcal{H}$ be a $k$-uniform hypergraph and let $\mathcal{F} \subseteq \mathcal{P}(V (\mathcal{H}))$ be an increasing family of sets 
such that $|A| \geq \eps v(\mathcal{H})$ for all $A \in \mathcal{F}$. 
Suppose that $\mathcal{H}$ is $(\mathcal{F},\eps)$-dense and
$p \in (0,1)$ is such that, for every $\ell \in [k]$,
\begin{align*}
\Delta_{\ell}(\mathcal{H}) \leq c \cdot p^{\ell-1} \frac{ e(\mathcal{H})}{v(\mathcal{H})}.
 \end{align*}
Then there exists a family $\mathcal{S}_r \subseteq \mathcal{I}_r(\mathcal{H})$ and functions $f : \mathcal{S}_r \to (\overline{\mathcal{F}})^r$ and $g : \mathcal{I}_r(\mathcal{H}) \to \mathcal{S}_r$ such that
the following conditions hold:
\begin{enumerate}
\item If $(S_1, \dots, S_r) \in \mathcal{S}_r$ then $\sum^ r_{i=1} |S_i| \leq Dp \cdot v(\mathcal{H})$;
\item for every $(I_1,\dots,I_r) \in \mathcal{I}_r(\mathcal{H})$, we have that $S \subseteq (I_1,\dots,I_r) \subseteq S \cup f(S)$ where $S:=g(I_1,\dots,I_r)$.
\end{enumerate}
\end{prop}

\begin{proof}[Proof of Theorem~\ref{thm:general1statement}]
Note that it suffices to prove the theorem under the additional assumption that $c \geq k$.
Pick $\eps>0$ so that, for all sufficiently large
$n\in \mathbb N$, $\mathcal{H}_n$ is $(r,\eps)$-supersaturated.
Define
\begin{align}\label{eq:const}
\eps':=\frac{\eps}{3}, \quad \eps'':=\frac{\eps'k}{c} \quad \text{and} \quad \delta:=\frac{\eps}{3c}. 
\end{align}
As $c \geq k$ we have  that $\eps'' \leq \eps'$. 
Apply Proposition~\ref{prop:rcontainers} with  $\eps''$ playing the role of $\eps$ to obtain $D>0$. 
Since $p_n \to 0$ as $n \to \infty$, if $n$ is sufficiently large then 
\begin{align}\label{eq:nlarge}
c D p_n \leq \eps'.
\end{align}
Next, we pick a sufficiently large constant $C>0$ so that
$(\dagger)$ in~(\ref{eq:concalc}) below holds for all  sufficiently large $n\in \mathbb N$.

Let $n\in \mathbb N$ be sufficiently large and set
 $\mathcal{F}_n:= \{ F \subseteq V(\mathcal{H}_n) :  e(\mathcal{H}_n[F]) \geq \eps' e(\mathcal{H}_n) \}$.
Clearly $\mathcal{F}_n$ is increasing and $\mathcal{H}_n$ is $(\mathcal{F}_n,\eps')$-dense, and thus 
 also $(\mathcal{F}_n,\eps'')$-dense. 
Further, for all $F \in \mathcal{F}_n$, we have\footnote{Here $\deg_{\mathcal{H}_n[F]}(y)$ 
denotes the number of edges $y$ lies in within 
$\mathcal{H}_n[F]$.}
\begin{align*}
k \eps' e(\mathcal{H}_n) \leq k e(\mathcal{H}_n[F]) = \sum_{y \in F} \deg_{\mathcal{H}_n[F]}(y)\le|F|\Delta_1(\mathcal H_n)\le\frac{|F| c \cdot e(\mathcal{H}_n)}{v(\mathcal{H}_n)},
\end{align*}
where the last inequality uses the bound on $\Delta_1(\mathcal{H}_n)$ in the statement of Theorem~\ref{thm:general1statement}.
Rearranging, we obtain
\begin{align*}
|F| \geq \frac{\eps' k}{c} v(\mathcal{H}_n) \stackrel{(\ref{eq:const})}{=} \eps'' v(\mathcal{H}_n),
\end{align*}
for all $F \in \mathcal{F}_n$. Thus $\mathcal{H}_n$ and $\mathcal{F}_n$ satisfy the hypothesis of Proposition~\ref{prop:rcontainers} with $\eps''$ and $p_n$ playing the roles of $\eps$ and $p$ respectively. Apply Proposition~\ref{prop:rcontainers} to obtain the family 
$\mathcal{S}_r \subseteq \mathcal{I}_r(\mathcal{H}_n)$
and functions  $f : \mathcal{S}_r \to (\overline{\mathcal{F}}_n)^r$ and $g : \mathcal{I}_r(\mathcal{H}_n) \to \mathcal{S}_r$.

Now let $q_n\ge C p_n$ and let $B_{q_n}$ be the event that $\mathcal{H}^v_{n,q_n}$ is not $r$-Ramsey. 
It suffices to show that $\mathbb{P}[B_{q_n}] \to 0$ as $n \to \infty$. Since $q_n\le q'_n\le1$ implies $\mathbb{P}[B_{q'_n}]\le\mathbb{P}[B_{q_n}]$, we may assume that $q_n=Cp_n$.

Suppose that $\mathcal{H}^v_{n,q_n}$ is not $r$-Ramsey, i.e., there exists a partition $V(\mathcal{H}^v_{n,q_n})=X_1\cup\dots\cup X_r$ with $(X_1,\dots,X_r) \in \mathcal{I}_r(\mathcal{H}_n)$.
By Proposition~\ref{prop:rcontainers}, there exists 
$S:=(S_1,\dots,S_r):=g(X_1,\dots,X_r) \in \mathcal{S}_r$ such that $S \subseteq (X_1,\dots,X_r) \subseteq S \cup f(S)$ and 
$|\cup_{i \in [r]}S_i| \leq D p_n v(\mathcal{H}_n)$. For brevity, we write $f(S)=:(f(S_1),\dots,f(S_r))$, $X(S):=\cup_{i \in [r]} S_i$, $Y(S):=\cup_{i \in [r]} f(S_i)$ and $Z(S):=X(S) \cup Y(S)$.\footnote{Note the slight abuse of the $f$ notation here.} 

\begin{claim}\label{claim:Zbound}
We have
$
|Z(S)|\le(1-\delta)v(\mathcal{H}_n).
$
\end{claim}
\begin{proofclaim}
Since $\mathcal{H}_n$ is $(r,\eps)$-supersaturated, in any partition $V(\mathcal{H}_n)=V_1\cup\dots\cup V_r$ there exists $i \in [r]$ such that $e(\mathcal H_n [V_i] )\geq \eps e(\mathcal{H}_n)$.
Any vertex in $\mathcal H_n$ is contained in at most $\Delta_1(\mathcal{H}_n)$ edges, 
and thus by deleting a set $U$ of fewer than $\delta v(\mathcal{H}_n)$ vertices, 
we delete at most $|U| \Delta_1(\mathcal{H}_n) < \delta c e(\mathcal{H}_n)$ edges.
Thus, given any induced subhypergraph of $\mathcal{H}_n$ with more than $(1-\delta)v(\mathcal{H}_n)$ vertices, 
for any partition $W_1\cup\dots\cup W_r$ of it,
there exists $i \in [r]$ such that $e(\mathcal{H}_n[W_i])>(\eps-\delta c) e(\mathcal{H}_n) = 2\eps' e(\mathcal{H}_n)$ edges (where the last equality follows from~\eqref{eq:const}).

Now we bound $e(\mathcal{H}_n[S_i \cup f(S_i)])$ for each $i \in [r]$.
Since $f(S_i) \in \overline{\mathcal{F}}_n$ we have $e(\mathcal{H}_n[f(S_i)]) < \eps' e(\mathcal{H}_n)$.
The number of edges in $\mathcal{H}_n[S_i \cup f(S_i)]$ containing at least one vertex from $S_i$ is at most 
\begin{align*}
|S_i| \cdot \Delta_1(\mathcal{H}_n) \leq D p_n v(\mathcal{H}_n) \cdot c \frac{e(\mathcal{H}_n)}{v(\mathcal{H}_n)} 
\stackrel{(\ref{eq:nlarge})}{\leq} \eps' e(\mathcal{H}_n).
\end{align*}
Thus, in total we have $e(\mathcal{H}_n[S_i \cup f(S_i)]) < 2\eps' e(\mathcal{H}_n)$.
Overall we see that the set $Z(S)$ can be partitioned into $r$ sets so that each set induces fewer than $2\eps'e(\mathcal{H}_n)$ edges in $\mathcal{H}_n$.
Combining with the conclusion of the last paragraph, this implies that $|Z(S)|\le(1-\delta)v(\mathcal H_n)$ as required.\qedclaim
\end{proofclaim}

As $X_1\cup\dots\cup X_r$ is a partition of $V(\mathcal{H}^v_{n,q_n})$, we must have that $X(S)\subseteq V(\mathcal{H}^v_{n,q_n})\subseteq Z(S)$. Crucially, this event depends uniquely on $S$ and thus
$$\mathbb P[B_{q_n}]\le\sum_{S\in\mathcal S_r}\mathbb P[X(S)\subseteq V(\mathcal{H}^v_{n,q_n})\subseteq Z(S)]\le\sum_{S\in\mathcal S_r} q_n^{|X(S)|} (1-q_n)^{\delta v(\mathcal{H}_n)} \leq\sum_{S\in\mathcal S_r} q_n^{|X(S)|} e^{-\delta q_n v(\mathcal{H}_n)},$$
where the first inequality follows from the union bound, while the second inequality follows from Claim~\ref{claim:Zbound}.

Recall that for every $S=(S_1,\dots,S_r)\in \mathcal{S}_r$ we have 
$|\cup_{i \in [r]}S_i| \leq D p_n v(\mathcal{H}_n)$. 
We may bound $|\mathcal S_r|$ from above as follows: 
choose a set of $i\le Dp_nv(\mathcal{H}_n)$ vertices 
from $V(\mathcal{H}_n)$ and then take an (ordered) partition of 
these vertices into $r$ classes. Therefore we have,
\begin{align}\label{eq:concalc}
\mathbb{P}[B_{q_n}] \leq \sum_{S \in \mathcal{S}_r} q_n^{|X(S)|} e^{-\delta q_n v(\mathcal{H}_n)} 
& \leq \sum_{i=0}^{Dp_nv(\mathcal{H}_n)} \binom{v(\mathcal{H}_n)}{i} r^{i} q_n^{i} e^{-\delta q_n v(\mathcal{H}_n)} \nonumber \\
& \leq (Dp_nv(\mathcal{H}_n)+1) (r q_n)^{Dp_nv(\mathcal{H}_n)} \binom{v(\mathcal{H}_n)}{Dp_nv(\mathcal{H}_n)} e^{-\delta q_n v(\mathcal{H}_n)} \nonumber \\
& \leq (Dp_nv(\mathcal{H}_n)+1) \left(\frac{r q_n ev(\mathcal{H}_n)}{Dp_nv(\mathcal{H}_n)}\right)^{Dp_nv(\mathcal{H}_n)} e^{-\delta q_n v(\mathcal{H}_n)} \nonumber \\
& = (Dp_nv(\mathcal{H}_n)+1) \left(\frac{r eC}{D}\right)^{\frac{D}{C}q_nv(\mathcal{H}_n)} e^{-\delta q_n v(\mathcal{H}_n)} \nonumber \\
& \stackrel{(\dagger)}{\leq} e^{-\frac{\delta}{2} q_nv(\mathcal{H}_n)}.
\end{align}
Note that the equality follows since $q_n=Cp_n$.
Now using that $p_nv(\mathcal{H}_n)$ and hence $q_nv(\mathcal{H}_n)$ tends to infinity as $n \to \infty$,
we have $\mathbb{P}[B_{q_n}] \to 0$ as $n \to \infty$, as required.
\end{proof}

\subsubsection{A remark on container theorems for solutions to systems of linear equations in abelian groups}
Note that we could have used Proposition~\ref{prop:rcontainers} to obtain a container theorem for solutions to systems of linear equations in abelian groups, and then used this directly to prove the $1$-statement of Lemma~\ref{lem:randomRado}. 
Saxton and Thomason also gave a container theorem in this setting; see Theorem~4.2 in~\cite{st-online}.
However, we do not use their result, as they give an alternative definition of $m_{G}(A)$ which we now state. 
(See Definitions~1.6 and~4.1~of their paper.)
First define an $\ell\times k$ matrix $A$ to have \emph{full image} if given any $b \in G^{\ell}$, there exists $x \in G^k$ such that $Ax=b$.
Now if $G$ is a finite abelian group which is not a field, define
$m'_{G}(A):=\frac{\ell+t-1}{t-1}$, where $t$ is the maximum $j$ for which 
$A_{\overline{W}}$ has full image
for all  $W\subseteq [k]$ with $|W|=j$.\footnote{Note that Saxton and Thomason only consider $G$ for which $t \geq 2$, and so $m'_G(A)$ is well-defined in this case.}

One can show that $m_{G}(A) \leq m'_{G}(A)$ for all finite abelian groups $G$ and full image matrices $A$ such that  $(A,G)$ is abundant. 
Further, there exist examples for which this inequality is strict. 
For example, consider the matrix
\begin{align*}
A:=\begin{pmatrix} 
1 & 1 & 1 & 0 & 0 \\
0 & 1 & 1 & 1 & 1 \\
\end{pmatrix}
\end{align*}
together with any finite abelian group $G$ which is not a field, for example $\mathbb{Z}_6$.
We have that $(A,G)$ is abundant and in particular $t \geq 2$.
By deleting the first three columns we obtain a submatrix which does not have full image,
and so $t=2$ and $m'_{G}(A)=3$. 
However, under our definition, it is easy to compute that $m_G(A)=2$. (Indeed, we have that $\rank_G(A_{\overline{W}})$ is equal to the number of distinct columns in $\overline{W}$, except for when there are $3$ distinct columns, in which case it equals $2$.)
Consequently, using Theorem~4.2 from~\cite{st-online} for such an example would lead to a suboptimal $1$-statement.

\subsection{Proof of the $0$-statement of Theorem~\ref{thm:mainramsey}}

The following proof adapts an argument of the second and third authors~\cite{ht} (which in turn builds on work from~\cite{random4}). Suppose $(\mathcal{H}_{n})_{n \in \mathbb{N}}$ is a sequence of $k$-uniform ordered hypergraphs that satisfies conditions (P1) and (P3)--(P5) of Theorem~\ref{thm:mainramsey}.
In particular, (P5) yields a sequence $(X_n)_{n \in \mathbb N}$.

Let $c>0$ be sufficiently small and pick $q_n\le c\hat p(\mathcal H_n)$. Our aim is to show that w.h.p. there is a red-blue colouring of the vertices of $\mathcal{H}^v_{n,X_n,q_n}$ so that there are no monochromatic edges. By definition of $\mathcal{H}^v_{n,X_n,q_n}$,
 such a vertex colouring  also does not produce a monochromatic edge in $\mathcal{H}^v_{n,q_n}$, so this would prove the $0$-statement of Theorem~\ref{thm:mainramsey}. We may further assume that $q_n=c\hat p(\mathcal H_n)$.

We say an (ordered) hypergraph is {\it Rado} if it has the property that however its vertices are red-blue coloured, there is always a monochromatic edge.\footnote{So being Rado is equivalent to being $2$-Ramsey.}
Furthermore, a Rado hypergraph is {\it Rado minimal}
if it is no longer Rado under the deletion of any edge.
If $\mathcal{H}^v_{n,X_n,q_n}$ is Rado, fix a spanning subhypergraph $H$ of $\mathcal{H}^v_{n,X_n,q_n}$ which is Rado minimal. 
Otherwise set $H:=\emptyset$. 
So it suffices to prove that w.h.p. $H=\emptyset$. We will do this by combining two lemmas;
\begin{itemize}
\item \emph{A deterministic lemma}, which states that if $H$ is Rado then it must contain one of a finite list of subhypergraphs.
\item \emph{A probabilistic lemma}, which states that w.h.p. $\mathcal{H}^v_{n,X_n,q_n}$, and hence $H$, does not contain any of these subhypergraphs.
\end{itemize}
Before we state these two lemmas, we provide some motivation and state some required definitions.

The first step is to build some structure in $H$.
The following claim is an analogue of 
Proposition~7.4 from~\cite{random4} and Claim~1 from~\cite{ht}.
The proof of the claim follows in the same manner.
\begin{claim}\label{claim:determ}
Suppose $H$ is non-empty, i.e., $H$ is a Rado minimal ordered hypergraph. 
Then for every edge $a$ of $H$ and every vertex $v \in a$, there exists an edge $b$ in $H$ such that $a \cap b =v$. 
\end{claim}

\begin{proof}
Let $a$ be an edge, and let $v \in a$ be such that for all edges $b$ such that $v \in b$, 
there exists another vertex $w \in a$ such that $w \in b$. 
Since $H$ is Rado minimal, 
it is possible to red-blue colour the vertices of $H-a$ so that there are no red or blue edges. 
Thus once we add $a$ back, it must be the case that $a$ is (w.l.o.g.) red since $H$ is Rado.
But then change the colour of $v$ to blue. 
If there is a monochromatic edge now, it must be a blue edge which contains $v$. 
However all edges containing $v$ also contain another vertex from $a$ which is red, thus we obtain a contradiction.  
\end{proof}

Using the above claim, one can build up more and more structure in $H$. 
Eventually, one can obtain enough structure so that $H$ contains one of the subhypergraphs listed in Lemma~\ref{lem:detclaim} below. 
We now define some hypergraph notation in order to state it. In what follows, $t$ is a positive integer.
Further, in these definitions we just view edges as sets of vertices (i.e., we are not concerned about the order of an edge).

\begin{itemize}
\item A \emph{simple path of length $t$} consists of edges $e_1,\dots,e_t$ 
such that $|e_i \cap e_j|=1$ if $j=i+1$, and $|e_i \cap e_j|=0$ if $j>i+1$. 
\item A \emph{fairly simple cycle of length $t$} $(t\ge2)$ consists of a simple path $e_1,\dots,e_t$ and an edge $e_0$ such that $|e_0 \cap e_1|=1$;  $|e_0 \cap e_i|=0$ for $2 \leq i \leq t-1$; $|e_0 \cap e_t|=s \geq 1$. 
\item A \emph{simple cycle} is a fairly simple cycle with $s=1$.
\item A simple path $P=e_1,\dots,e_t$ in $H$ is called \emph{spoiled} if there exists an edge $e^* \in E(H)$ such that 
$e^* \not \in E(P)$, $e^* \subseteq V(P)$ and $|e^* \cap e_1|=1$ where the vertex $e^*\cap e_1$ does not lie in $e_2$.  
\item A subhypergraph $H_0$ of $H$ is said to have a \emph{handle} if there is an edge $e^*$ in $H$ such that $|e^*| > |e^* \cap V(H_0)| \geq 2$.
\item A \emph{faulty simple path of length $t$} ($t \geq 3$) is a simple path $e_1,\dots,e_t$ together with two edges $e_x$ and $e_z$ such that 
$e_1,e_2,e_x$ form a simple cycle with $|e_x \cap e_i|=0$ for $i \geq 3$;
$e_{t-1},e_t,e_z$ form a simple cycle with $|e_z \cap e_i|=0$ for $i \leq t-2$;
 each edge has size $3$; the edges $e_x$ and $e_z$ may or may not be disjoint.

\item  A \emph{bad triple} is a set of three edges $e_1, e_x, e_y$, 
where $e_1 \cap e_x=\{x\}$, $e_1 \cap e_y=\{y\}$, $x \not=y$, and $|e_x \cap e_y| \geq 2$.

\item A \emph{bad tight path} is a set of three edges $e_1, e_2, e_3$ each of size $3$ such that 
$|e_1 \cap e_2|=2$, $|e_1 \cap e_3|=1$ and $|e_2 \cap e_3|=2$.

\end{itemize}
We also require the following definition in the proof of Lemma~\ref{lem:detclaim}.
\begin{itemize}
    \item A \emph{Pasch configuration} is a set of four edges $e_1, e_2, e_3, e_4$ of size $3$ such that $v_{ij}=e_i \cap e_j$ is a distinct vertex for each pair $i<j$.
\end{itemize}
Write $N:=n \cdot q_n = \mathbb E(v(\mathcal{H}^v_{n,X_n,q_n}))$ for brevity. 
\begin{lemma}[Deterministic lemma]\label{lem:detclaim}
If $H$ is non-empty then it contains at least one of the following structures:
\begin{itemize}
\item[(i)]  A simple path of length at least $\log N$.
\item[(ii)] A spoiled simple path of length at most $\log N$. 
\item[(iii)] A fairly simple cycle of length at most $\log N$, with a handle. 
\item[(iv)] A faulty simple path of length at most $\log N$.
\item[(v)] A bad triple.
\item[(vi)] A bad tight path.
\end{itemize}
\end{lemma}

\begin{lemma}[Probabilistic lemma]\label{lem:probclaim}
If $c>0$ is sufficiently small then w.h.p.
$\mathcal{H}^v_{n,X_n,q_n}$ does not contain any of the hypergraphs (i)--(vi) listed in Lemma~\ref{lem:detclaim}.
\end{lemma}

These two results  immediately prove the $0$-statement of Theorem~\ref{thm:mainramsey}.

\subsubsection{Proof of Lemma~\ref{lem:detclaim}}
Suppose for a contradiction that $H$ is non-empty but does not contain any of the structures defined in (i)--(vi). Note that $H$ is an $|X_n|$-uniform ordered hypergraph, and $|X_n|\ge 3$ by (P5).

Let $P=e_1, \dots, e_t$ be the longest simple path in $H$.
As we have no simple paths of length at least $\log N$, we have $t \leq \log N$.
By Claim~\ref{claim:determ}, $t\geq 2$.
Let $x,y$ be two vertices which belong only to $e_1$ in $P$, 
and let $e_x$ and $e_y$ be two edges of $H$ such that $e_z \cap e_1 = \{z\}$ for $z=x,y$ (such $e_x,e_y$ exist by Claim~\ref{claim:determ}). 
By the maximality of $P$, 
we have $h_z:=|V(P) \cap e_z| \geq 2$ for $z=x,y$. 

If $h_z=|X_n|$ for some $z=x,y$, then $P$ together with $e_z$ is a spoiled simple path, a contradiction. Otherwise, let $i_z:=\min\{i \geq 2: e_z \cap e_i \not= \emptyset\}$ for $z=x,y$, 
and assume without loss of generality that $i_y \leq i_x$. 
We know $e_1, \dots, e_{i_x}, e_x$ must not form a fairly simple cycle for which $e_y$ is a handle. 
This implies $e_y \subseteq e_1\cup \dots \cup e_{i_x}\cup e_x$. 
In particular, this means $e_x$ must contain all those vertices in $e_y$ which do not lie on $P$. If $|e_y\cap e_x|\ge2$ then $e_1$, $e_x$ and $e_y$ form a bad triple, a contradiction. Thus, $e_y\cap e_x$ consists of precisely one vertex $v_{xy}$ (and $v_{xy}$ lies outside of $P$). 
Now consider $e_1,\dots, e_{i_y}, e_{y}$. 
This is a fairly simple cycle that $e_x$ intersects in at least two vertices (namely $x$ and $v_{xy}$).
 Thus, we obtain a fairly simple cycle with a handle unless all the vertices in $e_x$ lie in $e_1,\dots, e_{i_y}, e_{y}$. 
 In particular, $e_x \subseteq (e_1 \cup e_{i_y} \cup e_y)$ as $i_y \leq i_x$. 
 This in turn implies $e_{i_y}=e_{i_x}$. 
 Indeed, otherwise $e_x$ must contain one vertex from $e_1$ and $|X_n|-1\geq 2$ vertices from $e_y$, 
 a contradiction as we already observed that $e_x$ only intersects $e_y$ in one vertex.
 
In summary, we have that $i_x=i_y$ and $e_x$ and $e_y$ intersect in a single vertex $v_{xy}$ outside of $P$.
As mentioned in the last paragraph, we must have $e_x \subseteq (e_1 \cup e_{i_y} \cup e_y)$. 
Similarly, we have that $e_1,\dots,e_{i_x},e_x$ form a fairly simple cycle for which $e_y$ is a handle (a contradiction)
unless $e_y \subseteq (e_1 \cup e_{i_x} \cup e_x)$.

The two previous inequalities imply $e_x\setminus\{x,v_{xy}\}\subseteq e_{i_y}$ and $e_y\setminus\{y,v_{xy}\}\subseteq e_{i_x}$. Recall here that $e_{i_x}=e_{i_y}$. We must have that $|e_{i_x} \setminus (e_x \cup e_y)| \geq 1$; 
indeed, otherwise either $e_x$ or $e_y$ contains the vertex $e_{i_x}\cap e_{i_x-1}$, which either 
contradicts the minimality of $i_x$ or $i_y$. Note that $e_x\setminus\{x,v_{xy}\}$ and $e_y\setminus\{y,v_{xy}\}$ are disjoint, they have size $|X_n|-2$ 
each, and their union is a strict subset of $e_{i_x}$. Thus $2(|X_n|-2)\le |e_{i_x}|-1=|X_n|-1$. This implies $|X_n|\le3$, and so $|X_n|=3$.

If $i_x \geq 3$ then $e_1, e_x, e_y$ form a fairly simple cycle for which $e_{i_x}$ is a handle (since $|e_{i_x} \setminus (e_x \cup e_y)| \geq 1$ and $e_1\cap e_{i_x}=\emptyset$), a contradiction. 
Thus we have that $i_x=2$, and so $e_1, e_x, e_y, e_{i_x}$ form a Pasch configuration.

Now repeat the maximal path process which we did for $e_1$ to find $e_x$ and $e_y$, except from the other end of the path.
That is, there must exist edges $e_z$ and $e_w$ such that $e_z \cap e_t = \{z\}$, $e_w \cap e_t = \{w\}$, 
where $z,w$ are vertices in $e_t$ that are not in $e_{t-1}$. 
By repeating the previous case analysis, we arrive at the conclusion that 
$e_{t-1}, e_t, e_z, e_w$ must also form a Pasch configuration where $e_z \cap e_w$ is a vertex $v_{zw}$ outside of $P$. If $t \geq 3$, then $e_1,\dots,e_t, e_x, e_z$ together form a faulty simple path.
Hence we must have $t=2$.

If the union of these two Pasch configurations contains $7$ vertices (i.e., $v_{xy} \not = v_{zw}$), 
then $e_1,e_2,e_x$ form a fairly simple cycle for which $e_z$ is a handle. 
So we now suppose that the two Pasch configurations cover the same $6$ vertices.
If we do not have $\{e_x, e_y\}=\{e_z, e_w\}$ then $e_x, e_z, e_y$ form a bad tight path.
Hence we do have equality and the two Pasch configurations we found are identical.

Relabel the edges and vertices as in the definition of a Pasch configuration.
We observe that $H$ cannot be just these four edges: 
such a hypergraph is not Rado, e.g., colour $v_{12}, v_{13}, v_{34}$ red, and the remaining vertices blue.
Also, by definition of Rado minimal, this cannot be a component of $H$. 
That is, there is an edge $e_5$ in $H$, where $e_5 \not= e_i$ for $i \in [4]$, 
and $e_5$ contains $s$ vertices from inside the Pasch configuration, where $s \geq 1$.
If $s=1$ then w.l.o.g.\ $e_5$ contains $v_{1,2}$; 
then $e_5$, $e_1$, $e_3$ is a simple path of length $3$, 
a contradiction to the longest path in $H$ of length $2$ found earlier.
If $s=2$ then whichever two vertices of the Pasch configuration $e_5$ contains, 
taking any of the simple cycles of the Pasch configuration together with $e_5$ gives a (fairly) simple cycle with a handle.
If $s=3$, first suppose $e_5=\{v_{1,2}, v_{1,3}, v_{2,3}\}$. 
Then $e_1,e_5,e_2$ is a bad tight path. 
If $e_5=\{v_{1,2}, v_{1,3}, v_{2,4}\}$, then again $e_1, e_5, e_2$ is a bad tight path.
For all other $3$-sets of vertices $e_5$ could contain, a symmetrical argument shows that we find a bad tight path.
Since all three values of $s$ give a contradiction, this concludes the proof.\qed

\subsubsection{Proof of Lemma~\ref{lem:probclaim}}
Let $c>0$ be sufficiently small and $n \in \mathbb N$ be sufficiently large. Recall that $q_n=c\hat p_n(\mathcal H_n)$. 
Given an $|X_n|$-uniform ordered hypergraph $G$, 
we will obtain an upper bound on the expected number of copies of $G$ in $\mathcal{H}^v_{n,X_n,q_n}$ 
by obtaining upper bounds on the expected number of possible assignments of vertices $v \in V(\mathcal{H}^v_{n,X_n,q_n})$ to an edge in which some vertices are already fixed, and others are to be assigned.
To make this precise, we need some notation.
Given an edge order $e_1, \dots, e_t$ of $E(G)$, 
we call a vertex $v$ \emph{new in $e_i$} if $v \in e_i$ but $v \not \in e_j$ for all $j<i$.
Otherwise we call $v \in e_i$ \emph{old in $e_i$}. 
Clearly (provided $G$ has no isolated vertices) each vertex in $G$ is new in one edge, and old in any subsequent edge that it appears in. 
Now, for $0 \leq w \leq |X_n|-1$, we will shortly provide    an upper bound $P_w$ on the expected number of edges in $\mathcal{H}^v_{n,X_n,q_n}$
that contain a fixed set of 
 $w$ (old) vertices; crucially,   $P_w$ is independent of the exact set of $w$  vertices considered.
Then one can obtain an upper bound on the expected number of copies of $G$ in $\mathcal{H}^v_{n,X_n,q_n}$ directly as some function of the $P_w$'s (which will depend on $G$). We first obtain bounds for these $P_w$'s.

Given a fixed set of $w$  vertices in $\mathcal{H}^v_{n,X_n,q_n}$, 
there are  $\binom{|X_n|}{w}w!\leq |X_n|!\leq k!$ ways of selecting some $W\subseteq X_n$ with $|W|=w$ and picking an ordering $v_1,\dots,v_w$ of these $w$  vertices. Then, there are at most $\Delta_{W}(\mathcal{H}_{n,X_n})$ choices for an edge $e$ in $\mathcal H_{n,X_n}$ such that $e_W=(v_1,\dots,v_w)$. Each vertex in $\mathcal{H}_{n,X_n}$ is included with probability $q_n$ in $\mathcal{H}^v_{n,X_n,q_n}$,
thus we obtain the upper bound
 \begin{equation*}
P_w = (k!)\cdot\left(\max_{W\subseteq X_n, \,|W|=w}\Delta_{W}(\mathcal{H}_{n,X_n})\right) \cdot (q_n)^{|X_n|-w}.
\end{equation*}
For $w=0$, we will use the  bound 
\begin{equation*}
P_0\le e(\mathcal{H}_{n,X_n}) \cdot (q_n)^{|X_n|}\le n^{|X_n|}\cdot (q_n)^{|X_n|}\le N^k.
\end{equation*}

We now use conditions (P3), (P4) and (P5) to bound $P_w$ further for values of $w$ greater than $0$. First, condition (P4) implies that, 
$$P_1\le(k!)\cdot b\cdot\frac{e(\mathcal H_{n,X_n})}{v(\mathcal H_n)} \cdot (q_n)^{|X_n|-1}.$$
By~\eqref{eq:P4a} and since  $n \in \mathbb N$ is sufficiently large, we have that $f_{n,X_n}\ge\eta \hat p_n(\mathcal H_n)$ for some constant $\eta>0$. It follows from the above that
\begin{align*}
P_1  \stackrel{\phantom{\text{(P3)}}}{\le} & (k!)\cdot b\cdot \frac{e(\mathcal{H}_{n,X_n})}{v(\mathcal{H}_{n})} \cdot (c\hat p_n(\mathcal H_n))^{|X_n|-1} \\
\stackrel{\eqref{eq:P4a}}{\le} & (k!)\cdot b\cdot\frac{e(\mathcal{H}_{n,X_n})}{v(\mathcal{H}_{n})} \cdot(f_{n,X_n})^{|X_n|-1}\cdot(c/\eta)^{|X_n|-1}\\ 
\stackrel{\eqref{eq:keyprobability}}{=} & (k!)\cdot b\cdot(c/\eta)^{|X_n|-1} \le \sqrt{c},
\end{align*}
where the last inequality follows from $c>0$ being sufficiently small. 

Similar calculations hold for higher values of $w$. Recall we have $q_n=c\hat p_n\ge cf_{n,W}$ for every $W\subseteq [k]$ with $|W|\ge2$. Let $2\le w\le|X_n|-1$ and pick $W\subset X_n$ with $|W|=w$ maximising $\Delta_W(\mathcal H_{n,X_n})$. 
As $n \in \mathbb N$ is sufficiently large
we have
\begin{align*}
P_w\stackrel{\text{(P3)}}{\le} & (k!)\cdot b\cdot\frac{e(\mathcal{H}_{n,X_n})}{e(\mathcal{H}_{n,W})}\cdot (q_n)^{|X_n|-w}
= (k!)\cdot b\cdot\frac{e(\mathcal{H}_{n,X_n})}{v(\mathcal{H}_{n})}\cdot (q_n)^{|X_n|-1}\cdot\frac{v(\mathcal H_n)}{e(\mathcal H_{n,W})}\cdot(q_n)^{-(w-1)} \\
\stackrel{\phantom{\text{(P3)}}}{\le} & \left((k!)\cdot b\cdot(c/\eta)^{|X_n|-1}\right)\cdot\frac{v(\mathcal H_n)}{e(\mathcal H_{n,W})}\cdot(c f_{n,W})^{-(w-1)}\\ 
\stackrel{\eqref{eq:keyprobability}}{=} & \left((k!)\cdot b\cdot(c/\eta)^{|X_n|-1}\right)\cdot c^{-(w-1)} \le (k!)\cdot b\cdot c\cdot\eta^{-(|X_n|-1)} \le\sqrt{c},
\end{align*}
where the last two inequalities follow from $|X_n|>w$ and $c>0$ being sufficiently small.

Finally, for any $2\le w\le|X_n|-1$, we have that $P_0\cdot P_2\cdot P_w$ is at most
\begin{align*}
 (k!)^2 \cdot&e(\mathcal{H}_{n,X_n})\cdot\left(\max_{W\subseteq X_n,|W|=2}\Delta_{W}(\mathcal{H}_{n,X_n})\right)\cdot\left(\max_{W\subseteq X_n,|W|=w}\Delta_{W}(\mathcal{H}_{n,X_n})\right)\cdot (q_n)^{3|X_n|-2-w} \\
\stackrel{\eqref{eq:P4b}}{=} & \, o(1).
\end{align*}

For brevity, let $P_{\ge w}:=\max\limits_{i\ge w} P_i$. To recap, we have shown that
\begin{equation}\label{eq:bounds}
P_0\le N^k,\quad P_{\ge 1} \le \sqrt{c}\quad\text{and}\quad P_0\cdot P_2\cdot P_{\ge 2} = o(1).
\end{equation}

\medskip

To prove the lemma it suffices to show that,
for each set of hypergraphs stated in Lemma~\ref{lem:detclaim}, the expected number of such structures in $\mathcal{H}^v_{n,X_n,q_n}$
is $o(1)$. Indeed, then by Markov's inequality w.h.p. $\mathcal{H}^v_{n,X_n,q_n}$ does not contain any of these structures.
In each case, we will find an edge order and then use~(\ref{eq:bounds}) to bound the expected number of these structures.
Let $G$ be an ordered hypergraph.
Given an edge order $e_1,\dots,e_t$ of $E(G)$,
we call an edge $e_i$
\begin{itemize}
\item \emph{initial} if all its vertices are new;
\item \emph{normal} if it has precisely one old vertex;
\item \emph{2-bad} if it has precisely two old vertices;
\item \emph{bad} if it has at least two and at most $|X_n|-1$ old vertices.
\end{itemize}
\medskip
Now observe the bounds $P_0, P_1, P_2$ and $P_{\geq 2}$ correspond
to upper bounds on
the expected number of 
ways one can place 
  an initial, normal, $2$-bad and bad edge in $\mathcal{H}^v_{n,X_n,q_n}$ respectively.

\medskip

{\it Case (i)}.
Note that for this case it suffices to show that w.h.p. $\mathcal{H}^v_{n,X_n,q_n}$
contains no simple path of length exactly
 $s:=\lceil \log N \rceil$; thus we only need to consider the expected number of simple paths of length $s$.
Let $e_1,\dots,e_s$ be the simple path. 
Here, $e_1$ is initial, and all other edges are normal. Thus, we can bound the expected number of simple paths of length $s$ in $\mathcal{H}^v_{n,X_n,q_n}$ by
\begin{align*}
 P_0 \cdot P_1^{s-1} 
\stackrel{(\ref{eq:bounds})}{\leq}  N^{k} \cdot (\sqrt{c})^{s-1} 
\leq (\sqrt{c})^{(\log N)-1} \cdot N^{k} \stackrel{\text{(P1)}}{=} o(1),
\end{align*}
which follows since $c>0$ is sufficiently small. 

\medskip

{\it Case (ii)}.
Let $e_1,\dots,e_t$ be the simple path and let $e^*$ be the spoiling edge. 
Choose $i$ to be the smallest $i \geq 2$ such that $|e_i \cap e^*| \geq 1$.
Consider the edge order  $e^*, e_i, e_{i-1},\dots, e_1, e_{i+1},\dots,e_t$.
Here $e^*$ is initial.
If $|e_i \cap e^*| \geq 2$ then $e_i$ is bad, 
the $e_j$ for $2 \leq j \leq i-1$ are normal,
$e_1$ is $2$-bad,
and each $e_j$ for $i+1 \leq j \leq t$ is normal or bad.
If $|e_i \cap e^*| =1$ then $e_j$ for $2 \leq j \leq i$ are normal,   
$e_1$ is $2$-bad, 
and each $e_j$ for $i+1 \leq j \leq t$ is normal or bad.
In particular, for at least one such $j$, $e_j$ is bad, since $e^*$ must intersect at least one of these edges (as $|X_n|\ge3$).
In either case, we have found an edge order with one 
initial edge, one $2$-bad edge, at least one further 
bad edge, and all other edges normal or bad edges. We 
have $t \leq \log N$ and  there are at most $tk$ 
choices for the location on the simple path of each 
vertex of $e^*$ intersecting it. Overall, the 
expected number of spoiled paths of length at most $\log N$ is at most
\begin{align*}
\sum_{t=2}^{\log N} (tk)^k  \cdot P_0 \cdot P_{\ge 1}^{t-2} \cdot P_2 \cdot P_{\ge2}\stackrel{(\ref{eq:bounds})}{\leq}  \sum_{t=2}^{\log N} (tk)^k  \cdot (\sqrt{c})^{t-2} \cdot o(1) = o(1),
\end{align*}
where the last inequality holds as $c>0$ is sufficiently small.

\medskip

{\it Case (iii)}.
Let $e_1,\dots,e_t,e_0$ be the fairly simple cycle and let $e^*$ be the handle.
Consider the edge order $e_0,e_t,\dots,e_{1},e^*$.
Here $e_0$ is initial,
$e_t$ is normal or bad, the
$e_j$ for $2 \leq j \leq t-1$ are normal,
$e_{1}$ is $2$-bad
and $e^*$ is bad. We have $t \leq \log N$ and  there are at most $tk$ choices for the location on the fairly simple cycle for the choice of where each of the $u$ vertices of the handle $e^*$ intersect it, where $2 \leq u \leq k-1$. Overall, the expected number of fairly simple cycles of length at  most $\log N$ with a handle is at most
\begin{align*}
\sum_{t=2}^{\log N}  \sum_{u=2}^{k-1} (tk)^u \cdot P_0 \cdot P_{\ge1}^{t-1} \cdot P_2 \cdot P_{\ge2}\stackrel{(\ref{eq:bounds})}{\leq} \sum_{t=2}^{\log N}  \sum_{u=2}^{k-1} (tk)^u \cdot (\sqrt{c})^{t-1} \cdot o(1)  = o(1).
\end{align*}
\medskip

{\it Case (iv)}.
Let $e_1,\dots,e_t, e_x,e_z$ be the faulty simple path.
Consider the edge order $e_1,e_x,e_2,\dots,$ $e_{t-1},e_z,e_t$.
Here $e_1$ is initial,
$e_2$ and $e_t$ are $2$-bad, 
$e_z$ is normal or $2$-bad (depending whether $e_z$ intersects $e_x$), 
and all other edges are normal. We have $t \leq \log N$, and a choice of whether $e_x$ and $e_z$ intersect or not. Overall, we can bound the expected number of faulty simple paths of length at most $\log N$ by
\begin{align*}
\sum_{t=3}^{\log N} P_0 \cdot P_1^{t-1} \cdot P_2^2 + \sum_{t=3}^{\log N} P_0 \cdot P_1^{t-2} \cdot P_2^3
\stackrel{(\ref{eq:bounds})}{\leq} 
2\sum_{t=3}^{\log N} (\sqrt{c})^{t-1} \cdot o(1) = o(1).
\end{align*}
\medskip

{\it Case (v)}.
Let $e_1,e_x,e_y$ be the bad triple.
Consider the edge order $e_x, e_y, e_1$.
Here, $e_x$ is initial,
$e_y$ is bad
and $e_1$ is $2$-bad. We have a choice of the size of $t:=|e_x \cap e_y|$. Overall, the expected number of bad triples is at most
\begin{align*}
\sum_{t=2}^{k-1} P_0 \cdot P_2 \cdot P_{\ge2}\stackrel{(\ref{eq:bounds})}{=}  o(1).
\end{align*}
\medskip

{\it Case (vi)}.
Let $e_1,e_2,e_3$ be a bad tight path. In this edge order, $e_1$ is initial, and $e_2$ and $e_3$ are $2$-bad. Overall, we can bound the expected number of bad tight paths by
\begin{align*}
P_0 \cdot P_2^2 \stackrel{(\ref{eq:bounds})}{=}  o(1). 
\end{align*}
\qed

\section{Proof of the random Rado lemma}\label{sec:proofblackbox}

Before we deduce the random Rado lemma (Lemma~\ref{lem:randomRado}) from Theorem~\ref{thm:mainramsey},
 we next prove Lemma~\ref{lem:kdistinct}.

\begin{proof}[Proof of Lemma~\ref{lem:kdistinct}]
We first prove the second statement of the lemma.
\begin{claim}\label{claim:specialk}
$$\lim_{n\to\infty}\frac{|k\-Sol_{S_n}^A([k])|}{|\Sol_{S_n}^A([k])|}=1.$$
\end{claim}
\begin{proofclaim}
We bound from above the number of non $k$-distinct solutions in $\Sol_{S_n}^A([k])$. For each such solution, we have $|S_n|$ choices for the repeated entry.
Using conditions (A3) and (A6), for $n$ sufficiently large, we obtain
\begin{align*}
\sum_{z \in S_n} \sum_{\stackrel{W \subseteq [k]}{|W|=2}} |\Sol_{S_n}^A((z,z),W,[k])|
&\stackrel{\text{(A3)}}{\leq} 
B \sum_{z \in S_n}\sum_{\stackrel{W \subseteq [k]}{|W|=2}} \frac{|\Sol_{S_n}^A([k])|}{|\Sol_{S_n}^A(W)|}\\
&=B|S_n|\sum_{\stackrel{W \subseteq [k]}{|W|=2}} \frac{|\Sol_{S_n}^A([k])|}{|\Sol_{S_n}^A(W)|}\stackrel{\text{(A6)}}{=}o(|\Sol_{S_n}^A([k])|).
\end{align*}
In other words, $|\Sol_{S_n}^A([k])|-|k\-Sol_{S_n}^A([k])|=o(|\Sol_{S_n}^A([k])|)$ which immediately implies the claim.\qedclaim
\end{proofclaim}

Next, we prove the inequality stated in the lemma. Let $Y\subseteq[k]$. Note that $k\-Sol_{S_n}^A(Y)\subseteq\Sol_{S_n}^A(Y)$, so the upper bound follows trivially. By definition, if $y\in k\-Sol_{S_n}^A(Y)$ then there exists some $x\in k\-Sol_{S_n}^A([k])$ such that $x_Y=y$. Thus,
$$|k\-Sol_{S_n}^A(Y)|\cdot\max_{y_0\in S^{|Y|}}|\Sol_{S_n}^A(y_0,Y,[k])|\ge |k\-Sol_{S_n}^A([k])|.$$
As $(A,S_n)$ is $B$-extendable, for $n$ large the above implies
$$|k\-Sol_{S_n}^A(Y)|\cdot\frac{B|\Sol_{S_n}^A([k])|}{|\Sol_{S_n}^A(Y)|}\ge |k\-Sol_{S_n}^A([k])|.$$
Rearranging gives
$$\frac{|k\-Sol_{S_n}^A(Y)|}{|\Sol_{S_n}^A(Y)|}\ge \frac{|k\-Sol_{S_n}^A([k])|}{B|\Sol_{S_n}^A([k])|}\ge\frac{1}{2B},$$
where the last inequality holds for $n$ sufficiently large  by Claim~\ref{claim:specialk}.
\end{proof}

Lemma~\ref{lem:randomRado} follows easily from Theorem~\ref{thm:mainramsey}. We use Lemma~\ref{lem:kdistinct} to restrict our attention to $k$-distinct solutions.

\begin{proof}[Proof of Lemma~\ref{lem:randomRado}]
Let $r\ge2$, $k  \geq 3$ and $\ell\geq 1$ be integers. Let $(S_n)_{n \in \mathbb N}$ be a sequence of finite subsets of abelian groups and $A$ be an $\ell \times k$ integer matrix. Suppose that conditions (A1)--(A6) from Lemma~\ref{lem:randomRado} hold. Let $\mathcal{H}_n$ 
be the $k$-uniform ordered hypergraph with
 vertex set $S_n$, and  edge set consisting of all $k$-distinct solutions $x=(x_1,\dots,x_k)$ to $Ax=0$ in $S_n$. 

Let $Y\subseteq[k]$. Observe that $e(\mathcal H_{n,Y})=|k\-Sol_{S_n}^A(Y)|$. As (A3) and (A6) hold,  Lemma~\ref{lem:kdistinct} implies that
\begin{equation}\label{eq:ordermag1}
e(\mathcal H_{n,Y})=\Theta{(|\Sol_{S_n}^A(Y)|)}\quad\text{and}\quad e(\mathcal H_{n})\sim|\Sol_{S_n}^A([k])|.
\end{equation}
It follows that, for $|Y|\ge2$, 
\begin{equation}\label{eq:ordermag2}
f_{n,Y}=\left(\frac{e(\mathcal H_{n,Y})}{v(\mathcal H_n)}\right)^{-\frac{1}{|Y|-1}}=\Theta \left(\frac{|\Sol_{S_n}^A(Y)|}{|S_n|}\right)^{-\frac{1}{|Y|-1}}=\Theta(p_Y(A,S_n))
\end{equation}
and thus $\hat p(\mathcal H_n)=\Theta(\hat p(A,S_n))$. We now prove that conditions (P1)--(P5) hold.
\begin{itemize}
\item[(P1):] Since $\hat p_n(\mathcal H_n)=\Theta(\hat p_n(A,S_n))$ and $|S_n|=v(\mathcal H_n)$, (A1) implies that (P1) holds.
\item[(P2):] Since $e(\mathcal H_{n})=|k\-Sol_{S_n}^A([k])|\sim|\Sol_{S_n}^A([k])|$ by~\eqref{eq:ordermag1},  (A2) implies that  (P2) holds.
\item[(P3):] For any $W\subseteq Y\subseteq[k]$ we have 
$$\Delta_W(\mathcal H_{n,Y})\le\max_{w_0\in S_n^{|W|}}\{|\Sol_{S_n}^A(w_0,W,Y)|\}\stackrel{\text{(A3)}}{\le}\Theta\left(\frac{|\Sol_{S_n}^A(Y)|}{|\Sol_{S_n}^A(W)|}\right)\stackrel{\eqref{eq:ordermag1}}{\le}\Theta\left(\frac{|e(\mathcal H_{n,Y})|}{|e(\mathcal H_{n,W})|}\right),$$
and so (P3) holds.
\item[(P4):] Pick $W\subset Y\subseteq[k]$ with $|W|=1$. We have
$$\Delta_W(\mathcal H_{n,Y})\le\max_{w_0\in S_n}|\Sol_{S_n}^A(w_0,W,Y)|\stackrel{\text{(A4)}}{=}O\left(\frac{|\Sol_{S_n}^A(Y)|}{|S_n|}\right)\stackrel{\eqref{eq:ordermag1}}{=}O\left(\frac{e(\mathcal H_{n,Y})}{v(\mathcal H_n)}\right).$$
\item[(P5):] Pick a sequence $(X_n)_{n\in\mathbb N}$ of subsets $X_n\subseteq[k]$ with $|X_n|\ge3$ for all $n\in\mathbb N$ as in Definition~\ref{def:compatibility}. Note that
$$f_{n,X_n}=\left(\frac{e(\mathcal H_{n,X_n})}{v(\mathcal H_n)}\right)^{-\frac{1}{|X_n|-1}}\stackrel{\eqref{eq:ordermag2}}{=}\Theta\left(\frac{\Sol_{S_n}^A(X_n)}{|S_n|}\right)^{-\frac{1}{|X_n|-1}}=\Theta(\hat p(A,S_n))=\Theta(\hat p(\mathcal H_n)),$$
and so $f_{n,X_n}=\Omega(\hat p(\mathcal H_n))$. It remains to verify equation~\eqref{eq:P4b}. Pick sequences $(W_n)_{n\in\mathbb N}$ and $(W'_n)_{n\in\mathbb N}$ with $W_n,W'_n\subset X_n$, $|W_n|=2$ and $|W'_n|\ge2$. Then
\begin{align*}
& \Delta_{W_n}(\mathcal H_{n,X_n})\cdot\Delta_{W'_n}(\mathcal H_{n,X_n})\cdot e(\mathcal H_{n,X_n})\cdot\hat p(\mathcal H_n)^{3|X_n|-2-|W'_n|}\\
\stackrel{\text{(P3)},\eqref{eq:ordermag1}}{\le} & 
\Theta \left (
\frac{|\Sol_{S_n}^A(X_n)|}{|\Sol_{S_n}^A(W_n)|}\cdot\frac{|\Sol_{S_n}^A(X_n)|}{|\Sol_{S_n}^A(W'_n)|}\cdot|\Sol_{S_n}^A(X_n)|\cdot\left(\frac{|\Sol_{S_n}^A(X_n)|}{|S_n|}\right)^{-\frac{3|X_n|-2-|W'_n|}{|X_n|-1}}
\right )
\\
\stackrel{\phantom{(P3),\eqref{eq:ordermag1}}}{=} & \Theta \left (\frac{|S_n|^2}{|\Sol_{S_n}^A(W_n)|}\cdot\frac{|S_n|}{|\Sol_{S_n}^A(W'_n)|}\cdot\left(\frac{|\Sol_{S_n}^A(X_n)|}{|S_n|}\right)^{\frac{|W'_n|-1}{|X_n|-1}}\right )\\
\stackrel{\phantom{(P3),\eqref{eq:ordermag1}}}{=} & \Theta \left (\frac{|S_n|^2}{|\Sol_{S_n}^A(W_n)|}\cdot\left(\frac{p_{W'_n}(A,S_n)}{p_{X_n}(A,S_n)}\right)^{|W'_n|-1}\right )\stackrel{\text{(A5)}}{=}o(1).
\end{align*}

\end{itemize}

As conditions (P1)--(P5) are satisfied, by Theorem~\ref{thm:mainramsey} there exist constants $c,C>0$ such that 
\begin{align*}
\lim_{n \rightarrow \infty}  \mathbb{P}[\mathcal{H}^v_{n,q_n} \text{ is $r$-Ramsey}]
=\begin{cases}
0 &\text{ if } q_n\leq {c}\hat p(\mathcal H_n); \\
 1 &\text{ if } q_n\geq {C}\hat p(\mathcal H_n).\end{cases}
\end{align*}
Note that $\mathcal{H}^v_{n,q_n}$ is $r$-Ramsey if and only if $S_{n,q_n}$ is $(A,r)$-Rado. Furthermore, we have $\hat p(\mathcal H_n)=\Theta(\hat p(A,S_n))$ and so there exist constants $c_0,C_0$ such that
\begin{align*}
\lim_{n \rightarrow \infty}  \mathbb{P}[S_{n,q_n} \text{ is $(A,r)$-Rado}]
=\begin{cases}
0 &\text{ if } q_n\leq {c_0}\cdot\hat p(A,S_n); \\
 1 &\text{ if } q_n\geq {C_0}\cdot\hat p(A,S_n).\end{cases}
\end{align*}
\end{proof}

\section{Proof of two auxiliary results}\label{sec:aux}
In this section, we  prove two auxiliary results. Lemma~\ref{lem:qAPbound} gives an upper bound on the number of $k$-APs in $\mathbf P_n$ containing a fixed prime number and it is used in the proofs of Theorem~\ref{gt1} and Theorem~\ref{rvdwp}. Theorem~\ref{thm:supersat-d} is a supersaturation result for integer lattices which is used to prove Theorem~\ref{thm:ramseynd}. 

\subsection{Proof of Lemma~\ref{lem:qAPbound}}\label{aux1}
We make use of the following sieve result. 

\begin{lemma}[\cite{taoblog}]\label{lem:taosieve}
Let $k\in\mathbb N$ with $k\ge2$ and let $q$ be a prime number. For each prime $p\ge k$ with $p\not=q$, let $E_p$ be the union of $k-1$ residue classes modulo $p$. Then $|E(n)|=O(n/\log^{k-1} n)$ where
$$E(n):=[n]\setminus\bigcup_{\substack{\text{$p\not=q$ prime}\\ k\le p\le\sqrt{n}}} E_p.$$
\end{lemma}
Lemma~\ref{lem:taosieve} is a weakened version of Theorem~32 from~\cite{taoblog}. A precise explanation of how to deduce Lemma~\ref{lem:taosieve} from the original statement in~\cite{taoblog} can be found in Appendix~\ref{appendix:sieve}.

\begin{proof}[Proof of the Lemma~\ref{lem:qAPbound}]
Let $k,\ell, n$ and $q$ be as in the statement of Lemma~\ref{lem:qAPbound}. Let $\mathcal A(n)$ be the set of $k$-APs in $\mathbb Z$ such that the $\ell$th term of the progression is $q$ and the common difference is an integer between $1$ and $n$ inclusive. In particular, we have $|\mathcal A(n)|=n$. Furthermore, the collection of $k$-APs in $\mathbf P_n$ whose $\ell$th element is $q$ is a subset of $\mathcal A(n)$. Thus, it suffices to show that the number of $k$-APs in $\mathcal A(n)$ consisting entirely of primes is $O(n/\log^{k-1}n)$.

For every prime number $p\ne q$, let  $\mathcal A_p$ be the set of $k$-APs in $\mathcal A(n)$ with at least one term divisible by $p$. Consider the set
$$\mathcal E(n):=\mathcal A(n)\setminus\bigcup_{\substack{\text{$p\ne q$ prime}\\ k\le p\le\sqrt{n}}} \mathcal A_p.$$ 

In other words, $\mathcal E(n)$ is the set of all $k$-APs in $\mathcal A(n)$ such that no term of the progression is divisible by a prime between $k$ and $\sqrt{n}$ (except $q$). In particular, $\mathcal E(n)$ contains all $k$-APs in $\mathcal A(n)$ whose terms (other than perhaps $q$) are all primes greater than $\sqrt{n}$. Conversely, the number of $k$-APs in $\mathcal A(n)$ containing at least one prime $p' \leq \sqrt{n}$ where $p' \neq q$  is $o(n/\log^{k-1} n)$. This follows from (i) the fact that there are at most $k$ arithmetic progressions of length $k$ containing both $q$ (in the $\ell$th position) and a fixed prime $p\not=q$  and as (ii) $\mathbf P_{\sqrt{n}}\sim\frac{2\sqrt{n}}{\log n}$ by the prime number theorem. 

By the observations above, the number of $k$-APs in $\mathcal A(n)$ consisting entirely of primes is at most $o(n/\log^{k-1} n)+|\mathcal E(n)|$; so it suffices to prove that $|\mathcal E(n)|=O(n/\log^{k-1}n)$.

\smallskip

Now we reformulate the problem in terms of common differences. Since the $\ell$th term of all $k$-APs in $\mathcal A(n)$ is equal to $q$, each $k$-AP in $\mathcal A(n)$ is uniquely determined by its common difference. For every prime $p\not=q$, let $E_p$ be the set of all common differences of $k$-APs in $\mathcal A_p$. It follows immediately that $|E_p|=|\mathcal A_p|$. Furthermore, the set of common differences of $k$-APs in $\mathcal A(n)$ is exactly $[n]$ by definition. Thus we have $|\mathcal E(n)|=|E(n)|$ where
$$E(n):=[n]\setminus\bigcup_{\substack{\text{$p\ne q$ prime}\\ k\le p\le\sqrt{n}}} E_p.$$

Applying Theorem~\ref{lem:taosieve}  yields $|\mathcal E(n)|=|E(n)|=O(n/\log^{k-1}n)$, as required. The only assumption of Theorem~\ref{lem:taosieve} that needs to be checked is that, for every prime $p \ge k$ with $p\not=q$, $E_p$ is the union of $k-1$ residue classes modulo $p$.  

Let $a\in[n]$ and $b_i:=q+a(i-\ell)$ for every $i\in[k]$. Then, $a\in E_p$ if and only if $(b_1,\dots,b_k)$ is a $k$-AP in $\mathcal A_p$, i.e., $p$ divides some $b_i$. Now, $p$ divides $b_i=q+a(i-\ell)$ if and only if $i\not=\ell$ and $a\equiv -q(i-\ell)^{-1} \, (\text{mod} \ p)$. Note that $i-\ell\not=0$ is always invertible modulo $p$ since $|i-\ell|<k\le p$. It follows that $E_p=\{a\in[n]:a\equiv-q(i-\ell)^{-1}\text{ (mod $p$), }i\in[k],i\not=\ell\}$. Hence, $E_p$ is the set of precisely $k-1$ residue classes modulo $p$, as required.
\end{proof}

\subsection{Supersaturation in the integer lattice}\label{aux2}

To prove Theorem~\ref{thm:supersat-d}, we first need a few intermediate results. Throughout this subsection, we will repeatedly use the following fact. 

\begin{fact}\label{fact:solbound}
Let $d,\ell, k\in\mathbb N$ and $A$ be an $\ell \times k$ integer matrix of rank $\ell$. Let $b\in(\mathbb Z^d)^\ell$ and $S\subseteq\mathbb Z^d$ be a finite subset. The number of solutions to $Ax=b$ in $S$ is at most $|S|^{k-\ell}$.
\end{fact}
\begin{proof}
As $A$ has rank $\ell$, there are $\ell$ linearly independent columns in $A$, w.l.o.g. the first $\ell$ columns. We bound the number of solutions $x=(x_1,\dots,x_k)$ to $Ax=0$ in $S$ as follows. There are at most $|S|^{k-\ell}$ choices for $x_{\ell+1},\dots,x_k$. For each such choice, $x':=(x_1,\dots, x_\ell)$ is a solution to $A'x'=b'$ in $S\subseteq \mathbb Z^d$ where $A'$ is the $\ell\times\ell$ matrix consisting of the first $\ell$ columns of $A$ and $b'$ is some vector in $(\mathbb Z^d)^\ell$. Since the columns of $A'$ are linearly independent, there is at most one solution to $A'x'=b'$ in $\mathbb Z^d$. Thus, there are at most $|S|^{k-\ell}$ solutions to $Ax=0$ in $S$, as required.
\end{proof}

The following  lemma  essentially states that if a supersaturation result holds for $S\subseteq\mathbb Z^d$ then it also holds for $S'\subseteq S$, provided that $S'$ contains almost all elements of $S$.

\begin{lemma}\label{lem:supersattruncation} 
Let $r,d, \ell , k \in\mathbb N$ and $\gamma>0$. Let $A$ be a partition regular $\ell\times k$ integer matrix of rank $\ell$ and let $S$ and $S'$ be finite sets such that $S'\subseteq S\subseteq\mathbb Z^d$ and $|S'|\geq (1-\gamma /(2k))|S|$.
Fix an $r$-colouring of $S$ (and thus $S'$). If there are at least $\gamma |S|^{k-\ell}$ monochromatic solutions to $Ax=0$ in $S$, then there are at least $(\gamma/2)|S|^{k-\ell}$ monochromatic solutions to $Ax=0$ in $S'$.
\end{lemma}
\begin{proof}
We determine an upper bound on the number of monochromatic solutions $(\roww{z}_1,\dots,\roww{z}_k)$ which lie in $S$ but not in $S'$. 
For one such solution, there must be some $\roww{z}_i$ which lies in $S$ but not in $S'$. We have $k$ choices for $i$ and $|S|-|S'|\leq \gamma|S|/(2k)$ choices for $\roww{z}_i$. The remaining elements $(\roww{z}_1,\dots,\roww{z}_{i-1},\roww{z}_{i+1},\dots,\roww{z}_k)$ form a solution to $A'x=\roww{b}$ in $S$ where $A'$ is obtained from $A$ by removing the $i$th column and $\roww{b}$ is some vector. Note that $A'$ has rank $\ell$. If not, then the $i$th column of $A$ cannot be expressed as a linear combination of the columns of $A'$. In particular, the $i$th entry of any solution to $Ax=0$ in $\mathbb Z$ must be $0$, contradicting the assumption that $A$ is partition regular. Since $A'$ has rank $\ell$, there are at most $|S|^{k-1-\ell}$ solutions to $A'x=\roww{b}$ in $S$ by Fact~\ref{fact:solbound}. Overall, there are at most $k\cdot \frac{\gamma |S|}{2k} \cdot |S|^{k-\ell-1}=\frac{\gamma}{2}|S|^{k-\ell}$ possible choices for $(\roww{z}_1,\dots,\roww{z}_k)$.

Therefore, the number of monochromatic solutions to $Ax=0$ in $S'$ is at least 
$$\gamma |S|^{k-\ell}-\frac{\gamma}{2} |S|^{k-\ell}= \frac{\gamma}{2} |S|^{k-\ell}.$$
\end{proof}

The following proposition asserts that if the statement of Theorem~\ref{thm:supersat-d} holds for $d\in\mathbb N$ then it also holds for $d-1$.

\begin{prop}\label{prop:reduction}
Let $r,d,n\in\mathbb N$ with $d \geq 2$ and $\delta>0$. Let $A$ be an $\ell \times k$ integer matrix of rank $\ell$.
If every $r$-colouring of $[n]^d$ yields at least $\delta n^{d(k-\ell)}$ monochromatic solutions to $Ax=0$ in $[n]^d$, then every $r$-colouring of $[n]^{d-1}$ yields at least $\delta n^{(d-1)(k-\ell)}$ monochromatic solutions to $Ax=0$ in $[n]^{d-1}$.
\end{prop}
\begin{proof}
Let $A,r,d,n$ and $\delta$ be as in the statement of the proposition. Let $f:[n]^d\to[n]^{d-1}$ be such that for $\roww{x}=(x_1,\dots,x_d)\in[n]^d$ we have $f(\roww{x}):=(x_1,\dots,x_{d-1})$, i.e., $f(\roww{x})$ is the vector consisting of the first $d-1$ entries of $\roww{x}$. We first prove the following property of $f$.  

\begin{claim}\label{claim:f}
Given a solution $(\roww{y}_1,\dots,\roww{y}_k)$ to $Ax=0$ in $[n]^{d-1}$, there are at most $n^{k-\ell}$ solutions $(\roww{x}_1,\dots,\roww{x}_k)$ to $Ax=0$ in $[n]^d$ such that $f(\roww{x}_i)=\roww{y}_i$ for every $i\in[k]$.
\end{claim}
\begin{proofclaim}
Let $S$ be the set of solutions $(\roww{x}_1,\dots,\roww{x}_k)$ in $[n]^d$ such that $f(\roww{x}_i)=\roww{y}_i$ for every $i\in[k]$. Let $q:S\to[n]^k$ with $q((\roww{x}_1,\dots,\roww{x}_k)):=(x_1,\dots,x_k)$ where $x_i$ is the $d$th entry of $\roww{x}_i$. Since $S$ is a set of solutions in $[n]^d$, it follows that the image $q(S)$ of $q$ is a set of solutions to $Ax=0$ in $[n]$. In particular, $|q(S)|\le n^{k-\ell}$ by Fact~\ref{fact:solbound}. Note that the assumption $f(\roww{x}_i)=\roww{y}_i$ for every $i\in[k]$ implies $q$ is injective, hence $|S|=|q(S)|\le n^{k-\ell},$ as required.\qedclaim
\end{proofclaim}

Pick an arbitrary $r$-colouring $\mathcal C:[n]^{d-1}\to[r]$ of $[n]^{d-1}$. Let $\mathcal C^*:[n]^d\to[r]$ be the $r$-colouring of $[n]^d$ where $\mathcal C^*(\roww{x}):=\mathcal C(f(\roww{x}))$. If $(\roww{x}_1,\dots,\roww{x}_k)$ is a monochromatic solution to $Ax=0$ in $[n]^d$ (with respect to $\mathcal C^*$) then $(f(\roww{x}_1),\dots,f(\roww{x}_k))$ is a monochromatic solution to $Ax=0$ in $[n]^{d-1}$ (with respect to $\mathcal C$). By assumption, there exist at least $\delta n^{d(k-\ell)}$ monochromatic solutions to $Ax=0$ in $[n]^d$ with respect to $\mathcal C^*$. Hence, by Claim~\ref{claim:f}, the number of monochromatic solutions in $[n]^{d-1}$ with respect to $\mathcal C$ is at least
$$\frac{\delta n^{d(k-\ell)}}{n^{k-\ell}}=\delta n^{(d-1)(k-\ell)},$$
as required.
\end{proof}

By Proposition~\ref{prop:reduction}, it is clear that if the statement of Theorem~\ref{thm:supersat-d} holds for an infinite sequence $d_1<d_2<\dots$ of positive integers then it holds for any $d\in\mathbb N$. Therefore, it suffices to prove the following relaxation of Theorem~\ref{thm:supersat-d}.

\begin{thm}\label{thm:supersat-c}
Let $r\in\mathbb N$, $c\in\mathbb N\cup\{0\}$ and let $A$ be a partition regular $\ell \times k$ integer matrix of rank $\ell$.
There exist $\delta,n_0>0$ such that for all $n \geq n_0$ the following holds: every $r$-colouring of $[n]^{2^c}$ yields at least $\delta n^{2^c(k-\ell)}$ monochromatic solutions to $Ax=0$ in $[n]^{2^c}$.
\end{thm} 

\begin{proof}[Proof of Theorem~\ref{thm:supersat-d}]
Let $r,d\in\mathbb N$ and let $A$ be a partition regular $\ell \times k$ integer matrix of rank $\ell$. Choose $c\in\mathbb N$ such that $d\le 2^c$. Let $\delta,n_0$ be the constants obtained by applying Theorem~\ref{thm:supersat-c} with parameters $r,c$ and $A$. By Proposition~\ref{prop:reduction}, the statement of Theorem~\ref{thm:supersat-d} holds for $A,r,\delta,n_0$ and every $d'\in\mathbb N$ with $d'\le 2^c$. In particular, it holds for $A,r,d,\delta,n_0$. 
\end{proof}

We conclude this subsection with the proof of Theorem~\ref{thm:supersat-c}. Before proceeding with the proof, we give a rough intuition of the main idea by considering an easier case. For integers $a,b\in\mathbb Z$, we write $[a,b]:=\{m\in\mathbb Z:a\le m\le b\}$. Let $A=(1\quad 1\quad -1)$. Given an arbitrary $r$-colouring $\mathcal C$ of $[-n,n]^2$, our aim is to find $\Theta(n^4)$ monochromatic solutions to $Ax=0$. 

Define a colouring $\mathcal C^*$ of $[(\eps n)^2]$, where $\eps>0$ is a small constant, as follows. If $z\in[(\eps n)^2]$, then $z=x+(\eps n)y$ for a unique pair $(x,y)\in[\eps n]\times[0,\eps n-1]$. Then colour $z$ with $(c_{-2},c_{-1},c_0,c_1,c_2)$ where $c_\lambda$ is the colour of $(x-\lambda(\eps n),y+\lambda)\in[-n,n]^2$ with respect to $\mathcal C$. Note that if $z_1+z_2-z_3=0$ is a solution in $[(\eps n)^2]$, then $(x_1+x_2-x_3)+(\eps n)(y_1+y_2-y_3)=0$. Since the $x_i$'s lie in $[\eps n]$, $|y_1+y_2-y_3|\leq 2$. 
Thus, there exists $\lambda\in\mathbb Z$ with $|\lambda|\le 2$ such that $(y_1+\lambda)+(y_2+\lambda)-(y_3+\lambda)=0$, which implies $(x_1-(\eps n)\lambda)+(x_2-(\eps n)\lambda)-(x_3-(\eps n)\lambda)=0$. 
In particular, $\{( x_j-\lambda (\eps n), y_j+\lambda)\}_{j=1,2,3}$ 
is a solution to $Ax=0$ in $[-n,n]^2$. Furthermore, suppose $\{z_j\}_{j=1,2,3}$ is a monochromatic solution to $Ax=0$ in $[(\eps n)^2]$ with respect to $\mathcal C^*$; say the colour is $(c_{-2},c_{-1},c_0,c_1,c_2)$. Then there exists $\lambda\in\mathbb Z$ with $|\lambda|\le 2$ such that
$\{(x_j-\lambda (\eps n),
y_j+\lambda)\}_{j=1,2,3}$ is a monochromatic solution to $Ax=0$ in $[-n,n]^2$ with respect to $\mathcal C$, the colour being $c_\lambda$.

It is not  difficult to see that different monochromatic solutions in  $[(\eps n)^2]$ with respect to $\mathcal C^*$ give rise to different 
monochromatic solutions in $[-n,n]^2$ with respect to $\mathcal C$. That is, we have defined 
 an injection from monochromatic solutions in $[(\eps n)^2]$ with respect to $\mathcal C^*$ to monochromatic solutions in $[-n,n]^2$ with respect to $\mathcal C$. By Theorem~\ref{thm:fgrsupersat} there are $\Theta(n^4)$ many of the former, yielding the required result.

\smallskip

The proof of Theorem~\ref{thm:supersat-c} proceeds by induction on $c$ and follows a similar idea, where instead of Theorem~\ref{thm:fgrsupersat} we invoke the 
induction hypothesis. Essentially, we project a solution to $Ax=0$ in $[(\eps n)^2]^{2^c}$ to an `almost' solution in $[\eps n]^{2^{c+1}}$, and then apply a shifting procedure to obtain an exact solution in $[n]^{2^{c+1}}$. The colouring $\mathcal C^*$ we 
require is more involved than the one above due to two technicalities: (i) ensuring that the parameter $\roww{\lambda}$ (which is the analogue of $\lambda$ in the general case) is a vector with integer entries, thereby implying that the shifting procedure is well-defined; (ii)~ensuring that solutions generated by the shifting procedure lie in $[n]^{2^{c+1}}$ rather than $[-n,n]^{2^{c+1}}$.  

We will use the following notation. Given two vectors $\roww{x}$ and $\roww{y}$, we denote their concatenation as $\roww{x}*\roww{y}$. Furthermore, we let $abs[\roww{x}]$ be the vector obtained by taking the absolute value of all entries of $\roww{x}$. For example, if $\roww{x}=(1,-2)$ and $\roww{y}=(3,4)$ then $\roww{x}*\roww{y}=(1,-2,3,4)$ and $abs[\roww{x}]=(1,2)$. We write $\roww{x}\cdot \roww{y}$ for the scalar product of $\roww{x}$ and $\roww{y}$. Let 
$\roww{1}$ denote  the all one vector.

\begin{proof}[Proof of Theorem~\ref{thm:supersat-c}]
Let $r\in\mathbb N$ and let $A$ be a partition regular $\ell \times k$ integer matrix of rank $\ell$. We proceed by induction on $c$. The case $c=0$ follows immediately from Theorem~\ref{thm:fgrsupersat}. For the induction hypothesis, suppose that the statement of Theorem~\ref{thm:supersat-c} holds for some $c\in\mathbb N\cup\{0\}$.  

Denote the $i$th row of $A$ as $\roww{a}_i$. Let $C:=\max_{i\in[\ell]}\{|\roww{a}_i\cdot\roww{1}|,1\}$. By relabelling the rows of $A$ and changing their sign (these operations do not affect the set of solutions), we may assume that $C=\max\{\roww{a}_1\cdot\roww{1},1\}\ge 1$. Let $E\ge0$ be the largest entry of $A$ in absolute value. Fix constants $\eps,P$ such that
$$Ek\ll 1/\eps\ll P,$$
where $P$ is a prime number that does not divide $C$. Let $d:=2^c$ and $R:=r^{(2Ek+1)^d}\cdot (CP)^d$. 
By assumption, we can apply Theorem~\ref{thm:supersat-c} with input $A,R$ and $c$ to obtain $\gamma,n_0>0$. 
Finally, we take $n\in\mathbb N$ with 
$n \ge\max\{n_0/\eps ,2^{d+2}Ek^2d/(\eps\gamma)\}$.
 For simplicity, we will assume that $\eps n$ is an integer. 

Let $\mathcal C$ be an arbitrary $r$-colouring of $[0,n]^{2d}$. We define an $R$-colouring $\mathcal C^*$ of $[Ek(\eps n)+1,(\eps n)^2]^d$ as follows. 
For any $\roww{z}\in[Ek(\eps n)+1,(\eps n)^2]^d$, there exists a unique pair of vectors $\roww{x}\in[\eps n]^d$ and $\roww{y}\in[Ek,\eps n]^d$ such that $\roww{z}=\roww{x}+\roww{y}(\eps n)$. 
Let $I:=[-Ek,Ek]^d$. We colour $\roww{z}$ with $(\{c_{\roww{s}}\}_{\roww{s}\in I},\roww{t})$ where for each $\roww{s} \in I$, $c_{\roww{s}}$ is the colour of $abs[(\roww{y}+\roww{s})*(\roww{x}-\roww{s}(\eps n))]\in[0,n]^{2d}$ with respect to $\mathcal C$ and $\roww{t}\in[CP]^d$ is the unique vector such that $\roww{y}\equiv\roww{t}$ (mod $CP$). Note that all entries of $(\roww{y}+\roww{s})*(\roww{x}-\roww{s}(\eps n))$ have absolute value at most $(\eps n)(Ek+1)\ll n$, so $abs[(\roww{y}+\roww{s})*(\roww{x}-\roww{s}(\eps n))]$ indeed lies in $[0,n]^{2d}$.

Observe that $\mathcal C^*$ is a well-defined $R$-colouring of 
$[Ek(\eps n)+1,(\eps n)^2]^d$: $\roww{x},\roww{y}$ and $\roww{t}$ are unique, so the colour of $\roww{z}\in[Ek(\eps n)+1,(\eps n)^2]^d$ in $\mathcal C^*$ is uniquely determined. There are $r$ possible colours for each $c_{\roww{s}}$ and $(CP)^d$ possible values for $\roww{t}$; since $|I|=(2Ek+1)^d$, at most $R$ colours have been used. 

\medskip

Next, we describe a shifting procedure which generates a monochromatic solution in $[0,n]^{2d}$ with respect to $\mathcal C$ given a monochromatic solution in $[Ek(\eps n)+1,(\eps n)^2]^d$ with respect to $\mathcal C^*$.

Suppose $(\roww{z}_1,\dots,\roww{z}_k)$ is a monochromatic solution to $Ax=0$ in $[Ek(\eps n)+1,(\eps n)^2]^d$ with respect to $\mathcal C^*$. For every $1\le i\le k$ we have $\roww{z}_i=\roww{x}_i+\roww{y}_i(\eps n)$ for some unique $\roww{x}_i\in[\eps n]^d$ and $\roww{y}_i\in[Ek,\eps n]^d$. Let $\roww{z}_i:=(z_i^1,\dots,z_i^d)$, $\roww{x}_i:=(x_i^1,\dots,x_i^d)$ and $\roww{y}_i:=(y_i^1,\dots,y_i^d)$. Pick an arbitrary $j\in[d]$ and let $\roww{x}:=(x_1^j,\dots,x_k^j)$, $\roww{y}:=(y_1^j,\dots,y_k^j)$ and $\roww{z}:=(z_1^j,\dots,z_k^j)$. Then we have $A\roww{z}=0$ and $\roww{z}=\roww{x}+(\eps n)\roww{y}$. In particular, for every $i\in[\ell]$ we have 
\begin{equation}\label{eq:x+y}
\roww{a}_i\cdot\roww{z}=\roww{a}_i\cdot\roww{x}+(\eps n)\roww{a}_i\cdot\roww{y}=0.
\end{equation}

Note that the absolute value of every entry of $\roww{a}_i$ is at most $E$ and every entry of $\roww{x}$ lies in $[\eps n]$, thus we obtain 
\begin{equation}\label{eq:xbound}
|\roww{a}_i\cdot\roww{x}|\le \eps n Ek.
\end{equation}
Now \eqref{eq:x+y} and \eqref{eq:xbound} immediately imply
\begin{equation}\label{eq:ybound}
|\roww{a}_i\cdot\roww{y}|\le Ek.
\end{equation}
Next, we distinguish two cases.

\smallskip

{\bf Case 1.} Suppose that $C=\roww{a_1}\cdot\roww{1}\ge 1$. Since $(\roww{z}_1,\dots,\roww{z}_k)$ is monochromatic with respect to $\mathcal C^*$, there exists some $t\in[CP]$ such that $\roww{y}\equiv t\roww{1}$ (mod $CP$). Hence, we have 
$$\roww{a}_1\cdot\roww{y}\equiv\roww{a}_1\cdot(t\roww{1})\equiv t(\roww{a}_1\cdot\roww{1})\equiv Ct\equiv 0 \; (\text{mod } C).$$
Set $\lambda:=-(\roww{a}_1\cdot\roww{y})/C$. The previous equality implies that $\lambda$ is an integer, while by \eqref{eq:ybound} we have $|\lambda|\le Ek/C$. Observe that
\begin{align}\label{andrew}
\roww{a}_1\cdot(\roww{y}+\lambda\roww{1})=\roww{a}_1\cdot\roww{y}+\lambda C=0.
\end{align}
As 
$\roww{y}\equiv t\roww{1}$ (mod $CP$), we have
$$0\stackrel{(\ref{andrew})}{=}
\roww{a}_1\cdot(\roww{y}+\lambda\roww{1})\equiv (t+\lambda)\roww{a}_1\cdot\roww{1} \equiv C(t+\lambda) \; (\text{mod } P),$$
and so we must have that $t+\lambda\equiv 0$ (mod $P$) since $P$ is a prime number which does not divide $C$. It follows that
\begin{equation}\label{eq:ymod}
\roww{a}_i\cdot(\roww{y}+\lambda\roww{1})\equiv (t+\lambda)(\roww{a}_i\cdot\roww{1}) \equiv 0 \;(\text{mod } P)
\end{equation}
for every $i\in[\ell]$. Note that 
$$|\roww{a}_i\cdot(\roww{y}+\lambda\roww{1})|\le |\roww{a}_i\cdot \roww{y}|+|\lambda(\roww{a}_i\cdot\roww{1})|\stackrel{\eqref{eq:ybound}}{\le} Ek+|\lambda C| \le 2Ek.$$
As we picked $P$ so that $P\gg Ek$, the previous inequality and \eqref{eq:ymod} imply that
$$\roww{a}_i\cdot(\roww{y}+\lambda\roww{1})=0$$
for every $i\in[\ell]$.
Furthermore, this together with \eqref{eq:x+y}
implies that,
for every $i\in[\ell]$ we have 
$$\roww{a}_i\cdot(\roww{x}-\lambda(\eps n)\roww{1})
=0.$$

\smallskip

{\bf Case 2.} Suppose $\roww{a}_1\cdot\roww{1}=0$ (and so $\roww{a}_i\cdot\roww{1}=0$ for every $i\in[\ell]$). Then set $\lambda:=0$ and observe that 
$$\roww{a}_i\cdot(\roww{y}+\lambda\roww{1})=\roww{a}_i\cdot\roww{y}\equiv t(\roww{a}_i\cdot\roww{1})\equiv 0 \;(\text{mod } P).$$  
By  \eqref{eq:ybound} and as $P\gg Ek$, we deduce that $\roww{a}_i\cdot(\roww{y}+\lambda\roww{1})=0$ and so $\roww{a}_i\cdot(\roww{x}-\lambda(\eps n)\roww{1})=0$, for every $i \in [\ell]$.

\medskip

In both cases, we obtain two solutions $(y_i^j+\lambda)_{i\in[k]}$ and $(x_i^j-\lambda(\eps n))_{i\in[k]}$ to $Ax=0$ in $\mathbb Z$, where $|\lambda|\le Ek$. In general, since $j$ was arbitrarily chosen, for every $j\in[d]$ there exists $\lambda_j\in [-Ek,Ek]$ such that $(y_i^j+\lambda_j)_{i\in[k]}$ and $(x_i^j-\lambda_j(\eps n))_{i\in[k]}$ are two solutions to $Ax=0$ in $\mathbb Z$. 
By setting $\roww{\lambda}:=(\lambda_1,\dots,\lambda_d)$, we obtain that $(\roww{y}_i+\roww{\lambda})_{i \in [k]}$ and $(\roww{x}_i-\roww{\lambda}(\eps n))_{i\in[k]}$ are two solutions to $Ax=0$ in $\mathbb Z^d$.

Note that $\roww{y}_i+\roww{\lambda}\in[0,2\eps n]^d$ for every $i\in[k]$ since $\roww{y}_i\in[Ek,\eps n]^d$ and $\roww{\lambda}\in[-Ek,Ek]^d$. Furthermore, for every $m\in [d]$, the $m$th entries of the vectors $(\roww{x}_i-\roww{\lambda}(\eps n))_{i\in[k]}$ must have the same sign since $\roww{x}_i\in[\eps n ]^d$ and $\roww{\lambda}$ has integer entries.
 Hence $(abs[(\roww{y}_i+\roww{\lambda})*(\roww{x}_i-\roww{\lambda}(\eps n))])_{i\in[k]}$ is a solution to $Ax=0$ in $[0,n]^{2d}$.

Crucially, all vectors $(abs[(\roww{y}_i+\roww{\lambda})*(\roww{x}_i-\roww{\lambda}(\eps n))])_{i\in[k]}$ have the same colour in $\mathcal C$ by the construction of $\mathcal C^*$ and the fact that $\roww{\lambda}\in I=[-Ek,Ek]^d$.
Namely, if all $(\roww{z}_1,\dots,\roww{z}_k)$ are 
coloured with, say, $(\{c_{\roww{s}}\}_{\roww{s}\in I},\roww{t})$ in $\mathcal C ^*$ then all 
$(abs[(\roww{y}_i+\roww{\lambda})*(\roww{x}_i-\roww{\lambda}(\eps n))])_{i\in[k]}$ are coloured with 
$c_{\roww\lambda}$ in $\mathcal C$. Hence, 
$(abs[(\roww{y}_i+\roww{\lambda})*(\roww{x}_i-\roww{\lambda}(\eps n))])_{i\in[k]}$ is a monochromatic solution to $Ax=0$ in $[0,n]^{2d}$.

\medskip

Arbitrarily extend the $R$-colouring $\mathcal C^*$ of
$[Ek(\eps n)+1,(\eps n)^2]^d$ to an $R$-colouring
$\mathcal C^{**}$
of $[(\eps n)^2]^d$. As $(\eps n)^2\geq n_0$, the induction hypothesis implies that there are 
at least $\gamma(\eps n)^{2d(k-\ell)} $
monochromatic solutions in $[(\eps n)^2]^d$ with respect to $\mathcal C^{**}$.
Note that $[Ek(\eps n)+1,(\eps n)^2]^d$
contains all but at most
$dEk(\eps n)^{2d-1} \leq \gamma (\eps n)^{2d}/(2k)$ of the elements of 
$[(\eps n)^2]^d$. Thus, 
by Lemma~\ref{lem:supersattruncation}, there are at least $(\gamma/2) (\eps n)^{2d(k-\ell)}$ monochromatic solutions in $[Ek(\eps n)+1,(\eps n)^2]^d$ with respect to $\mathcal C^*$; moreover,
using the procedure described above, we can produce a monochromatic solution in $[n]^{2d}$ with respect to $\mathcal C$ from each one of them. It remains to check how many times we could have counted each solution  in $[n]^{2d}$.

Observe that a different monochromatic solution $(\roww{z_1},\dots,\roww{z}_k)$ in $[(Ek(\eps n)+1,(\eps n)^2]^d$ would have generated different $k$-tuples $(y_i^j+\lambda)_{i\in[k]}$ and $(x_i^j-\lambda(\eps n))_{i\in[k]}$, since $\roww{x}_i$ and $\roww{y}_i$ are unique and by \eqref{eq:x+y}.  On the other hand, the preimage of a singleton in $abs:[-n,n]^d\to[0,n]^d$ has size at most $2^d$. Hence, a monochromatic solution in $[n]^{2d}$ with respect to $\mathcal C$ could be counted at most $2^d$ times. 

In conclusion, there are at least $(\gamma/2^{d+1}) (\eps n)^{2d(k-\ell)}$ monochromatic solutions to $Ax=0$ in $[0,n]^{2d}$ with respect to $\mathcal C$. Using Lemma~\ref{lem:supersattruncation}, there are at least $(\gamma/2^{d+2}) (\eps n)^{2d(k-\ell)}$ monochromatic solutions to $Ax=0$ in $[n]^{2d}$. Taking $\delta:=(\gamma/2^{d+2}) \eps^{2d(k-\ell)}$, the statement of Theorem~\ref{thm:supersat-c} holds for $c+1$ since $2d=2^{c+1}$. This concludes the inductive step and the proof.
\end{proof}

\section{Proof of the applications of the random Rado lemma}\label{sec:proofapplications}

In this section, we prove the applications of the random Rado lemma presented in Sections~\ref{sec:primes} and~\ref{sec:easyapplications},
that is Theorems~\ref{rvdwp},~\ref{thm:ramseynd},~\ref{thm:sec4ex2},~\ref{thm:exponent}, ~\ref{thm:fields} and~\ref{thm:fields2}.
The first uses the full version of the random Rado lemma, Lemma~\ref{lem:randomRado}.
The others all use the simplified version, Lemma~\ref{lem:easyRado}.

\begin{proof}[Proof of Theorem~\ref{rvdwp}]
Let $r,k\in\mathbb N$ with $r\ge2$ and $k\ge3$. Let  $A=(a_{ij})$ be the $(k-2)\times k$ matrix where
$$a_{ij}:=\begin{cases}
 1 & \text{ if $i=j$;} \\
-2 & \text{ if $i+1=j$;} \\
 1 & \text{ if $i+2=j$;} \\
 0 & \text{ otherwise.}
\end{cases}$$
Hence $k\-Sol_{\mathbf{P}_n}^A([k])$ corresponds to the set of all $k$-APs in $\mathbf{P}_n$. 

The prime number theorem yields
\begin{align}
|\mathbf{P}_n| &= \Theta \left( \frac{n}{\log n} \right). \label{eq:p1}   
\end{align}
Note that any solution to $Ax=0$ that does not correspond to a $k$-AP must be a trivial solution (i.e., all entries in $x$ are the same). Moreover, 
given fixed integers $a,b$ and $1\leq i< j \leq k$,
there is at most one $k$-AP that has $a$ in its $i$th position and $b$ in its $j$th position. Thus, for any 
$Y \subseteq [k]$ with $2 \leq |Y| \leq k$,
each element of $\Sol_{\mathbf{P}_n}^A(Y)$ corresponds
 to either a trivial solution to $Ax=0$ or 
a single $k$-AP. Hence, by Theorem~\ref{gt}
we have that 
\begin{align}
|\Sol_{\mathbf{P}_n}^A(Y)| &= \Theta \left( \frac{n^2}{\log^kn} \right) \quad \text{for each $Y \subseteq [k]$ with $2 \leq |Y| \leq k$}. \label{eq:p2} 
\end{align}
Therefore, (\ref{eq:p1}) and~(\ref{eq:p2}) imply that
\begin{align}
p_Y(A,\mathbf{P}_n) &= \Theta \left( \left(\frac{n}{\log^{k-1}n} \right)^{-\frac{1}{|Y|-1}} \right)
\quad \text{for each $Y \subseteq [k]$ with $2 \leq |Y| \leq k$}. \label{eq:p3} 
\end{align}
Note  that $[k]$ maximises $p_Y(A,\mathbf{P}_n)$ over all $Y \subseteq [k]$ with $|Y| \geq 2$, and so
\begin{equation}\label{eq:hatp}
\hat p_n:=\hat p(A,\mathbf P_n)=
\Theta \left( 
n^{-\frac{1}{k-1}} \log n\right).
\end{equation}
We now verify the conditions of Lemma~\ref{lem:randomRado}.
\begin{itemize}
\item[(A1):] By~(\ref{eq:p1}) and~(\ref{eq:hatp}), we have $|\mathbf{P}_n| \hat{p}_n = \Theta (n^{\frac{k-2}{k-1}}) \to \infty$
and $\hat{p}_n \to 0$
as $n \to \infty$, since $k\ge3$.

\item[(A2):] Theorem~\ref{thm:supersatprimes} immediately implies that $\mathbf{P}_n$ is $(A,r)$-supersaturated.

\item[(A3):] 
Let $W\subseteq Y\subseteq[k]$. First, consider the case $|W|=1$. 
Fix a prime $q \in \mathbf{P}_n$. If $|Y|=1$ then $|\Sol_{\mathbf{P}_n}^A(q,W,Y)|\le1$. If $|Y|\ge2$, we have
\begin{align}\label{eq:A3primes}
|\Sol_{\mathbf{P}_n}^A(q,W,Y)| 
  =
O\left(\frac{n}{\log^{k-1} n}\right)  
\stackrel{(\ref{eq:p1}),(\ref{eq:p2})}{=}O \left( \frac{|\Sol_{\mathbf{P}_n}^A(Y)|}{|\mathbf{P}_n|} \right) = O \left( \frac{|\Sol_{\mathbf{P}_n}^A(Y)|}{|\Sol_{\mathbf{P}_n}^A(W)|} \right),
\end{align}
where we used 
Lemma~\ref{lem:qAPbound} for the first equality
and $|\Sol_{\mathbf{P}_n}^A(W)| \leq |\mathbf{P}_n|$ for the third.

Now suppose $|W|\geq 2$. Since specifying two elements (and their position) of a $k$-AP uniquely determines the latter,
 for any $z_W \in S^{|W|}$, we have 
\begin{align*}
|\Sol_{\mathbf{P}_n}^A(z_W,W,Y)| \leq 1 \stackrel{\eqref{eq:p2}}{=} O\left( \frac{|\Sol_{\mathbf{P}_n}^A(Y)|}{|\Sol_{\mathbf{P}_n}^A(W)|} \right).
\end{align*}
Therefore, there exists $B\geq 1$ such that $(A,\mathbf{P}_n)$ is $B$-extendable for every $n$ sufficiently large, as required.

\item[(A4):] Let $W\subset Y\subseteq[k]$ with $|W|=1$. Since $|Y|\ge2$, equation~\eqref{eq:A3primes} holds. The first two equalities of~\eqref{eq:A3primes}  immediately imply condition (A4).

\item[(A5):] 
We take $X_n=[k]$ for all $n\in\mathbb N$. We have $|X_n|=k\ge3$ and $p_{X_n}(A,\mathbf{P}_n)=\hat p_n$. Let $W,W'\subset[k]$ with $|W|=2$ and $|W'|\ge2$.
As $|W'| \leq k-1$ we have
\begin{align*}
\frac{|\mathbf{P}_n|^2}{|\Sol_{\mathbf{P}_n}^A(W)|} \left( \frac{p_{W'}(A,\mathbf{P}_n)}{p_{[k]}(A,\mathbf{P}_n)}\right)^{|W'|-1} 
\stackrel{(\ref{eq:p1}),(\ref{eq:p2}),(\ref{eq:p3})}{=} O\left( \log^{k-2}n \left(\frac{n}{\log^{k-1}n}\right)^{\frac{|W'|-k}{k-1}}\right) \to 0,    
\end{align*}
as $n \to \infty$, as required.

\item[(A6):] Let $W \subseteq [k]$ with $|W|=2$. Then~(\ref{eq:p1}) and~(\ref{eq:p2}) imply
\begin{align*}
\frac{|\mathbf{P}_n|}{|\Sol_{\mathbf{P}_n}^A(W)|} = O \left( \frac{\log^{k-1}n}{n} \right) \to 0,
\end{align*}
as $n \to \infty$, as required.
\end{itemize}

\smallskip

Since conditions (A1)--(A6) hold, we can apply Lemma~\ref{lem:randomRado}. Note that $\mathbf P_{n,p}$ is $(A,r)$-Rado if and only if $\mathbf P_{n,p}$ is $(r,k)$-van der Waerden. Thus Lemma~~\ref{lem:randomRado} yields  Theorem~\ref{rvdwp}. 
\end{proof}

\begin{proof}[Proof of Theorem~\ref{thm:ramseynd}]
Let $r,d\in\mathbb N$ with $r\ge2$. Let $A$ satisfy the assumptions of Theorem~\ref{thm:ramseynd}. We will verify that the conditions of Lemma~\ref{lem:easyRado} hold, where we take $S_n:=[n]^d$. Note that $[n]$ is a finite subset of the field $\mathbb Q$ and $m_{S_n}(A)=m(A)$. 

\begin{itemize}

\item[(C2):] Theorem~\ref{thm:supersat-d} implies that $([n]^d)_{n \in \mathbb{N}}$ is $(A,r)$-supersaturated.

\item[(C3):] An immediate consequence of Theorem~\ref{thm:supersat-d} is that there exists some $\delta>0$ such that the number of solutions to $Ax=0$ in $[n]^d$ is at least $\delta n^{d(k-\ell)}$, for all $n$ sufficiently large. Since there are at most $n^{d(k-\ell)}$ solutions to $Ax=0$ in $[n]^d$ (by e.g., Lemma~\ref{lem:keycounting}), $(A,[n]^d)$ is $\delta$-rich for $n$  sufficiently large.

\item[(C4):] Since $A$ is irredundant, $A$ is irredundant with respect to $\mathbb{Q}^d$. Furthermore, by Theorem~\ref{thm:supersat-d}, $A$ is $3$-partition regular in $[n]^d$ for $n$ sufficiently large. 
Thus, we can apply Lemma~\ref{lem:primpliesabundant} with $\mathbb F:=\mathbb{Q}$ and $S:=[n]^d$ to obtain that $(A,[n]^d)$ is abundant for $n$ sufficiently large.

\item[(C1):] 
By (C4), $m(A)$ is well-defined (see Remark~\ref{rmk:m(a)abb}). As $A$ is full rank, considering $W=[k]$ yields $m(A) \geq \frac{k-1}{k-1-\ell}>1$. 
Since (C3) holds, Lemma~\ref{lem:pSA} yields $\hat{p}_n= \Theta(n^{-d/m(A)})$.
Therefore, we have $|[n]^d| \hat {p}_n \to \infty$ and $\hat{p}_n \to 0$ as $n \to \infty$.

\item[(C5):] Note that $\rank_{S_n}(A)>0$ for every $n$ since $S_n\not=\{0\}$. As (C3) and (C4) are satisfied, we can apply Lemma~\ref{lem:pWXtozero}. Lemma~\ref{lem:pWXtozero} and (C1) imply that (C5) holds. 

\end{itemize}

\smallskip

As conditions (C1)--(C5) hold, we can apply Lemma~\ref{lem:easyRado} to obtain Theorem~\ref{thm:ramseynd}.
\end{proof}

\begin{proof}[Proof of Theorem~\ref{thm:exponent}]
Let $r\in\mathbb N$ with $r\ge2$ and let $A$ and $G$ be as in the statement of Theorem~\ref{thm:exponent}. Note that $G^n$ is a finite abelian group. We will verify that the hypothesis of Lemma~\ref{lem:easyRado} holds, where we take $S_n:=G^n$.  
We will repeatedly use the fact that $\rank_{G^n}(A_W)=\rank_G(A_W)$ for any $W\subseteq [k]$; this follows easily from Definition~\ref{def:rankfinite}. 
Thus, $m_{G^n}(A)=m_{G}(A)$.
\begin{itemize}

\item[(C2):] Since $G$ is a finite abelian group with exponent $s$ and $A$ satisfies the $s$-columns condition, Theorem~\ref{thm:exponentsupersat} implies $(G^n)_{n \in \mathbb{N}}$ is $(A,r)$-supersaturated.

\item[(C3):] Equation~\eqref{eq:rank1} implies $(A, G^n)$ is $1$-rich.

\item[(C4):] Let $W\subseteq[k]$ with $|W|=2$. By assumption, $(A,G)$ is abundant, and so $\rank_G(A_{\overline W})=\rank_G(A)$. This implies $\rank_{G^n}(A_{\overline W})=\rank_{G^n}(A)$. 
As $W$ is arbitrary, $(A,G^n)$ is abundant.

\item[(C1):] By (C4), $m_{G^n}(A)$ is well-defined and positive (see Remark~\ref{rmk:m(a)abb}). 
Since (C3) holds, 
Lemma~\ref{lem:pSA} yields 
$\hat{p}_n :=\hat p(A,G^n)= \Theta((|G|^n)^{-1/m_{G^n}(A)})$.
As $\rank_{G^n}(A)=\rank_G(A)>0$ and by using $W=[k]$, we have $m_{G^n}(A) \geq \frac{k-1}{k-1-\rank_{G^n}(A)}>1$.
Therefore we have $|G|^n \hat{p}_n \to \infty$ and $\hat{p}_n \to 0$ as $n \to \infty$.

\item[(C5):] Let $Z\subseteq[k]$ and 
$$D(Z):=\frac{|Z|-1-\rank_G(A)+\rank_G(A_{\overline Z})}{|Z|-1}.$$
As $(A, G^n)$ is $1$-rich,~\eqref{eq:key2} implies  that $p_Z(A,G^n)=|G|^{-n \cdot D(Z)}$ for every $n \in \mathbb{N}$. Pick $X\subseteq[k]$, $|X|\ge2$ so that the function $D(X)$ is minimised;
additionally choose $X$ so
that $|X|$ is as small as possible under this assumption.
 It follows that $p_X(A,G^n)=\hat p(A,G^n)$.
 For every $|Z|=2$ we have $D(Z)=1$, by (C4), while $D([k])<1$ since $\rank_G(A)>0$. Hence, $|X|\ge3$.

Let $W'\subset X$ with $|W'|\ge2$. We have
$$\frac{p_{W'}(A,G^n)}{p_X(A,G^n)}=(|G|^n)^{-D(W')+D(X)}.$$

Since $X$ is minimal and there are finitely many subsets $Z\subset X$, there exists $\eps>0$ such that we have $D(Z)-D(X)\ge\eps$ for every $Z\subset X$. It follows that 
$$\frac{p_{W'}(A,G^n)}{p_X(A,G^n)}\le|G|^{-\eps n}\to0$$
as $n \to \infty$.

\end{itemize}

\smallskip

As conditions (C1)--(C5) hold, we can apply Lemma~\ref{lem:easyRado} to obtain  Theorem~\ref{thm:exponent}.

\end{proof}

\begin{proof}[Proof of Theorem~\ref{thm:sec4ex2}]
Since $m_{\mathbb Z_4}(A)=4/3$ when $A=(2\;\; 2\; -2)$ (see Example~\ref{examp:mZ4}), to prove the theorem
it suffices to verify that the hypothesis of Theorem~\ref{thm:exponent} holds in our setting.  

Firstly, $\mathbb Z_4$ is an abelian group with exponent $4$. The matrix $A$ satisfies the $4$-columns condition: 
the last two entries of $A$ sum to zero  and they span the first entry.

Next, we check $(A,\mathbb Z_4)$ is abundant and $\rank_{\mathbb Z_4}(A)>0$.
The total number of solutions to $2x_1+2x_2-2x_3=0$ in $\mathbb Z_4$ is $4\cdot4\cdot2=4^{3-1/2}$; so
 $\rank_{\mathbb Z_4}(A)=1/2>0$.
For any $W\subseteq[3]$ with $|W|=2$, the equation $A_{\overline W}x=0$ in $\mathbb Z_4$ is the same as $2x=0$ in $\mathbb Z_4$. We have $2=4^{1-1/2}$ solutions to $2x=0$ in $\mathbb Z_4$, thus $\rank_{\mathbb Z_4}(A_{\overline W})=1/2=\rank_{\mathbb Z_4}(A)$. 
In particular, $(A,\mathbb Z_4)$ is abundant.
\end{proof}

\begin{proof}[Proof of Theorem~\ref{thm:fields}]
Consider any full rank $\ell \times k$
matrix $A$  with entries from $\mathbb F$ that satisfies the hypothesis of the theorem.
By Definition~\ref{def:rankPID}, 
$\rank _{\mathbb F}(A)=\rank _{\mathbb F^n}(A) $.
Moreover, as $\mathbb F^n$ is an abelian group, one
can also define $\rank _{\mathbb F^n}(A)$ analogously to Definition~\ref{def:rankfinite}. In particular, both these notions of $\rank _{\mathbb F^n}(A)$ are equivalent. Thus, 
(\ref{eq:rank1}) holds in our setting (with $S={\mathbb F}^n$).

Similarly, one can define $m_{\mathbb F^n}(A)$ analogously to Definition~\ref{def:mGA}. Further, as
$\rank _{\mathbb F}$ is the same as $\rank _{\mathbb F^n} $ we have that 
$m_{\mathbb F^n}(A)=m_{\mathbb F}(A)$, where the latter parameter was defined in (\ref{m(A)deffield}).

As (\ref{eq:rank1}) holds,
one can check that  all the results used to prove
Lemma~\ref{lem:easyRado} extend to this setting of matrices with entries from $\mathbb F$.
Thus, the statement of Lemma~\ref{lem:easyRado} holds
for $S_n:=\mathbb F^n $ and full rank matrices $A$ 
with entries from $\mathbb F$. In particular, the notions of $(A,r)$-saturated, abundant and $\eps$-rich can be analogously defined in this setting. (For a formal justification of all of this, see the PhD thesis of the first author~\cite{phdandrea}.)

Since $m_{\mathbb F^n}(A)=m_{\mathbb F}(A)$, to prove the theorem it suffices to show that (C1)--(C5) hold with respect to $S_n=\mathbb F^n$.
\begin{itemize}

\item[(C2):] Theorem~\ref{thm:fieldsupersat} implies that $(\mathbb{F}^n)_{n \in \mathbb{N}}$ is $(A,r)$-supersaturated.

\item[(C3):] Equation~\eqref{eq:rank1} implies that $(A,\mathbb F^n)$ is $1$-rich.

\item[(C4):] Theorem~\ref{thm:bdh} implies that $A$ is $3$-partition regular in $\mathbb{F}^{n}\setminus\{0\}^n$ (for $n$ sufficiently large). By assumption $A$ is irredundant with respect to $\mathbb F^n$. Arguing precisely as in the proof of  Lemma~\ref{lem:primpliesabundant}, one can conclude that $(A,\mathbb{F}^n)$ is abundant for all sufficiently large $n$.

\item[(C1):] 
As $(A,\mathbb{F}^n)$ is abundant for sufficiently large $n$, $(A,\mathbb{F})$ is abundant.
Remark~\ref{rmk:m(a)abb} implies that
$m_{\mathbb F}(A)$ is well-defined and positive. As $A$ is full rank, considering $W=[k]$ yields $m_{\mathbb F}(A) \geq \frac{k-1}{k-1-\ell}>1$. 
Since (C3) holds, (the analogue of) Lemma~\ref{lem:pSA} yields $\hat{p}_n = 
\Theta(|\mathbb{F}^n|^{-1/m_{\mathbb F}(A)})$.
Therefore we have $|\mathbb{F}^n|\hat {p}_n \to \infty$ and $\hat{p}_n \to 0$ as $n \to \infty$.

\item[(C5):] Since (C3) and (C4) are satisfied, we can apply (the analogue of) Lemma~\ref{lem:pWXtozero}. This together with (C1) implies that (C5) holds.
\end{itemize}
\end{proof}

\begin{proof}[Proof of 
Theorem~\ref{thm:fields2}]
Let $A$ be as in the theorem. Recall that as $A$ is an integer matrix, we can view it as 
a matrix with entries from $\mathbb F$ (i.e., an entry $a_{ij}$ corresponds to $a_{ij} \cdot 1_{\mathbb F}\in \mathbb F$).
Thus, we show that the theorem follows directly from Theorem~\ref{thm:fields}. 

Note that the  finite field $\mathbb{F}$ of order $q^k$ has exponent $q$ (in particular, $\langle 1_{\mathbb F}\rangle\cong\mathbb Z_q$) so the fact that 
$A$ satisfies the 
$q$-columns condition immediately implies it also satisfies the columns condition over $\mathbb F$. We can therefore apply 
Theorem~\ref{thm:fields} to obtain the theorem.
\end{proof}

\section{Further applications and concluding remarks}\label{sec:conclusion}
We have seen a number of applications of the (simplified) random Rado lemma. Of course, there are many other potential applications; we give a couple more  in the next subsection.
In Section~\ref{sec:res} we consider a resilience result. 
In the final subsection, we consider some possible future research directions.
All proofs omitted in this section appear in the PhD thesis of the first author~\cite{phdandrea}.

\subsection{Further applications of the random Rado lemma}
The next result is an analogue of the random Rado theorem for $\mathbb Z_n$. 
\begin{thm}\label{thm:Zn}
For all irredundant partition regular full rank integer matrices $A$, and all $r \geq 2$, there exist constants $C,c>0$ such that the following holds.
$$\lim _{n \rightarrow \infty} \mathbb P [\mathbb Z_{n,p} \text{ is } (A,r)\text{-Rado}]=\begin{cases}
0 &\text{ if } p \leq cn^{-1/m(A)}; \\
 1 &\text{ if } p \geq Cn^{-1/m(A)}.\end{cases} $$\qed 
\end{thm}
Theorem~\ref{thm:Zn}  follows from Lemma~\ref{lem:randomRado},\footnote{In particular, one cannot apply Lemma~\ref{lem:easyRado}  here as not all choices of $A$ ensure $(A,\mathbb Z_n$) is abundant; see~\cite{phdandrea}.}
though actually
 the $1$-statement of Theorem~\ref{thm:Zn}
is also an immediate consequence of Theorem~\ref{r3}.
The reader may wonder why the parameter $m(A)$ rather than $m_{\mathbb Z_n}(A)$
appears in the statement of Theorem~\ref{thm:Zn}.
In particular,  $m_{\mathbb Z_n}(A)$ may not equal $m(A)$ and it can be the case that 
 $m_{\mathbb Z_n}(A)$ may not be the same for all values of $n$. 
For example, for $A=\begin{pmatrix} 2 & 2 & -2 \end{pmatrix}$ and even $n$ we have $m_{\mathbb Z_n}(A)=2/(1+\log_n2)$. 
However, in this case we have $n^{-1/m_{\mathbb Z_n}(A)}=\frac{1}{\sqrt{2}}n^{-1/m(A)}$.
More generally,  
we have that $n^{-1/m(A)}=\Theta (n^{-1/m_{\mathbb Z_n}(A)})$ and so we can use
$m(A)$ in the statement of Theorem~\ref{thm:Zn}
instead of $m_{\mathbb Z_n}(A)$.

\medskip

In general, our machinery ensures that, given any `well behaved' sequence $(S_n)_{n\in\mathbb N}$ of finite abelian groups, one can obtain a corresponding random Ramsey-type theorem if one can prove a supersaturation result.
It would therefore be interesting to prove other supersaturation results for abelian groups \`a la Theorem~\ref{thm:exponentsupersat}. Furthermore,
it is likely that  `projection' type
arguments such as  in the proof of 
Theorem~\ref{thm:supersat-d} can be employed to obtain
new supersaturation results from existing results (such as Theorems~\ref{thm:exponentsupersat} and~\ref{thm:supersat-d}).

As noted in Section~\ref{subsec:supersat}, sometimes supersaturation results are easy to obtain, for example, for matrices $A$ that are translation-invariant. 
In such cases there are fewer conditions to check in order to prove a random Ramsey-type result.
By applying Lemma~\ref{lem:easyRado} we obtain the following.

\begin{thm}\label{thm:transl-inv}
Let $(S_n)_{n\in\mathbb N}$ be a sequence of finite abelian groups and $A$ be an integer matrix that is translation-invariant with
respect to $S_n$ for all sufficiently large
$n \in \mathbb N$.
If (C1), (C4) and (C5) are satisfied then
there exist constants $C,c>0$ such that the following holds.
$$\lim _{n \rightarrow \infty} \mathbb P [S_{n,p} \text{ is } (A,r)\text{-Rado}]=\begin{cases}
0 &\text{ if } p \leq cn^{-1/m_{S_n}(A)}; \\
 1 &\text{ if } p \geq Cn^{-1/m_{S_n}(A)}.\end{cases}$$ \qed
\end{thm}

\subsection{Resilience in the primes}\label{sec:res}
The following is an immediate consequence of Theorem~\ref{gt}. 
\begin{lemma}\label{lem:gtsupersat}
Let $r \geq 1$ and $k \geq 3$.
Given any $\gamma>0$ there exist $n_0,\eps>0$ such that for all $n \geq n_0$, the following holds.
Suppose $X \subseteq \mathbf{P}_n$ with $|X| \geq \gamma |\mathbf{P}_n|$. Then  for any $r$-colouring of $X$
there are at least $\eps n^2/\log^kn$ monochromatic $k$-APs in $X$.\qed
\end{lemma}
Using Lemma~\ref{lem:gtsupersat} one can obtain the following resilience-type result.
\begin{thm}\label{bls-strengthening}
Let $r \geq 1$, $k \geq 3$ and $\delta>0$. There exists a constant $C>0$ such that if $p \geq Cn^{-1/(k-1)}\log n$, then w.h.p. $\mathbf{P}_{n,p}$ has the property that for any subset $X \subseteq \mathbf{P}_{n,p}$ with $|X| \geq \delta |\mathbf{P}_{n,p}|$, whenever $X$ is $r$-coloured, there is a monochromatic $k$-AP.\qed
\end{thm}
By setting $r=1$ in Theorem~\ref{bls-strengthening}, we obtain a sharpening of Theorem~\ref{thm:bls}, since the property considered is that of being $(\delta,k)$-Szemer\'edi.
For $r \geq 2$, Theorem~\ref{bls-strengthening} strengthens the $1$-statement of Theorem~\ref{rvdwp} as it states that w.h.p.\ not only can we find a monochromatic $k$-AP in $\mathbf{P}_{n,p}$ whenever $r$-coloured, but we can afford to delete a $(1-\delta)$-proportion of $\mathbf{P}_{n,p}$ and  still guarantee such a monochromatic $k$-AP.

The proof of Theorem~\ref{bls-strengthening}  is an easy adaption of the proof of the $1$-statement of Theorem~\ref{thm:mainramsey}, using ideas from the proof of a  resilience-type result in~\cite{hst}.

\subsection{Further directions and examples}\label{sec:ring}
Another direction in which one can generalise the random Rado lemma is by removing the requirement of matrices having integer entries.
Given a matrix $A$ with entries in a ring $R$ one can ask whether a finite colouring of an $R$-module yields a monochromatic solution to $Ax=0$ or not.
If one can in fact prove a suitable supersaturation result, then our machinery could yield a random Ramsey-type theorem for this setting. 
Theorem~\ref{thm:fields}
already provides one such result in this direction.

\smallskip

Condition (P5) in the statement of Theorem~\ref{thm:mainramsey} arises as an artifact of our proof. It would be interesting to investigate possible  relaxations of this condition; this could in turn lead to a relaxation of condition (A5) from 
Lemma~\ref{lem:randomRado}. 

Condition (C4) in Theorem~\ref{lem:easyRado}
states that  $(A,S_n)$ is abundant for every sufficiently large $n \in \mathbb N$. This property is used to verify, e.g., (A5) when deducing Lemma~\ref{lem:easyRado} from 
Lemma~\ref{lem:randomRado}.
However, (as briefly touched upon when discussing Theorem~\ref{thm:Zn}) just because a sequence $(S_n)_{n \in \mathbb N}$ of finite abelian groups does not satisfy (C4) does not necessarily mean one cannot apply Lemma~\ref{lem:randomRado} instead.

The following is another example of this.
Let $A:=\begin{pmatrix} 2 & -2 & 63 & 65 \end{pmatrix}$ and $S_n:=\mathbb{Z}_{128}^n$ for every $n\in\mathbb N$.
The pair $(A,S_n)$ is not abundant since $\rank_{S_n}(A)=1$, whereas $\rank_{S_n}(A_{\{1,2\}})=6/7$, so we cannot apply Lemma~\ref{lem:easyRado}.
However, conditions (A1)--(A6) do hold for $(S_n)_{n\in\mathbb N}$ and $A$, and so we can apply Lemma~\ref{lem:randomRado} to deduce the following theorem.

\begin{thm}\label{newthm}
Let $A=\begin{pmatrix} 2 & -2 & 63 & 65 \end{pmatrix}$ and $r\ge2$. 
There exist constants $c,C>0$ such that 
$$\lim _{n \rightarrow \infty} \mathbb P [(\mathbb Z_{128})^n_p \text{ is } (A,r)\text{-Rado}]=\begin{cases}
0 &\text{ if } p \leq c\,128^{-2n/3}; \\
 1 &\text{ if } p \geq C\,128^{-2n/3}.\end{cases}$$
 \qed
\end{thm}

Conversely, consider $B:=\begin{pmatrix} 4 & -4 & 63 & 65 \end{pmatrix}$. 
The sequence $(S_n)_{n\in\mathbb N}$ is $(B,r)$-supersaturated for every $r\in\mathbb N$: 
this follows from Theorem~\ref{thm:exponentsupersat}, as $\mathbb Z_{128}$ is an abelian group with exponent $128$ and $B$ satisfies the $128$-columns condition as $4-4+63+65\equiv 0 \  (\text{mod} \ 128)$.  
However, we cannot apply Lemma~\ref{lem:randomRado}, as condition (A5) does not hold. Indeed,
it is easy to check that, for $W\subseteq[4]$, 
\begin{equation*}
|\Sol_{S_n}^B(W)|=\begin{cases}
|S_n|^3          &\text{ if $|W|=[4]$;} \\
|S_n|^{|W|-2/7}      &\text{ if $\{3,4\}\subseteq W\neq[4]$;} \\
|S_n|^{|W|}           &\text{ otherwise.}
\end{cases}
\end{equation*}
It immediately follows that, for $|W|\ge2$,
\begin{equation*}
p_W(B,S_n)=\left(\frac{|\Sol_{S_n}^B(W)|}{|S_n|}\right)^{-\frac{1}{|W|-1}}=\begin{cases}
|S_n|^{-2/3}          &\text{ if $|W|=[4]$;} \\
|S_n|^{-(|W|-9/7)/(|W|-1)}      &\text{ if $\{3,4\}\subseteq W\neq[4]$;} \\
|S_n|^{-1}           &\text{ otherwise.}
\end{cases}
\end{equation*}
Note that if $(X_n)_{n \in \mathbb N}$ is a sequence such that 
$X_n \subseteq [4]$ and
$p_{X_n}(B,S_n)=\Omega(\hat p(B,S_n))$, then $X_n=[4]$ for all sufficiently large $n \in \mathbb N$. In this case,
for $W_n:=\{3,4\}\subset [4]$ we have
$$\frac{|S_n|^2}{|\Sol_{S_n}^B(W_n)|}\left(\frac{p_{W_n}(B,S_n)}{p_{[4]}(B,S_n)}\right)=|S_n|^{5/21}=(128)^{5n/21}\not\to 0$$
as $n\to\infty$. So  $(S_n)_{n\in\mathbb N}$ is not compatible with respect to $B$.

\smallskip

{\noindent \bf Open access statement.}
	This research was funded in part by  EPSRC grant EP/V048287/1. For the purpose of open access, a CC BY public copyright licence is applied to any Author Accepted Manuscript arising from this submission.

\smallskip

{\noindent \bf Data availability statement.}
The proofs omitted in this paper can be found in the PhD thesis of the first author~\cite{phdandrea}.

\section*{Appendix A}
\renewcommand{\thesubsection}{A.\arabic{subsection}}
\setcounter{subsection}{0}

\subsection{Proof of Lemma~\ref{lem:taosieve}}\label{appendix:sieve} 
Recall Lemma~\ref{lem:taosieve} is the sieve result we used in the proof of Lemma~\ref{lem:qAPbound} in Section~\ref{sec:aux}. Lemma~\ref{lem:taosieve} is a corollary of the following result from~\cite{taoblog}. 

\begin{lemma}[\protect{\cite[Theorem 32]{taoblog}}]\label{lem:originaltaosieve}
Let $t$ and $C$ be fixed natural numbers and let $n\in\mathbb N$ with $n\ge2$. For each prime number $p\leq\sqrt{n}$, let $E_p$ be the union of $\omega(p)$ residue classes modulo $p$, where $\omega(p) = t$ for all $p\ge C$, and $\omega(p)<p$ for all $p$. Then for any $\varepsilon > 0$, one has
$$|[n]\setminus\bigcup_{\substack{\text{$p$ prime}\\ p\le\sqrt{n}}} E_p| \leq (2^t t! + \varepsilon) {\mathfrak S} \frac{n}{\log^t n}$$
whenever $n$ is sufficiently large depending on $k,C,\varepsilon$, and where $\mathfrak S$ is the singular series
$${\mathfrak S} := \prod_{\text{prime $p$}}\left(1-\frac{1}{p}\right)^{-t} \left(1 - \frac{\omega(p)}{p}\right).$$
\end{lemma}

We now explain how to deduce Lemma~\ref{lem:taosieve} from Lemma~\ref{lem:originaltaosieve}.

\begin{proof}[Proof of Lemma~\ref{lem:taosieve}] 
Let $k\in\mathbb N$ with $k\ge2$ and let $q$ be a prime number. For each prime $p\ge k$ with $p\not=q$, let $E_p$ be the union of $k-1$ residue classes modulo $p$. We want to prove that $|E(n)|=O(n/\log^{k-1}n)$ where
$$E(n):=[n]\setminus\bigcup_{\substack{\text{$p\not=q$ prime}\\ k\le p\le\sqrt{n}}} E_p.$$

Fix $C:=k$, $t:=k-1$, $\varepsilon:=1$ and pick $n\in\mathbb N$ sufficiently large so that Lemma~\ref{lem:originaltaosieve} holds. For every prime $p<k$, let $\omega(p):=0$ and $E_p:=\emptyset$. We have
$$E(n)=[n]\setminus\bigcup_{\substack{\text{$p\not=q$ prime}\\ p\le\sqrt{n}}} E_p.$$
Note that if $q<k$ then $E_q=\emptyset$ by definition. In this case set $E'(n):=E(n)\setminus E_q=E(n)$.

If $q\ge k$, let $E_q$ be the union of $k-1$ residue classes modulo $q$ that maximises $|E'(n)|$ where 
$$E'(n):=E(n)\setminus E_q=[n]\setminus\bigcup_{\substack{\text{$p$ prime}\\ p\le\sqrt{n}}} E_p.$$
Since the residue classes modulo $q$ induce a partition of $E(n)$, by the maximality of $E_q$ we obtain
\begin{equation*}
|E'(n)|=|E(n)\setminus E_q|\ge\left(1-\frac{k-1}{q}\right)|E(n)|\ge\frac{1}{k}|E(n)|.
\end{equation*}

\smallskip

Applying Lemma~\ref{lem:originaltaosieve} with respect to parameters $t,C,\varepsilon$ and the $E_p$'s, we obtain $|E'(n)|\le Tn/\log^{k-1}n$ where $T$ is a constant depending uniquely on $k$. Thus, $|E(n)|\le kTn/\log^{k-1}n$ and so indeed $|E(n)|=O(n/\log^{k-1}n)$, as required.
\end{proof}

\subsection{Proof of Theorem~\ref{gt1}}\label{appendix:2ndmoment}
The proof of Theorem~\ref{gt1} is an easy application of the second moment method and it is nearly identical to the  analogous proof for arithmetic progressions in the integers (see~\cite[Example 3.2]{jlr}). We include it here for completeness.

\begin{proof}[Proof of Theorem~\ref{gt1}]
Let $X$ be the number of $k$-APs in $\mathbf P_{n,p}$. By Theorem~\ref{gt} there are $\Theta(n^2/\log^k n)$ $k$-APs in $\mathbf P_n$ and so we have $\mathbb{E}(X)=\Theta(p^kn^2/\log^kn)$. 
If $p =o( n^{-2/k}\log n)$ then $\mathbb{E}(X)=o(1)$. By Markov's inequality, $\mathbb{P}(X\ge1)\le\mathbb{E}(X)=o(1)$ and in particular $\mathbb P(X=0)=1-o(1)$, as required.

For the $1$-statement, pick an arbitrary ordering of the $k$-APs in $\mathbf P_n$ and let $X_i$ be the indicator variable for the $i$th $k$-AP to appear in $\mathbf{P}_{n,p}$. Formally,
$$X_i=\begin{cases}
1 &\text{ if the $i$th $k$-AP lies in $\mathbf P_{n,p}$ and}\\
0 &\text{ otherwise}.
\end{cases}$$
Now we bound the number of pairs of $k$-APs in $\mathbf P_n$ that share exactly one element: first fix a prime $q\in\mathbf P_n$ ($O(n/\log n)$ choices) and then pick two $k$-APs in $\mathbf P_n$ that contain $q$ ($O(n^2/\log^{2k-2}n)$ choices by Lemma~\ref{lem:qAPbound}). Thus, there are $O(n^3/\log^{2k-1}n)$ such pairs. Furthermore, there are $O(n^2/\log^k n)$ pairs of $k$-APs in $\mathbf P_n$ that share at least two elements (since a $k$-AP is uniquely determined by two elements and their positions within the $k$-AP). 

Note that $X=\sum\limits_i X_i$. It follows that
\begin{align*}
\Var(X)=\sum_{i,j} \Cov(X_i,X_j)\le
\sum_{i,j}
\mathbb E(X_iX_j)=O\left(\frac{n^3}{\log^{2k-1}n}p^{2k-1}+\sum_{t=2}^{k}\frac{n^2}{\log^{k}n}p^{2k-t}\right),
\end{align*}
where here the sums are only for $i,j$ (including $i=j$) such that the
$i$th and $j$th $k$-APs in $\mathbf P_n$ intersect.
Chebyshev's inequality yields 
\begin{align*}
\mathbb P(X=0) \leq \frac{\Var(X)}{\mathbb{E}(X)^2}=O\left(\frac{\log n}{np}+\sum_{t=2}^{k}\frac{\log^kn}{n^2p^t}\right).
\end{align*}
If  $p =\omega( n^{-2/k}\log n )$ the RHS of the equality above is $o(1)$ and so $\mathbb P(X=0)=o(1)$, as required.
\end{proof}

\subsection{Deducing the $1$-statement of the random Ramsey theorem from Theorem~\ref{thm:mainramsey}}\label{seca2}

In this section we show that Theorem~\ref{thm:mainramsey} implies the $1$-statement of Theorem~\ref{randomramsey}.

Let~$H$ be as in the statement of Theorem~\ref{randomramsey}. Write $E(H)=\{h_1,\dots,h_{e(H)}\}$ and let $r\geq 2$. Without loss of generality we may assume that $H$ has no isolated vertices. 
For $n\in\mathbb N$, let~$\mathcal H_n$ be the $e(H)$-uniform  ordered hypergraph with vertex set $V(\mathcal H_n):=E(K_n)$ where an ordered $e(H)$-tuple $(e_1,\dots,e_{e(H)})\in (E(K_n))^{e(H)}$ is an edge of $\mathcal H_n$ if and only if there is an injective graph homomorphism $g:H\to K_n$ such that $g(h_i)=e_i$ for every $i\in[e(H)]$;
in other words, there is a copy of $H$ in $K_n$ with edges $e_1,\dots,e_{e(H)}$ where $e_i$ plays the role of $h_i$ for every $i\in[e(H)]$. 

For brevity, we write $\mathcal H:=\mathcal H_n$.
For $W\subseteq[e(H)]$, let $H_W$ be the subgraph of $H$ spanned by the set of edges $\{h_i:i\in W\}\subseteq E(H)$; so $H_W$ contains no isolated vertices. Recall that the restriction $\mathcal H_W$ of $\mathcal H$ was defined just after Remark~\ref{remark:restriction}.

\begin{claim}\label{claim:eHw}
We have $e(\mathcal H_W)=\Theta(n^{v(H_W)})$ for every $W\subseteq[e(H)]$.    
\end{claim}
\begin{proofclaim}
Say $W=\{{t_1},\dots,{t_{|W|}}\}$ where $t_1<\dots<t_{|W|}$. Then, provided $n\geq v(H)$, an ordered $|W|$-tuple $(e_{1},\dots,e_{{|W|}})\in (E(K_n))^{|W|}$ is an edge of $\mathcal H_W$ if and only if there is an injective graph homomorphism $g:H_W\to K_n$ such that $g(h_{t_i})=e_i$ for every $i\in[|W|]$. 
There are $\Theta\left(n^{v(H_W)}\right)$  such homomorphisms. 
Conversely, at most $2^{|W|}$ injective homomorphisms $g:H_W\to K_n$ correspond to the same edge in $\mathcal H_W$.\footnote{More precisely, there are exactly $2^a$ injective homomorphisms $g:H_W\to K_n$ that corresponds to the same edge in $\mathcal H_W$, where $a$ is the number of isolated edges in $H$ which lie in $H_W$.} 
Hence $e(\mathcal H_W)=\Theta\left(n^{v(H_W)}\right)$, as claimed. \qedclaim
\end{proofclaim}

We can now compute $\hat p(\mathcal H)$. We have $v(\mathcal H)=e(K_n)={n \choose 2}$. This and Claim~\ref{claim:eHw} yield 
\begin{equation}\label{eq:fWm2}
f_{n,W}=\left(\frac{e(\mathcal H_W)}{v(\mathcal H)}\right)^{-\frac{1}{|W|-1}}=\Theta\left(\left(\frac{n^{v(H_W)}}{n^2}\right)^{-\frac{1}{e(H_W)-1}}\right)=\Theta\left(n^{-\frac{v(H_W)-2}{e(H_W)-1}}\right).
\end{equation}
Note that $\frac{e(H_W)-1}{v(H_W)-2}=d_2(H_W)$ when $|W|\ge 2$. Thus, we have
\begin{equation}\label{eq:m2-a}
\max_{\stackrel{W \subseteq [k]}{|W| \geq 2}}\frac{e(H_W)-1}{v(H_W)-2}\le\max_{H'\subseteq H}d_2(H')=m_2(H).
\end{equation}

By assumption, $H$ contains a connected component $K$ with at least two edges. We have $d_2(K)=\frac{e(K)-1}{v(K)-2}\ge1$ since $K$ is connected. 
Suppose the value of $m_2(H)$ is achieved by $H'\subseteq H$. 
Then $H'$ contains at least two edges, since $d_2(H')\ge d_2(K)\ge1$. 
Furthermore, $H'$ has no isolated vertices. Indeed, deleting an isolated vertex from $H'$ would yield a subgraph $H''$ of $H$ with $d_2(H'')>d_2(H')$, a contradiction. 
Set $W:=\{i:h_i\in E(H')\}$. 
Since $H'$ has no isolated vertices, it follows that $H_W=H'$. 
Furthermore, $|W|\ge2$ since $H'$ has at least two edges. We have $d_2(H_W)=d_2(H')=m_2(H)$, and so equality holds in~\eqref{eq:m2-a}. That is,
\begin{equation*}\label{eq:m2-b}
\max_{\stackrel{W \subseteq [k]}{|W| \geq 2}} \frac{e(H_W)-1}{v(H_W)-2}=m_2(H).
\end{equation*}
Using the equality above, taking the maximum over all $W\subseteq[e(H)]$ of size at least $2$ in~\eqref{eq:fWm2} yields 
$$\hat p(\mathcal H)=\Theta\left(n^{-\frac{1}{m_2(H)}}\right).$$

Next, we verify conditions (P1)--(P4).

\begin{itemize}
\item[(P1):] 
Recall we proved that $m_2(H)\ge1$. 
This inequality implies $n^{-\frac{1}{m_2(H)}}\to0$ and $n^{2-\frac{1}{m_2(H)}}\to\infty$ as $n\to\infty$. 
In particular, $\hat p(\mathcal H)\to0$ and $\hat p(\mathcal H)v(\mathcal H)\to\infty$ as $n\to\infty$. 

\item[(P2):] 
By  Ramsey's theorem, there exists an $m\in\mathbb N$ such that any $r$-edge-colouring of $K_m$ yields a monochromatic copy of $H$. For $n \geq m$, 
pick an $r$-edge-colouring of $K_n$. 
Every subset of $V(K_n)$ of size $m$ contains a monochromatic copy of $H$, and each copy of $H$ belongs to ${n-v(H)}\choose{m-v(H)}$ subsets of $V(K_n)$ of size $m$. 
Thus, the number of monochromatic copies of $H$ in $K_n$ is at least
$$\frac{{n\choose m}}{{{n-v(H)}\choose{m-v(H)}}}\ge\left(\frac{n}{m}\right)^{v(H)}.$$

Equivalently, given any $r$-colouring of the vertices of $\mathcal H$, there are at least $(n/m)^{v(H)}$ monochromatic edges in $\mathcal H$. 
By Claim~\ref{claim:eHw} we have $e(\mathcal H)=e(\mathcal H_{[e(H)]})=\Theta(n^{v( H)})$.
Together, this implies that 
 the sequence of ordered hypergraphs $(\mathcal H_n)_{n\in\mathbb N}$ is $r$-supersaturated.

\item[(P3):] 
Let $W\subseteq Y\subseteq[e(H)]$. 
Given $(e_1,\dots,e_{|W|})\in \mathcal H_W$, there are at most $n^{v(H_Y)-v(H_W)}$ edges in $\mathcal H_Y$ that restrict to $(e_1,\dots,e_{|W|})$. 
This is because each such edge corresponds to an injective graph homomorphism $g:H_Y\to K_n$ with $g(h_i)=e_i$ for every $i\in W$. 
For all such homomorphisms, the images of the edges in $E(H_W)$ are fixed, and we have at most $n^{v(H_Y)-v(H_W)}$ choices for the image of $V(H_Y)\setminus V(H_W)$. 
Thus,
$$\Delta_W(\mathcal H_Y)\le n^{v(H_Y)-v(H_W)}.$$
By Claim~\ref{claim:eHw}, we have $e(\mathcal H_W)=\Theta(n^{v(H_W)})$ and $e(\mathcal H_Y)=\Theta(n^{v(H_Y)})$. 
Combining this with the previous inequality yields
$$\Delta_W(\mathcal H_Y)\le \Theta\left(\frac{e(\mathcal H_Y)}{e(\mathcal H_W)}\right),$$
and so condition (P3) holds.

\item[(P4):] 
Let $W\subset Y\subseteq[e(H)]$ with $|W|=1$. 
Note that $E(\mathcal H_W)=V(\mathcal H)$. 
This combined with condition (P3) gives
$$\Delta_W(\mathcal H_Y)\le\Theta\left(\frac{e(\mathcal H_Y)}{e(\mathcal H_W)}\right)=\Theta\left(\frac{e(\mathcal H_Y)}{v(\mathcal H_W)}\right),$$
and so condition (P4) holds.
\end{itemize}

\smallskip

Since conditions (P1)--(P4) hold, by Theorem~\ref{thm:mainramsey} we obtain the following $1$-statement. 
There exists some $C>0$ such that 
$$\lim_{n \rightarrow \infty} \mathbb{P}\left[\text{$\mathcal H_{n,q}^v$ is $r$-Ramsey}\right]=1 \quad\text{ if } q \geq  Cn^{-\frac{1}{m_2(H)}}.$$
This is equivalent to
$$\lim_{n \rightarrow \infty} \mathbb{P}\left[\text{$G_{n,p}$ is $(H,r)$-Ramsey}\right]=1 \quad\text{ if } p \geq  Cn^{-\frac{1}{m_2(H)}},$$
which is precisely the $1$-statement of Theorem~\ref{randomramsey}.

\end{document}